\documentclass[11pt]{article}

\usepackage[english]{babel}

\usepackage[T1]{fontenc}
\usepackage{lmodern}

\usepackage[utf8]{inputenc}
\usepackage[a4paper,margin=2.5cm]{geometry}
\usepackage[dvipsnames]{xcolor}
\usepackage{eufrak}
\usepackage{chemist}
\usepackage{amsmath,amsfonts,amsthm,amssymb}

\DeclareMathOperator{\supp}{supp}
\DeclareMathOperator*{\esssup}{ess\,sup}

\usepackage{blindtext}
\usepackage{mathtools}
\usepackage{mathrsfs}
\newcommand{\defeq}{\mathrel{\mathop:}=}
\newcommand*\Laplace{\mathop{}\!\mathbin\bigtriangleup}

\usepackage{booktabs,tabularx}
\usepackage{apxproof}
\usepackage{environ}
\usepackage{dsfont}
\usepackage{etoolbox}
\usepackage{fancyvrb}
\usepackage{ifthen}
\usepackage{kvoptions}
\usepackage{bibunits}
\usepackage{comment}
\usepackage{chemformula}
\usepackage{mathrsfs}
\usepackage{float}
\usepackage{xcolor}
\usepackage{enumerate}
\usepackage[toc,page]{appendix}
\usepackage{graphicx,color}
\usepackage{hyperref}
\usepackage{mathtools}
\usepackage{setspace}
\usepackage[toc,page]{appendix}
\usepackage{cite}
\usepackage[normalem]{ulem} 

\makeatletter
\renewcommand*\env@matrix[1][\arraystretch]{%
  \edef\arraystretch{#1}%
  \hskip -\arraycolsep
  \let\@ifnextchar\new@ifnextchar
  \array{*\c@MaxMatrixCols c}}
\makeatother

\let\OLDthebibliography\thebibliography
\renewcommand\thebibliography[1]{
  \OLDthebibliography{#1}
  \setlength{\parskip}{0pt}
  \setlength{\itemsep}{0pt plus 0.3ex}
}

\newtheorem{theorem}{Theorem}[section]
\newtheorem{lemma}[theorem]{Lemma}
\newtheorem{corollary}[theorem]{Corollary}
\newtheorem{assumption}{Assumption}
\newtheorem{proposition}[theorem]{Proposition}

\numberwithin{equation}{section}

\newtheorem*{theorem*}{Theorem}
\newtheorem*{proposition*}{Proposition}
\newtheorem*{corollary*}{Corollary}
\newtheorem*{lemma*}{Lemma}

\theoremstyle{definition}
\newtheorem{definition}[theorem]{Definition}
\newtheorem{remark}[theorem]{Remark}

\newcommand{\ignore}[1]{}

\newcommand{\llbracket}{[\![}
\newcommand{\rrbracket}{]\!]}

\usepackage{amsmath}
\usepackage{graphicx}

\begin{document}

\title{Can deleterious mutations surf deterministic population waves?}

\author{João Luiz de Oliveira Madeira\thanks{Department of Statistics, University of Oxford, UK} \and Marcel Ortgiese\footnote{Department of Mathematical Sciences, University of Bath, UK} \and Sarah Penington\footnotemark[2]}


\maketitle

\begin {abstract}
{In spatially structured populations, rare neutral mutations can spread through large regions during a range expansion, a phenomenon known as gene surfing. Whether deleterious mutations can also surf remains poorly understood. To address this question, we study a deterministic version of the spatial Muller's ratchet, given by an infinite system of reaction--diffusion equations describing an asexual population subject to mutation, migration, and density-dependent reproduction and death.
After establishing that the system of PDEs is well-posed, we analyse the distribution of deleterious mutations within the population. In the monostable regime, we derive quantitative bounds on the ratio between the density of individuals carrying a given number of mutations and the density of mutation-free individuals. Under a Fisher--KPP condition, we further determine the spreading speed of the population into an empty habitat, confirming non-rigorous computations of Foutel-Rodier and Etheridge. Finally, using a tracer dynamics approach, we show that deleterious mutations cannot surf deterministic waves: although they are present at the expansion front, they only arise as recent descendants of the wild type.}
\end{abstract}


\section{Introduction}
\label{Paper02_introduction}

Mutations are a major driver of evolution.
Depending on their effect on an individual's ability to survive and reproduce, mutations may be classified as deleterious, advantageous, or neutral. Deleterious mutations reduce fitness, advantageous mutations increase fitness, and neutral mutations have no direct effect on fitness.

Under natural selection, individuals with higher fitness are expected to leave more descendants, and deleterious mutations are therefore generally less likely to reach fixation. In spatially structured populations, however, reproductive success depends not only on an individual's fitness but also on the local population density~\cite{foutel2020spatial}. High population density may reduce the per-capita growth rate through competition for resources, but it may also enhance growth through cooperative interactions between individuals.

Edmonds et al.~\cite{edmonds2004mutations} observed that, during a range expansion, rare neutral mutations can propagate across large regions, a phenomenon called \emph{gene surfing}. Roughly speaking, when competition dominates over cooperation, the per-capita growth rate is highest at low population densities. In this case, the population dynamics is said to be of \emph{Fisher-KPP} type. In this scenario, a neutral mutation appearing at the leading edge of the expanding population can spread locally with high probability, since the population density at the front is usually very low. Since individuals near the front colonise the new habitat, such mutations can become prevalent over large spatial regions~\cite{foutel2020spatial}. Importantly, this expansion is not driven by the mutation’s effect on fitness, but purely by its appearance at the front of an expanding population~\cite{hallatschek2008gene, birzu2018fluctuations, edmonds2004mutations}. When cooperation prevails over competition, the maximum per-capita growth rate is achieved at an intermediate rather than very low population density. The population dynamics is then said to exhibit an \emph{Allee effect}. In this case, the maximum growth rate is achieved behind the expansion front, which reduces the opportunity for gene surfing~\cite{hallatschek2008gene}.

Gene surfing of neutral mutations has been studied through simulations of discrete particle systems, partial differential equations (PDEs), and stochastic PDEs (see~\cite{edmonds2004mutations,hallatschek2008gene, roques2012allee}), and rigorously in PDE-based models~\cite{garnier2012inside, roques2012allee}. We also refer the reader to~\cite{roques2012allee} for a discussion on the possible different definitions of gene surfing of neutral mutations.

A natural question is then whether or not gene surfing of deleterious mutations could also occur~\cite{foutel2020spatial}. If an allele is sufficiently deleterious relative to the population average, then one would expect surfing of this allele not to be possible. However, for weakly deleterious mutations, experimental studies suggest several possible outcomes. Deleterious mutations may fail to surf altogether~\cite{do2015no}, may surf only transiently before being eliminated by selection~\cite{simons2014deleterious}, or may spread across large spatial regions~\cite{peischl2013accumulation,travis2007deleterious,peischl2015expansion,peischl2016genetic,rougemont2023allele}.

Assuming that deleterious mutations can surf population waves, and since deleterious mutations occur at much higher frequency than beneficial mutations, repeated surfing events of deleterious mutations would be expected during a range expansion, leading to a progressive decline in fitness at the expansion front. This phenomenon is called \emph{expansion load}~\cite{peischl2015expansion}. Although simulation studies strongly support the existence of expansion load~\cite{peischl2013accumulation,peischl2015expansion,peischl2016genetic,gilbert2017local, zeitler2023purging}, there is biological evidence for both existence~\cite{zeitler2023purging,rougemont2023allele,peischl2016genetic} and non-existence~\cite{do2015no,simons2014deleterious,takou2021maintenance} of expansion load. In analogy to gene surfing of neutral mutations, the presence of an Allee effect is also expected to reduce the likelihood that deleterious mutations surf.

To address the question of whether or not deleterious mutations can surf, Foutel-Rodier and Etheridge introduced the \emph{spatial Muller's ratchet}~\cite{foutel2020spatial}, a spatial extension of a mechanism first proposed by Muller~\cite{muller1964relation} to explain the evolution of recombination and sexual reproduction. In asexual reproduction, chromosomes are inherited as indivisible units, so the number of mutations along an ancestral lineage can only increase. Since the majority of mutations in biological systems are deleterious~\cite{bao2022mutations}, this leads to a gradual decline in overall population fitness~\cite{etheridge2009often}. In Muller's ratchet models, when all individuals in the most-adapted class in the population acquire at least one additional deleterious mutation, the minimum mutational load in the population increases — this is known as a ‘click' of the ratchet. Muller's ratchet has been used to explain the accumulation of deleterious mutations in microorganisms, in mitochondrial DNA, and in the Y chromosome in mammals (see e.g.~\cite{howe2008muller, andersson1996muller, kaiser2010muller} and references therein). From a mathematical perspective, Muller's ratchet has been extensively studied in populations without spatial structure~\cite{casanova2022quasi,pfaffelhuber2012muller,haigh1978accumulation,etheridge2009often,mariani2020metastability}.

The model introduced in~\cite{foutel2020spatial} is an interacting particle system where individuals migrate, and reproduce and die with rates that depend on the local population density. To study expansion load and replicate a similar mechanism to the classical Muller's ratchet, in the model in~\cite{foutel2020spatial}, after each reproduction event, the offspring can accumulate an extra deleterious mutation with respect to their parent. Foutel-Rodier and Etheridge conjectured that, under an appropriate scaling regime, the particle system should converge to the solution of a system of PDEs that we now introduce. For each $k \in \mathbb{N}_0$, for $t > 0$, let $u_k(t, \cdot)$ denote the population density corresponding to individuals carrying exactly~$k$ mutations. Since individuals may carry arbitrarily many mutations, the limiting system involves infinitely many coupled equations. For all $t > 0$ and $x \in \mathbb R$, the total population density of individuals living at time $t$ and position $x$ is $\|u(t,x) \|_{\ell_1} = \sum_{k = 0}^{\infty} u_k(t,x)$. Foutel-Rodier and Etheridge proposed that the behaviour of the population is captured in the deterministic regime by a sequence of population densities $u = (u_k)_{k\in \mathbb N_0}$, which is the solution of the following system of PDEs: for $k\in \mathbb N_0$ and $t>0$,
\begin{equation} \label{Paper02_conjectured_PDE_heuristics}
\begin{aligned}
    \partial_t u_k = \frac{m}{2} \Laplace u_k + r(B\| u \|_{\ell_1} + 1)(u_k(1-s)^{k}(1- \mu) + \mathds{1}_{\{k \geq 1\}}u_{k-1}(1-s)^{k-1}\mu - u_{k} \| u \|_{\ell_1}), 
\end{aligned}
\end{equation}
where~$m\in (0,\infty)$ is a migration rate, $r \in (0,\infty)$ is a Malthusian growth parameter, $s \in (0,1)$ represents the impact on fitness of each deleterious mutation, and the parameter $B \in [0, \infty)$ captures the relative effect of cooperation and competition in the local population dynamics. Here larger values of $B$ indicate a greater role of cooperation in the dynamics. The parameter $\mu \in (0,1)$ reflects the probability that, after a birth event, an offspring individual accumulates an additional mutation with respect to their parent. The mutations are assumed to affect only the per-capita reproduction rate, and the migration rate is the same across all individuals. Foutel-Rodier and Etheridge also predict~\cite[Equation~(9)]{foutel2020spatial} that the (critical) spreading speed into an empty habitat of a population governed by the dynamics of~\eqref{Paper02_conjectured_PDE_heuristics} depends on the parameter $B \in [0, \infty)$, and is given by $c^*>0$, where:
\begin{equation} \label{Paper02_conjecture_spreading_speed_foutel_rodier_etheridge}
    c^* \defeq \left\{\begin{array}{lcl}
        \sqrt{2mr(1-\mu)} & \textrm{if} & \displaystyle B \leq \frac{2}{1-\mu},  \\
          \displaystyle \frac{B(1-\mu) + 2}{2}\sqrt{\frac{mr}{B}} & \textrm{if} & \displaystyle B > \frac{2}{1-\mu}.
    \end{array}\right.
\end{equation}
The change in the expression for the spreading speed at $B = \frac{2}{1-\mu}$ reflects a transition from \emph{pulled} to \emph{pushed} expansion waves~\cite{roques2012allee,garnier2012inside}. Since pulled waves occur under Fisher--KPP dynamics, and pushed waves can occur with the presence of an Allee effect, the formula~\eqref{Paper02_conjecture_spreading_speed_foutel_rodier_etheridge} predicts a qualitative change in the mechanism driving the range expansion as the parameter~$B$ increases.

For an expanding population, in the PDE limit scaling, Foutel-Rodier and Etheridge also state that, under a weak selection--low mutation regime, i.e.~assuming that $s,\mu \ll 1$, the stationary solution $\hat{u} = (\hat{u}_k)_{k \in \mathbb{N}_{0}}$ of the system of PDEs~\eqref{Paper02_conjectured_PDE_heuristics} with $\hat{u}_{0} \neq 0$  satisfies
\begin{equation} \label{Paper02_approximating_equilibrium_weak_selection_low_mutation_regime}
    \frac{\hat{u}_{k}}{\| \hat{u} \|_{\ell_1}} \approx \exp\left(-\frac{\mu}{s}\right) \frac{(\mu/s)^{k}}{k!} \quad \forall \, k \in \mathbb{N}_0.
\end{equation}
Foutel-Rodier and Etheridge then conclude that although deleterious mutations are observed at stationarity, this results from mutation-selection equilibrium, rather than proper gene surfing.

The main goal of the current article is to rigorously establish whether or not deleterious mutations can surf deterministic range expansions. From a mathematical perspective, there are several challenges in rigorously answering this question and in proving the conjectures of Foutel-Rodier and Etheridge in~\cite{foutel2020spatial}. The existence and well-posedness of the interacting particle system proposed in~\cite{foutel2020spatial} were established in~\cite{madeira2025existence}. The rigorous derivation of a generalised version of the PDE system~\eqref{Paper02_conjectured_PDE_heuristics} as the scaling limit of this interacting particle system is carried out in our companion work~\cite{madeira2025LLN}. In this article, we study a general version of the system of PDEs~\eqref{Paper02_conjectured_PDE_heuristics} and first address well-posedness of solutions. Existence of mild solutions follows directly from the law of large numbers in our companion article~\cite{madeira2025LLN}. In the present article, we then prove uniqueness of solutions under very mild assumptions.

The main part of this article then deals with the long-term analysis of the system of PDEs.
A major challenge here is
the fact that the interaction is non-local in type space, i.e.~fractions of the population carrying different numbers of mutations affect the reproduction and death rates of each other. As a consequence, standard comparison-principle arguments are not directly applicable. In this article, 
we instead use a Feynman--Kac representation of~\eqref{Paper02_conjectured_PDE_heuristics}. Classically, a Feynman--Kac formula is a probabilistic representation of the solution of a linear parabolic PDE in terms of the expectation of a suitable functional of a stochastic process. Here, we follow~\cite{penington2018spreading} (which in turn is based on ideas in~\cite{bramson1983convergence}) and apply the Feynman--Kac formula also in the non-linear setting.
Using this method, we derive the spreading speed in the Fisher-KPP regime (that is, for $B \leq 1/(1-\mu)$), and establish an equilibrium
mutational profile that agrees with~\eqref{Paper02_approximating_equilibrium_weak_selection_low_mutation_regime} in the weak selection--low mutation regime. Finally, by constructing suitable 
tracer dynamics, we show that gene surfing of deleterious mutations does not happen in the deterministic regime, even in the Fisher-KPP case. In particular, our results provide a rigorous justification of several conjectures proposed in~\cite{foutel2020spatial}. The results of this article also play a key role in our companion work~\cite{madeira2025LLN}, where the PDE analysis is combined with a hydrodynamic limit to derive asymptotic properties of the underlying interacting particle system.

\medskip

\noindent \textbf{Structure of the article.} In Section~\ref{Paper02_model_description}, we define the system of PDEs which corresponds to the deterministic version of the spatial Muller's ratchet, and state the main results of this article. In Section~\ref{Paper02_example_pop_dynamics}, we discuss how our results can be applied to different population dynamics. In Section~\ref{Paper02_subsection_heuristics}, we informally present the key ideas behind the proofs of our results. In Section~\ref{Paper02_PDES_sequence_spaces}, we review 
the notion of solutions to the (infinite) system of PDEs used in this article and 
state the existence of solutions that follows from our companion article~\cite{madeira2025LLN}. In Section~\ref{Paper02_uniqueness_weak_solutions_section}, we establish uniqueness and regularity properties of solutions to the system of PDEs studied throughout this article, and their asymptotic behaviour is established in Section~\ref{Paper02_section_asymptotic_behaviour_PDE}.

\medskip
\noindent\textbf{Notation.} Throughout the article we will use the following notation. Let $\ell_{1}$ denote the space of summable real-valued sequences, i.e. 
\begin{equation*}
    \ell_{1} = \ell_{1}\left(\mathbb{N}_{0}\right) \defeq \left\{z = \left(z_{k}\right)_{k \in \mathbb{N}_{0}} \in \mathbb{R}^{\mathbb{N}_{0}}: \, \| z \|_{\ell_1} \defeq \sum_{k = 0}^{\infty}\vert z_{k} \vert < \infty \right\}.
\end{equation*}
Moreover, let $\ell_{1}^{+}$ denote the subset of $\ell_{1}$ consisting of the summable sequences with non-negative entries. For any 
$K \in \mathbb{N}$, we let $\llbracket K \rrbracket \defeq \{1,2,\ldots, K\}$. Let $\lambda$ be the Lebesgue measure on $\mathbb R^d$ for any $d\in \mathbb N$.

For a metric space $(\mathcal{S}, d_{\mathcal{S}})$, we let $\mathscr{C}(\mathcal{S}; \mathbb{R})$ denote the space of continuous real-valued functions on $\mathcal{S}$. Let $\mathscr{C}^{1}(\mathbb{R})$ denote the set of continuously differentiable real-valued functions defined on $\mathbb R$. Let $\mathscr{C}^{1,2}((0,\infty) \times \mathbb{R}; \mathbb{R})$ denote the set of real-valued functions on $(0,\infty) \times \mathbb{R}$ that are continuously differentiable in the first coordinate and twice continuously differentiable in the second coordinate.

For a set $\mathcal{S}$, suppose $f,g: \mathcal{S} \rightarrow (0,\infty)$ are functions on $\mathcal{S}$. We write $f \lesssim g$ to indicate that there exists a constant $C > 0$ such that for every $x \in \mathcal{S}$, $f(x) \leq Cg(x)$. If the constant $C$ depends on a parameter $p$, then we write $f \lesssim_p g$. For $z \in \mathbb R$, we let $z^+ \defeq z \vee 0$.

\section{Model definition and main results} \label{Paper02_model_description}

We now introduce the general class of reaction–diffusion systems that will be studied throughout the article. Let $m > 0$ be the migration rate, and let $(s_{k})_{k \in \mathbb N_0}$ be a sequence of fitness parameters, where $s_k \geq 0$ denotes the fitness of an individual carrying $k$ mutations. Let $q_{+},q_{-}: [0,\infty) \rightarrow [0,\infty)$ be non-negative functions representing the (density-dependent) birth rate of individuals carrying no mutations, and the death rate of all individuals, respectively, let $\mu \in [0,1]$ be the mutation probability, and let $f: \mathbb R \rightarrow \ell_1^+$. We will precisely state our assumptions on $(s_{k})_{k \in \mathbb{N}_{0}}$, $q_{+}$, $q_{-}$ and~$f$ later in this section.

In this article, we will study the solution $u = (u_k)_{k \in \mathbb N_0}: [0,\infty) \times \mathbb{R} \rightarrow \ell_{1}$ of an infinite system of PDEs given by
\begin{equation} \label{Paper02_PDE_scaling_limit}
\begin{aligned}
    \partial_{t} u_{k} & = \frac{m}{2} \Laplace u_{k} + F_{k}(u) \quad 
    \forall k \in \mathbb N_0, \, t > 0, \\
    u(0,\cdot) & = f(\cdot),
\end{aligned}
\end{equation}
where $F = (F_k)_{k \in \mathbb{N}_0}: \ell_{1} \rightarrow \ell_{1}$ is given by, for all $v = (v_k)_{k \in \mathbb{N}_0} \in \ell_1$,
\begin{equation} \label{Paper02_reaction_term_PDE}
    F_{k}(v) = q_{+}(\| v \|_{\ell_1})\left(s_{k}(1-\mu)v_{k} + \mathds{1}_{\{k \geq 1\}}s_{k-1}\mu v_{k-1}\right) - q_{-}(\| v \|_{\ell_1})v_{k} \quad 
    \forall \, k \in \mathbb N_0.
\end{equation}
Note that, in the system of PDEs above, letting $u_k(t, \cdot)$ denote the population density corresponding to individuals carrying exactly~$k$ mutations at time $t$, the term $\displaystyle \frac{m}{2} \Laplace u_{k}$ corresponds to migration of individuals carrying exactly $k$ mutations; $q_{+}(\| u\|_{\ell_{1}})s_{k}(1 - \mu)u_{k}$ is the rate at which individuals carrying exactly $k$ mutations give birth to new individuals carrying exactly $k$ mutations; $q_{-}(\| u\|_{\ell_{1}}) u_{k}$ is the rate at which individuals carrying $k$ mutations die; and $q_{+}(\| u\|_{\ell_{1}})s_{k-1}\mu u_{k-1}$ is the rate at which individuals carrying exactly $k-1$ mutations give birth to new individuals carrying $k$ mutations (when $k \geq 1$).
See~\cite[Section 2]{madeira2025LLN} for an informal definition of the spatial Muller's ratchet particle system for which the system of PDEs~\eqref{Paper02_PDE_scaling_limit} appears as the scaling limit.

We will now state our conditions on $(s_{k})_{k \in \mathbb{N}_{0}}$, $q_{+}$, $q_{-}$ and $f$. We assume that all the mutations in our model are deleterious, i.e.~that $(s_k)_{k \in \mathbb{N}_0}$ is decreasing, so that any mutation decreases the rate of reproduction. More precisely, we assume:

\begin{assumption} [Fitness parameters]\label{Paper02_assumption_fitness_sequence}
The sequence of fitness parameters $\left(s_{k}\right)_{k \in \mathbb{N}_{0}}$ satisfies the following conditions:
\begin{enumerate}[(i)]
\item $s_{0} = 1$.
\item $s_{k} \geq 0$ for all $k \in \mathbb{N}_{0}$.
\item $\left(s_{k}\right)_{k \in \mathbb{N}_{0}}$ is monotonically non-increasing, i.e.~$s_{k} \geq s_{k+1}$ for all $k \in \mathbb{N}_{0}$.
\item $\displaystyle \lim_{k \rightarrow \infty} s_{k} = 0$.
\end{enumerate}
\end{assumption}

We make the following assumptions on the birth and death rates.

\begin{assumption} [Birth and death polynomial rates]\label{Paper02_assumption_polynomials}
Suppose that $q_{+}, q_{-}: [0,\infty) \rightarrow [0,\infty)$ are polynomials such that $0 \leq \deg q_{+} < \deg q_{-}$.
\end{assumption}

By Assumption~\ref{Paper02_assumption_polynomials}, the leading coefficient of the polynomial $q_{+}$ is non-negative, and the leading coefficient of $q_-$ is strictly positive. Note that the fact that $\deg q_{+} < \deg q_{-}$ implies local regulation of the population density. Indeed, for sufficiently large population density, death rates increase faster than birth rates, preventing uncontrolled local growth.

The initial condition of the system of PDEs~\eqref{Paper02_PDE_scaling_limit} is given by a function $f: \mathbb{R} \rightarrow \ell_{1}^{+}$ satisfying the following assumptions. Recall that we let $\lambda$ denote the Lebesgue measure on $\mathbb R$.

\begin{assumption} [Initial condition] \label{Paper02_assumption_initial_condition}
Suppose $f = \left(f_{k}\right)_{k \in \mathbb{N}_{0}}: \mathbb{R} \rightarrow \ell_{1}^{+}$ satisfies the following conditions:
\begin{enumerate}[(i)]
    \item $f$ is continuous $\lambda$-almost everywhere, i.e.~there exists $\mathcal{N}^{(1)} \subset \mathbb{R}$ such that $\lambda(\mathcal{N}^{(1)}) = 0$ and $f$ is continuous on $\mathbb{R} \setminus \mathcal{N}^{(1)}$.
    \item $f \in L_{\infty}(\mathbb{R}; \ell_{1})$, i.e.~$\esssup_{x \in \mathbb{R}} \| f(x) \|_{\ell_1} < \infty$.
    \item There exists $\mathcal{N}^{(2)} \subset \mathbb{R}$ such that $\lambda(\mathcal{N}^{(2)}) = 0$ and
    \begin{equation*}
        \lim_{k \rightarrow \infty} \; \sup_{x \in \mathbb{R} \setminus \mathcal{N}^{(2)}} \; \sum_{j \geq k} f_{j}(x) = 0.
    \end{equation*}
\end{enumerate}
\end{assumption}

Our first main result concerns existence and uniqueness of solutions to the system of PDEs~\eqref{Paper02_PDE_scaling_limit}, as well as regularity properties of solutions. 

\begin{theorem} \label{Paper02_deterministic_scaling_foutel_rodier_etheridge}
Suppose $m > 0$, $\mu \in [0,1]$, $(s_k)_{k \in \mathbb N_0}$ satisfies Assumption~\ref{Paper02_assumption_fitness_sequence}, $q_+,q_-: [0, \infty) \rightarrow [0, \infty)$ satisfy Assumption~\ref{Paper02_assumption_polynomials}, and $f \in L_{\infty}(\mathbb R; \ell_1)$ satisfies Assumption~\ref{Paper02_assumption_initial_condition}. There exists a unique non-negative global mild solution $u = (u_k)_{k \in \mathbb N_0}: [0, \infty) \times \mathbb R \rightarrow \ell_1^+$ to the system of PDEs~\eqref{Paper02_PDE_scaling_limit} such that
\begin{enumerate}[(i)]
\item For all $T > 0$,
\[\sup_{t \in [0,T]} \; \| {u}(t, \cdot) \|_{L_{\infty}(\mathbb{R}; \ell_{1})} < \infty.
\]
\item $u_k \in \mathscr{C}^{1,2}((0, \infty) \times \mathbb R; \mathbb R)$ for every $k \in \mathbb N_0$.
\item The map $(0, \infty) \times \mathbb R \ni (t,x) \mapsto \| u(t,x) \|_{\ell_1}$ is in $\mathscr C^{1,2}((0,\infty) \times \mathbb R; \mathbb R)$.
\end{enumerate}
\end{theorem}

We will carefully define $L_{p}$ spaces of $\ell_{1}$-valued functions and what it means to say that $u$ is a mild solution of~\eqref{Paper02_PDE_scaling_limit} in Section~\ref{Paper02_PDES_sequence_spaces}. We will refer to the mild solution $u = (u_k)_{k \in \mathbb N_0}: [0, \infty) \times \mathbb R \rightarrow \ell_1^+$ to the system of PDEs~\eqref{Paper02_PDE_scaling_limit} satisfying the conditions of Theorem~\ref{Paper02_deterministic_scaling_foutel_rodier_etheridge} as the unique continuous mild solution to~\eqref{Paper02_PDE_scaling_limit}.

\subsection{Asymptotic behaviour in the monostable regime} \label{Paper02_sec:subsection_asymptotic_behaviour_PDE}

In~\cite{foutel2020spatial}, Foutel-Rodier and Etheridge state (without rigorous proofs) that if the system of PDEs~\eqref{Paper02_PDE_scaling_limit} is started from equilibrium, 
then the proportions of individuals carrying a certain number of mutations would subsequently be the same in the bulk as in the front of the range expansion. Using this hypothesis, they determine (non-rigorously) the spreading speed of the population into an empty habitat.

Our next main results establish rigorous versions of these conjectures under certain conditions on the reaction term~\eqref{Paper02_reaction_term_PDE} of the system of PDEs~\eqref{Paper02_PDE_scaling_limit}. To introduce these conditions, we first recall the terminology usually used to describe one-dimensional reaction-diffusion equations. Consider the PDE
\begin{equation} \label{Paper02_one_dimensional_RD}
    \partial_{t}U = \frac{m}{2}\Laplace U + g(U) \quad \textrm{for } t > 0,
\end{equation}
where $g \in \mathscr{C}^{1}(\mathbb{R})$, $g(0) = g(1) = 0$ and $\int_{0}^{1} g(U) \, dU > 0$. We say that the reaction term $g$ is \textit{monostable} if $g'(0) > 0$, $g'(1) < 0$ and $g(x) > 0$ for all $x \in (0,1)$ (see e.g.~\cite{garnier2012inside}). We say that the reaction term $g$ is of \emph{Fisher-KPP type} if~$g$ is monostable and $g(U) \leq g'(0)U$ for all $U \in (0,1)$. From a population dynamics point of view, the Fisher-KPP condition corresponds to the setting in which competition dominates over cooperation. In our next definition, we adapt these concepts to the system of PDEs~\eqref{Paper02_PDE_scaling_limit}.

\begin{definition}[Monostable and Fisher-KPP reaction terms] \label{Paper02_assumption_monostable_condition}
We say that the reaction term $F=(F_{k})_{k \in \mathbb{N}_{0}}: \ell_1^+ \rightarrow \ell_1$ defined in~\eqref{Paper02_reaction_term_PDE} is \emph{monostable} if the sequence of fitness parameters $(s_{k})_{k \in \mathbb{N}_{0}}$ satisfies Assumption~\ref{Paper02_assumption_fitness_sequence} and is strictly decreasing, the mutation rate satisfies $\mu \in (0,1)$, and the functions $q_{+},q_{-}: [0,\infty) \rightarrow [0,\infty)$ satisfy Assumption~\ref{Paper02_assumption_polynomials} and the following conditions:
\begin{enumerate}[(i)]
    \item $q_{+}(U) \geq q_{-}(U) \quad \forall \, U \in [0,1]$.
    \item $q_{+}(U) > 0 \; \forall \, U \in [0,1]$.
    \item $q_{+}(1) = q_{-}(1)$.
    \item $q'_{+}(1) < q'_{-}(1)$.
    \item $(1-\mu)q_{+}(0) - q_{-}(0) > 0$.
\end{enumerate}
Finally, if in addition to the conditions above, we also have
\begin{equation} \label{Paper02_fisher_kpp_condition_muller_ratchet}
        (1-\mu)q_{+}(U) - q_{-}(U) \leq (1-\mu)q_{+}(0) - q_{-}(0) \quad \forall \, U \in [0,1],
\end{equation}
then we say that the reaction term $F=(F_{k})_{k \in \mathbb{N}_{0}}$ is of \emph{Fisher-KPP type}.
\end{definition}

Note that if the reaction term $g(U) = U(q_+(U) - q_-(U))$ of the one-dimensional PDE
\begin{equation} \label{Paper02_correspondent_one_dimensional_system}
    \partial_{t} U = \frac{m}{2} \Laplace U + U(q_{+}(U) - q_{-}(U))
\end{equation}
is monostable, then conditions (i), (iii) and (iv) of Definition~\ref{Paper02_assumption_monostable_condition} are automatically satisfied. Moreover, if the reaction term of~\eqref{Paper02_correspondent_one_dimensional_system} is monostable, then $q_{+}(0) >q_{-}(0)$, and therefore condition~(v) of Definition~\ref{Paper02_assumption_monostable_condition} must hold for sufficiently small $\mu$. We also must have $q_+(U) > 0$ for all $U \in (0,1)$, but the additional condition $q_+(1) > 0$ in condition (ii) is not implied by the monostability of the reaction term of~\eqref{Paper02_correspondent_one_dimensional_system}; this is a technical assumption required in our proofs. This condition, however, is not restrictive, since it is reasonable to assume in biological models that the per-capita reproduction and death rates are both strictly positive in high population density.

In order to determine the evolution of the ratios between the population densities carrying particular numbers of mutations, we will need some control on these ratios in the initial condition. This will be our next assumption. Recall that we denote the Lebesgue measure on $\mathbb R$ by $\lambda$.

\begin{assumption}[Control on the initial prevalence of mutations] \label{Paper02_assumption_initial_condition_modified}
Let $f = \left(f_{k}\right)_{k \in \mathbb{N}_{0}}: \mathbb{R} \rightarrow \ell_{1}^{+}$ be a function satisfying Assumption~\ref{Paper02_assumption_initial_condition} and the following additional conditions:
\begin{enumerate}[(i)]
    \item $\displaystyle \| f \|_{L_{\infty}(\mathbb{R}; \ell_{1})} \defeq \esssup_{x \in \mathbb{R}} \, \| f(x) \|_{\ell_{1}} \leq 1$.
   \item $\lambda(\{x \in \mathbb R: \, f_0(x) > 0\}) > 0$.
   \item There exists a sequence $\left(\hat{\pi}_{k}\right)_{k \in \mathbb{N}_{0}} \in \ell_{1}^{+}$ such that for each $k \in \mathbb{N}$ and $\lambda$-almost every $x \in \mathbb R$,
   \begin{equation*}
       0 \leq f_{k}(x) \leq \hat{\pi}_{k} f_{0}(x).
   \end{equation*}
\end{enumerate}
\end{assumption}

For $\mu \in (0,1)$, let $(\alpha_{k})_{k \in \mathbb{N}_{0}} = \Big(\alpha_{k}(\mu, (s_j)_{j \in \mathbb N_0})\Big)_{k \in \mathbb{N}_{0}}$ be given by $\alpha_{0} \defeq 1$ and
\begin{equation} \label{Paper02_definition_alpha_k}
    \alpha_{k} \defeq \prod_{i=1}^{k} \frac{\mu s_{i-1}}{(1-\mu) (1 - s_{i})} \quad \; \forall \, k \in \mathbb{N}.
\end{equation}
Since, by Assumption~\ref{Paper02_assumption_fitness_sequence}, $\lim_{k \rightarrow \infty} s_{k} = 0$, we have that $(\alpha_{k})_{k \in \mathbb{N}_{0}} \in \ell_{1}^{+}$. Moreover, under the assumption that $s_k > 0$ for every $k \in \mathbb{N}_0$, which is required in Definition~\ref{Paper02_assumption_monostable_condition} for the reaction term $F = (F_k)_{k \in \mathbb{N}_0}$ defined in~\eqref{Paper02_reaction_term_PDE} to be monostable, we also have $\alpha_{k} > 0$ for every $k \in \mathbb{N}_{0}$.

\begin{remark}
We claim that, at least heuristically, it is straightforward to see that the sequence~$(\alpha_k)_{k \in \mathbb N_0}$ defined in~\eqref{Paper02_definition_alpha_k} gives the proportions of the population carrying different numbers of mutations at equilibrium. Indeed, suppose that the polynomials $q_+$ and $q_-$ satisfy Assumption~\ref{Paper02_assumption_polynomials} and are strictly positive on $(0,\infty)$. Let $F = (F_k)_{k \in \mathbb N_0}$ be the reaction term defined in~\eqref{Paper02_reaction_term_PDE}, and let $u^{(\mathrm{eq})} = (u^{(\mathrm{eq})}_k)_{k \in \mathbb N_0} \in \ell_1^+$ be a stationary solution of~\eqref{Paper02_PDE_scaling_limit} such that $u^{(\mathrm{eq})}_0 \neq 0$ and $F(u^{(\mathrm{eq})}) = 0$.
From the definition of $F_0$ in~\eqref{Paper02_reaction_term_PDE}, we obtain
\[
q_-(\|u^{(\mathrm{eq})}\|_{\ell_1}) = (1-\mu) q_+(\|u^{(\mathrm{eq})}\|_{\ell_1}).
\]
Using this identity together with the definition of the reaction term $F$,~\eqref{Paper02_definition_alpha_k} and the fact that we assume $q_+$ and $q_-$ to be strictly positive on $(0,\infty)$, it follows that $
u^{(\mathrm{eq})}_k = \alpha_k u^{(\mathrm{eq})}_0 \; \forall\, k \in \mathbb N_0$.
Thus, up to normalisation,~$(\alpha_k)_{k \in \mathbb N_0}$ gives the equilibrium distribution of the different numbers of mutations.
\end{remark}

Let
\begin{equation} \label{Paper02_definition_max_min_birth_polynomial}
\mathfrak{Q}_{\min} \defeq \displaystyle \min_{U \in [0,1]} q_{+}(U) \quad \textrm{and} \quad \mathfrak{Q}_{\max} \defeq \displaystyle \max_{U \in [0,1]} q_{+}(U).
\end{equation}
By Definition~\ref{Paper02_assumption_monostable_condition}(ii), if the reaction term $F = (F_{k})_{k \in \mathbb{N}_{0}}$ is monostable, then $\mathfrak{Q}_{\max} \geq \mathfrak{Q}_{\min} > 0$. We can now state our result determining the prevalence of mutations during a range expansion.

\begin{theorem} \label{Paper02_prop_control_proportions}
    Suppose that $(s_k)_{k \in \mathbb N_0}$, $q_+$ and $q_-$ satisfy Assumptions~\ref{Paper02_assumption_fitness_sequence} and~\ref{Paper02_assumption_polynomials}, that $\mu \in (0,1)$, that $m > 0$, that the reaction term $F = (F_{k})_{k \in \mathbb{N}_{0}}$ defined in~\eqref{Paper02_reaction_term_PDE} is monostable in the sense of Definition~\ref{Paper02_assumption_monostable_condition},  and that $f$ satisfies Assumption~\ref{Paper02_assumption_initial_condition_modified}. Let $u = (u_{k})_{k \in \mathbb{N}_{0}}$ be the unique continuous mild solution to the system of PDEs~\eqref{Paper02_PDE_scaling_limit}. Define $(\alpha_k)_{k \in \mathbb N_0}$ as in~\eqref{Paper02_definition_alpha_k}, and $\mathfrak{Q}_{\min}$,~$\mathfrak{Q}_{\max}$ as in~\eqref{Paper02_definition_max_min_birth_polynomial}. Then there exist functions $\underline{\pi} = \left(\underline{\pi}_{k}\right)_{k \in \mathbb{N}_0}: [0, \infty) \rightarrow \ell_1^+$ and $\overline{\pi} = \left(\overline{\pi}_{k}\right)_{k \in \mathbb{N}_0}: [0,\infty) \rightarrow \ell_{1}^{+}$ such that the following limits hold in $\ell_1$:
    \begin{equation*}
        \lim_{T \rightarrow \infty} \underline{\pi}(T) = \Bigg(\alpha_{k}\left(\frac{\mathfrak{Q}_{\min}}{\mathfrak{Q}_{\max}}\right)^{k}\Bigg)_{k \in \mathbb{N}_0} \quad \textrm{ and } \quad  \lim_{T \rightarrow \infty} \overline{\pi}(T) = \Bigg(\alpha_{k}\left(\frac{\mathfrak{Q}_{\max}}{\mathfrak{Q}_{\min}}\right)^{k}\Bigg)_{k \in \mathbb{N}_0},
    \end{equation*}
    and such that 
    for all $k \in \mathbb{N}$, $T > 0$ and $x \in \mathbb{R}$, 
    \begin{equation} \label{Paper02_final_control_prop_eq_determ}
        \underline{\pi}_{k}(T) u_{0}(T,x) \leq u_{k}(T,x) \leq \overline{\pi}_{k}(T) u_{0}(T,x).
    \end{equation}
\end{theorem}

\begin{remark} \label{Paper02_very_good_control_proportions_FKPP}
\begin{enumerate}[(i)]
    \item One can always take $\underline{\pi}_0(T) = \overline{\pi}_0(T) = 1$ for all $T \geq 0$, and so~\eqref{Paper02_final_control_prop_eq_determ} holds trivially for $k=0$.
    \item If $\mathfrak{Q}_{\min} = \mathfrak{Q}_{\max}$, i.e.~if $q_{+}$ is a constant function, Theorem~\ref{Paper02_prop_control_proportions} shows that the ratio of population density carrying exactly $k$ mutations to population density without mutations converges uniformly (in space) as $T \rightarrow \infty$ to $\alpha_{k}$, where $\alpha_{k}$ is given by~\eqref{Paper02_definition_alpha_k}. This is the case for the system corresponding to the classical Fisher-KPP equation, with $q_{+}(U) = 1$ and $q_{-}(U) = U$ for all $U \in [0,1]$.
\end{enumerate}

\end{remark}

We highlight again that by combining the definition of $\alpha_{k}$ in~\eqref{Paper02_definition_alpha_k}, the fact that $\lim_{k \rightarrow \infty} s_{k} = 0$ by Assumption~\ref{Paper02_assumption_fitness_sequence}, and the fact that Definition~\ref{Paper02_assumption_monostable_condition}(ii) implies that $0 < \mathfrak{Q}_{\min} \leq \mathfrak{Q}_{\max}$, we see that both sequences $$\left(\alpha_{k}\left(\frac{\mathfrak{Q}_{\min}}{\mathfrak{Q}_{\max}}\right)^{k}\right)_{k \in \mathbb{N}_{0}}\quad \textrm{and} \quad \left(\alpha_{k}\left(\frac{\mathfrak{Q}_{\max}}{\mathfrak{Q}_{\min}}\right)^{k}\right)_{k \in \mathbb{N}_{0}}$$ are elements of $\ell_{1}^{+}$. 
Theorem~\ref{Paper02_prop_control_proportions} implies that the ratio between the local population density with $k$ mutations and the local population density with $0$ mutations can be uniformly bounded
away from $0$ and $\infty$ at large times.
Although this result is not the same as the conjecture from~\cite{foutel2020spatial} described at the start of this subsection, it indicates that the proportion of the population carrying a certain number of mutations is well behaved during a range expansion. It also implies that if the population without mutations spreads into an empty habitat with some speed, then the population with mutations also spreads into the empty habitat with the same speed. Using these consequences of Theorem~\ref{Paper02_prop_control_proportions}, we can determine the spreading speed for solutions of the system of PDEs in the case when the reaction term $F = (F_k)_{k \in \mathbb N_0}$ is of Fisher-KPP type.

\begin{theorem} \label{Paper02_spreading_speed_Fisher_KPP}
    Suppose that $(s_k)_{k \in \mathbb N_0}$, $q_+$ and $q_-$ satisfy Assumptions~\ref{Paper02_assumption_fitness_sequence} and~\ref{Paper02_assumption_polynomials}, that $\mu \in (0,1)$, that $m > 0$, that the reaction term $F = (F_{k})_{k \in \mathbb{N}_{0}}$ defined in~\eqref{Paper02_reaction_term_PDE} is of Fisher-KPP type in the sense of Definition~\ref{Paper02_assumption_monostable_condition},  and that $f$ satisfies Assumption~\ref{Paper02_assumption_initial_condition_modified}. Further, assume that there exists $R > 0$ such that the support of the initial condition satisfies $\supp f \subset (-\infty, R]$. Let $$c^{*} \defeq \sqrt{2m\Big((1-\mu)q_{+}(0) - q_{-}(0)\Big)},$$ and let $u = (u_k)_{k \in \mathbb N_0}$ denote the unique continuous mild solution to the system of PDEs~\eqref{Paper02_PDE_scaling_limit} with initial condition~$f$. Then
    \begin{equation*}
        \lim_{T \rightarrow \infty} \, \sup_{x \geq c^{*}T} \| u(T,x) \|_{\ell_{1}} = 0.
    \end{equation*}
    Moreover, there exists a sequence $(\nu_{k})_{k \in \mathbb{N}_{0}} \in \ell_{1}^{+}$ with $\nu_k>0$ for every $k \in \mathbb N_0$ such that for any $\varepsilon \in (0,c^{*})$,
    \begin{equation*}
        \liminf_{T \rightarrow \infty} \, \inf_{x \in [0, (c^{*} -\varepsilon)T]} \, u_{k}(T,x) \geq \nu_{k} \quad \forall \, k \in \mathbb N_0.
    \end{equation*}
\end{theorem}

We will apply this result to some biologically relevant examples in Section~\ref{Paper02_example_pop_dynamics} below.

\subsection{Tracer dynamics and gene surfing} \label{Paper02_sec:tracer_dynamics_intro}

Our final result answers the question that motivated this article:~can deleterious mutations surf population waves? Theorem~\ref{Paper02_prop_control_proportions} 
implies that in the monostable regime, the subpopulation carrying deleterious mutations propagates through space together with the subpopulation without mutations. 
It remains then to answer whether this phenomenon occurs due to mutation-selection equilibrium or if deleterious mutations can indeed surf. To distinguish these effects, 
we will use the idea of \emph{tracer dynamics} introduced by Hallatschek and Nelson in~\cite{hallatschek2008gene} and formalised by Garnier and co-authors in~\cite{garnier2012inside} in the deterministic setting.

In the setting of the underlying particle system, the idea of tracer dynamics is to label some subset of particles at time $0$. The label does not affect the mutation probability or the migration, reproduction and death rates. When a labelled particle reproduces, its offspring particle is also a labelled particle. The aim is to track the number of labelled particles at each spatial location over time, and use this to infer information about the typical history of particles in the system.
In order to translate this idea to the context of the system of PDEs~\eqref{Paper02_PDE_scaling_limit}, we take $f^{*} = (f^{*}_{k})_{k \in \mathbb{N}_{0}}: \mathbb{R} \rightarrow \ell_{1}^{+}$, and we think of $f^{*}_{k}$ as indicating the initial population density that carries exactly $k$ mutations and is labelled. As usual, we will take $f = (f_k)_{k \in \mathbb N_0}: \mathbb R \rightarrow \ell_1^+$ and think of $f_k$ as indicating the initial population density that carries exactly $k$ mutations. We make the following assumptions about $f^{*}$.

\begin{assumption}[Tracer dynamics initial condition] \label{Paper02_assumption_initial_condition_labelled_particles}
    Let $f: \mathbb{R} \rightarrow \ell_{1}^{+}$ satisfy Assumptions~\ref{Paper02_assumption_initial_condition} and~\ref{Paper02_assumption_initial_condition_modified}. Let $f^{*} = (f^{*}_{k})_{k \in \mathbb{N}_{0}}: \mathbb{R} \rightarrow \ell_{1}^{+}$ satisfy the following conditions:
    \begin{enumerate}[(i)]
        \item $f^{*} \in L_{\infty}(\mathbb{R};\ell_{1})$, i.e.~$\esssup_{x \in \mathbb R} \| f^*(x) \|_{\ell_1} < \infty$.
        \item $f^{*}$ is continuous almost everywhere, i.e.~there exists $\mathcal{N} \subset \mathbb{R}$ such that $\lambda(\mathcal{N}) = 0$ and $f^*$ is continuous on $\mathbb R \setminus \mathcal N$.
        \item $f_{k}^{*}(x) \leq f_{k}(x)$ for all $x \in \mathbb{R}$ and all $k \in \mathbb{N}_{0}$.
    \end{enumerate}
\end{assumption}

Note that condition~(iii) 
guarantees that the initially labelled population is a subset of the total initial population. For $t \geq 0$, $x \in \mathbb R$ and $k \in \mathbb N_0$, we indicate the density of the labelled population carrying exactly $k$ mutations at position $x$ at time $t$ by $u^{*}_{k}(t,x)$. Suppose $f$ and $f^*$ satisfy Assumption~\ref{Paper02_assumption_initial_condition_labelled_particles}. Then, recalling the heuristics after~\eqref{Paper02_PDE_scaling_limit}-\eqref{Paper02_reaction_term_PDE}, we can see from the definition of the labelled set in the particle system setting that $u^{*} = (u^{*}_{k})_{k \in \mathbb{N}_{0}}: [0, \infty) \times \mathbb{R} \rightarrow \ell_{1}^{+}$ should satisfy the system of PDEs
\begin{equation} \label{Paper02_PDE_system_labelled_particles}
\begin{aligned}
    \partial_{t} u^{*}_{k} & = \frac{m}{2} \Laplace u^{*}_{k} + F^{*}_{k}(u, u^{*}) \quad 
    \forall k \in \mathbb N_0, \, t > 0, \\
    u^{*}(0,\cdot) & = f^{*}(\cdot),
\end{aligned}
\end{equation}
where $u: [0,\infty) \times \mathbb{R} \rightarrow \ell_{1}^{+}$ is the unique continuous mild solution to the system of PDEs~\eqref{Paper02_PDE_scaling_limit}, and the reaction term $F^{*} = (F^{*}_{k})_{k \in \mathbb{N}_{0}}: \ell_{1}^{+} \times \ell_{1} \rightarrow \ell_{1}$ is given by, for all $u = (u_k)_{k \in \mathbb N_0} \in \ell_1^+$ and $u^* = (u^*_k)_{k \in \mathbb N_0} \in \ell_1$,
\begin{equation} \label{Paper02_reaction_term_PDE_labelled}
    F^{*}_{k}(u, u^{*}) \defeq q_{+}(\| u \|_{\ell_1})\left(s_{k}(1-\mu)u^{*}_{k} + \mathds{1}_{\{k \geq 1\}}s_{k-1}\mu u^{*}_{k-1}\right) - q_{-}(\| u \|_{\ell_1})u^{*}_{k} \quad 
    \forall \, k \in \mathbb N_0.
\end{equation}
The definition of $F^{*}= (F^{*}_{k})_{k \in \mathbb{N}_{0}}$ reflects the fact that in the particle system with tracer dynamics, the per-capita birth and death rates of labelled particles are determined by the local density of the total population, but only labelled particles can give birth to other labelled particles. 
In order to study gene surfing of deleterious mutations, we label some subset of the initial population that carries at least one mutation, and determine the behaviour of $u^{*}(T, \cdot)$ as $T \rightarrow \infty$.

\begin{theorem} \label{Paper02_thm_tracer_dynamics_no_gene_surfing}
    Suppose that $(s_k)_{k \in \mathbb N_0}$, $q_+$ and $q_-$ satisfy Assumptions~\ref{Paper02_assumption_fitness_sequence} and~\ref{Paper02_assumption_polynomials}, that $\mu \in (0,1)$ and $m > 0$, and that $f$ and $f^*$ satisfy Assumption~\ref{Paper02_assumption_initial_condition_labelled_particles}, and that the reaction term $F = (F_{k})_{k \in \mathbb{N}_{0}}$ defined in~\eqref{Paper02_reaction_term_PDE} is monostable in the sense of Definition~\ref{Paper02_assumption_monostable_condition}. Then there exists a unique mild solution $u^{*} = (u^{*}_{k})_{k \in \mathbb{N}_{0}}: [0, \infty) \times \mathbb{R} \rightarrow \ell_{1}^{+}$ to the system of PDEs~\eqref{Paper02_PDE_system_labelled_particles} satisfying the following conditions:
    \begin{enumerate}[(i)]
        \item $u^* \in \mathscr C((0,\infty) \times \mathbb R; \ell_1)$.
         \item $\sup_{t \in [0,T]} \| u^*(t,\cdot) \|_{L_{\infty}(\mathbb R; \ell_1)} < \infty$ for all $T > 0$.
        \item $u^{*}_{k} \in \mathscr{C}^{1,2}((0, \infty) \times \mathbb{R}; \mathbb R)$ for every $k \in \mathbb N_0$.
    \end{enumerate}
    Moreover, if $f^{*}_{0}(\cdot) \equiv 0$, then
    \begin{equation*}
        \lim_{T \rightarrow \infty} \| u^{*}(T, \cdot) \|_{L_{\infty}(\mathbb{R}; \ell_{1})} = 0.
    \end{equation*}
\end{theorem}

We will give a precise definition of continuous mild solutions to the system of PDEs~\eqref{Paper02_PDE_system_labelled_particles} in Section~\ref{Paper02_tracer_dynamics_section}, and define the norm $\| \cdot \|_{L_{\infty}(\mathbb R; \ell_1)}$ in Section~\ref{Paper02_PDES_sequence_spaces}. Theorem~\ref{Paper02_thm_tracer_dynamics_no_gene_surfing} complements the description of the asymptotic behaviour of the solution of~\eqref{Paper02_PDE_scaling_limit} given by Theorem~\ref{Paper02_prop_control_proportions} (in the monostable case). It shows that for a population expanding according to~\eqref{Paper02_PDE_scaling_limit}, although by Theorem~\ref{Paper02_prop_control_proportions} there is a positive fraction of the population carrying deleterious mutations in both the front and the bulk of the expanding population, this fraction descends from the expanding subpopulation without mutations and is not the result of surfing of previously existing deleterious mutations; see Figure~\ref{Paper02_FKPP_with_tracer} for a simulation of this phenomenon. Therefore, we can say that in the monostable regime, deleterious mutations cannot surf deterministic population waves.

\begin{figure}[htp]
\centering
\includegraphics[width=\linewidth,trim=0cm 0cm 0cm 0cm,clip=true]{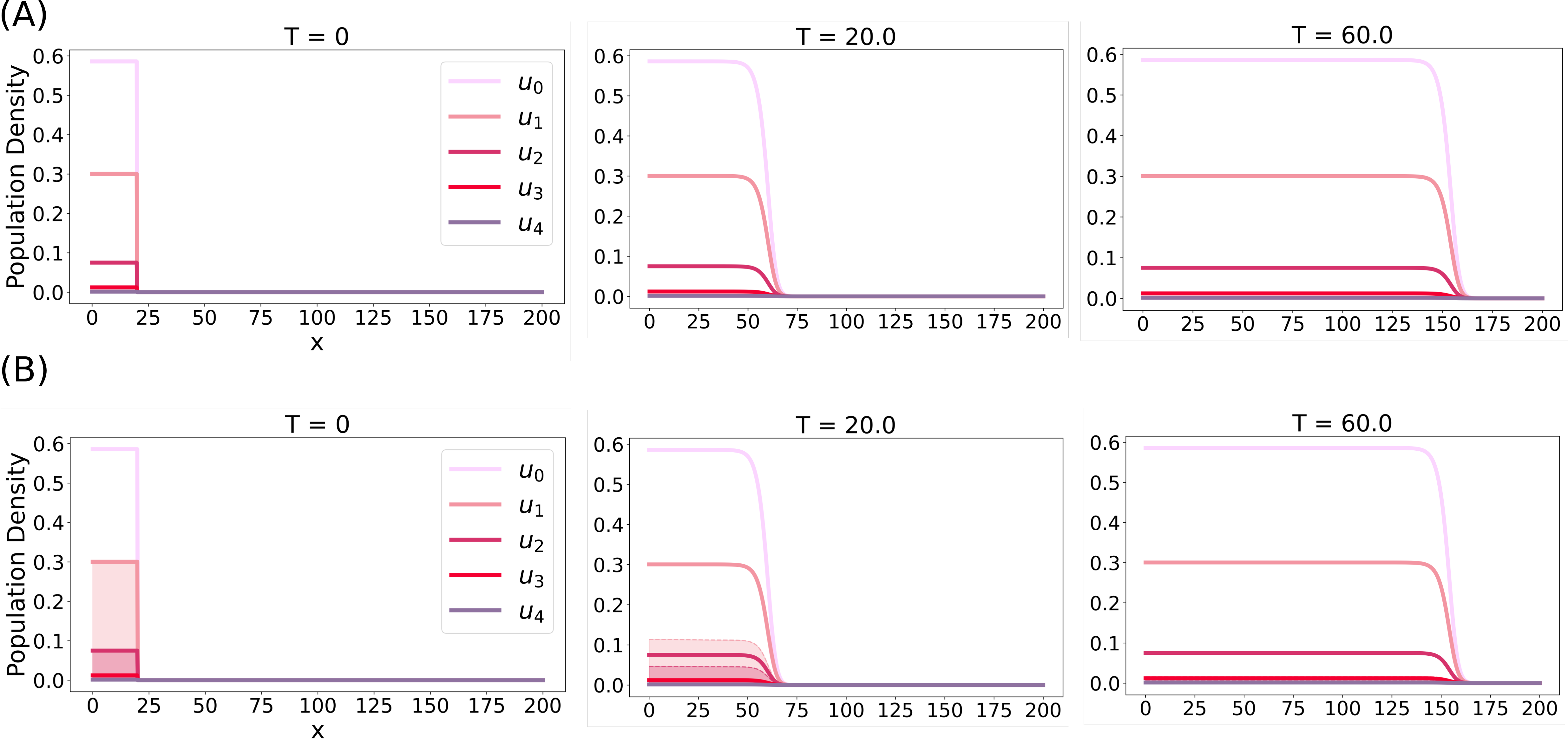}
\caption{\label{Paper02_FKPP_with_tracer}~Numerical simulation of a population with Fisher–KPP type dynamics, governed by the systems of PDEs~\eqref{Paper02_PDE_scaling_limit} and~\eqref{Paper02_PDE_system_labelled_particles}. The parameters used in the simulation are $m = 3$, $q_+(U) = 1$ and $q_-(U) = U$ for all $U \geq 0$, $\mu = 0.025$, and $s_k = 0.95^k$ for all $k \in \mathbb{N}_0$. The horizontal axis indicates spatial position $x$, and the vertical axis shows population density. The time at which the snapshot is taken is denoted by $T$. Colours indicate the density of population carrying different numbers of mutations, i.e.~$u_k(T,\cdot)$ for $k \in \{0,1,2,3,4\}$ as shown in the legend of the $T = 0$ figure. In~(A), the simulation is performed without tracers. The simulation illustrates that the fractions of the population carrying positive numbers of mutations invade previously uncolonised habitat at the same speed as the fraction carrying no mutations. In~(B), the shaded regions indicate the labelled fraction of the population. At time $T = 0$, only the subpopulation without mutations is unlabelled. As shown in Theorem~\ref{Paper02_thm_tracer_dynamics_no_gene_surfing}, the labelled population, corresponding to descendants of
the initial mutant subpopulation, does not propagate with the expansion wave. This shows that the propagation of the fractions of the population carrying mutations occurs due to a mutation-selection equilibrium, and deleterious mutations cannot surf in the deterministic setting.}
\end{figure}

\section{Application to different population dynamics} \label{Paper02_example_pop_dynamics}

Since the system of PDEs in~\eqref{Paper02_PDE_scaling_limit} is rather general, it can be used to model asexual populations under different biological assumptions. For instance, taking, for some~$r > 0$ and~$B \geq 0$,
\begin{equation} \label{Paper02_foutel_etheridge_birth_death_rates}
   q_{+}(U) \defeq r(BU + 1) \quad \textrm{ and } \quad q_{-}(U) \defeq r(BU+1)U \quad \forall U \geq 0,
\end{equation}
we recover the same system of PDEs proposed by Foutel-Rodier and Etheridge in~\cite{foutel2020spatial} (see~\eqref{Paper02_conjectured_PDE_heuristics}). Phenomenologically, $r$ indicates the Malthusian growth factor, and $B$ reflects the dependence of the growth rate on the local population density. Note that the value of the population density $U$ at which the maximum  of the function
\begin{equation*}
    [0, \infty) \ni U \mapsto q_{+}(U) - q_{-}(U) = r(1 - U)(BU+1)
\end{equation*}
is attained depends on the parameter $B$. If $0 \leq B \leq 1$, then the function attains its maximum when $U = 0$, i.e.~the maximum per-capita growth rate is attained at $U = 0$. We say that in this case the population dynamics (neglecting, for now, the effect of deleterious mutations) satisfies a Fisher-KPP condition~\cite{garnier2012inside}. Biologically, this means that competition prevails over cooperation in determining the population dynamics.

For $B > 1$, the function $[0,\infty) \ni U \mapsto q_{+}(U) - q_{-}(U)$ attains its maximum at $U = \frac{B-1}{2B} > 0$. In this case, we say the population (neglecting the effect of deleterious mutations) exhibits a weak Allee effect, since the maximum per-capita growth rate is attained at an intermediate population density rather than at low population density \cite{foutel2020spatial}. From a biological point of view, this scenario corresponds to the case where both competition and cooperation affect the population dynamics.

Since the polynomials $q_+$ and $q_-$ satisfy Definition~\ref{Paper02_assumption_monostable_condition}(i)-(v) for any $B \geq 0$, $r > 0$ and $\mu \in (0,1)$, from Theorem~\ref{Paper02_prop_control_proportions} we have quantitative bounds on the proportions of the population carrying $k \in \mathbb{N}_0$ mutations in the solution of the system of PDEs~\eqref{Paper02_PDE_scaling_limit} at large times, for suitable initial conditions. Theorem~\ref{Paper02_thm_tracer_dynamics_no_gene_surfing} implies that fractions with mutations expand through space due to the fact that the fraction of the population without mutations propagates and due to a mutation-selection equilibrium, i.e.~there is no gene surfing of deleterious mutations in this regime. From a modelling perspective, in the spatial Muller's ratchet, stochastic behaviour arises from the finite local population size, while the behaviour of very large populations is described by the deterministic system of PDEs~\eqref{Paper02_PDE_scaling_limit} (see~\cite[Theorem~2.1]{madeira2025LLN}). We highlight that our results do not answer whether gene surfing is possible in the stochastic setting. However, they shed light on the understanding of surfing: even if deleterious mutations can surf when the population is finite (i.e.~in the stochastic case), the effect of the surfing must be limited, since Theorem~\ref{Paper02_thm_tracer_dynamics_no_gene_surfing} shows that no gene surfing of deleterious mutations occurs in the deterministic limit (in the monostable regime). Therefore, our results may 
help explain the contradictory findings regarding gene surfing of deleterious mutations in biological data (see the discussion in Section~\ref{Paper02_introduction} and the references therein).

The assumptions for Theorem~\ref{Paper02_spreading_speed_Fisher_KPP}, which 
determines the spreading speed, require that the reaction term $F = (F_k)_{k \in \mathbb N_0}$ defined in~\eqref{Paper02_reaction_term_PDE} is of Fisher-KPP type in the sense of Definition~\ref{Paper02_assumption_monostable_condition}, i.e.~that for all $U \in [0,1]$,
\begin{equation} \label{Paper02_regime_fisher_kpp_mullers_ratchet}
    (1-\mu)q_{+}(U) - q_{-}(U) = r(1 - \mu - U)(BU + 1) \leq (1-\mu)q_{+}(0) - q_{-}(0) = r(1 - \mu).
\end{equation}
Inequality~\eqref{Paper02_regime_fisher_kpp_mullers_ratchet} holds for all $U \in [0,1]$ if and only if $0 \leq B \leq \frac{1}{1 - \mu}$. We note that for $B \in \left(1, \frac{1}{1-\mu}\right]$, the reaction term $F = (F_k)_{k \in \mathbb N_0}$ is of Fisher-KPP type but, as noted above, the population dynamics, neglecting the effect of mutations, exhibits a weak Allee effect.

If $0 \leq B \leq \frac{1}{1 - \mu}$, then by Theorem~\ref{Paper02_spreading_speed_Fisher_KPP}, the spreading speed of a population evolving according to the system of PDEs~\eqref{Paper02_PDE_scaling_limit} into an empty habitat, under suitable assumptions on the initial condition, is given by $$c^{*} = \sqrt{2mr(1-\mu)}.$$
This confirms the calculations in~\cite{foutel2020spatial} for $B \leq \frac{1}{1-\mu}$ (recall~\eqref{Paper02_conjecture_spreading_speed_foutel_rodier_etheridge}), with a more general choice of fitness parameters $(s_k)_{k \in \mathbb N_0}$ and mutation rate $\mu$ than considered in~\cite{foutel2020spatial}. Theorem~\ref{Paper02_spreading_speed_Fisher_KPP} also implies that, under the same assumptions, $c^*$ is the spreading speed of the fraction of the population carrying $k$ mutations, for every~$k \in \mathbb{N}_0$. Note that Theorem~\ref{Paper02_spreading_speed_Fisher_KPP} does not apply when $B > \frac{1}{1-\mu}$. Consequently, the validity of~\eqref{Paper02_conjecture_spreading_speed_foutel_rodier_etheridge} for these values of~$B$ remains an open problem.

\medskip
Another interesting example is to take
\begin{equation*}
    q_{+}(U) \defeq U(B+1) \quad \textrm{and} \quad q_{-}(U) \defeq U^{2} + B \quad \forall \, U \geq 0, \textrm{ for some } B \in \left(0, 1/2\right).
\end{equation*}
Then the function $U \mapsto q_{+}(U) - q_{-}(U) = (1-U)(U-B)$ is negative for $U \in [0,B)$. We say that the population (neglecting the effect of deleterious mutations) exhibits a strong Allee effect \cite{garnier2012inside}. Biologically, this means that cooperation is fundamental for population growth, 
as the population shrinks at low population densities. In this case, our Theorem~\ref{Paper02_deterministic_scaling_foutel_rodier_etheridge} again shows existence and uniqueness of a mild solution of the corresponding system of PDEs~\eqref{Paper02_PDE_scaling_limit}. However, the asymptotic behaviour of solutions of~\eqref{Paper02_PDE_scaling_limit} in this case is not covered by Theorems~\ref{Paper02_prop_control_proportions} and~\ref{Paper02_spreading_speed_Fisher_KPP}, because the reaction term $F = (F_{k})_{k \in \mathbb{N}_{0}}$ given by~\eqref{Paper02_reaction_term_PDE} is not monostable in the sense of Definition~\ref{Paper02_assumption_monostable_condition}. Since we do not expect gene surfing of deleterious mutations in the presence of a strong Allee effect, we conjecture that a result analogous to Theorem~\ref{Paper02_thm_tracer_dynamics_no_gene_surfing} should hold in that case.
To our knowledge, there are no rigorous 
results on the spreading speed and long-term asymptotic behaviour of (even finite) systems of PDEs in the presence of a strong Allee effect, and hence this is an interesting open problem.

We can also consider the effect of different choices of fitness parameters $(s_k)_{k \in \mathbb N_0}$. We first consider the model introduced in~\cite{foutel2020spatial}, where the sequence of fitness parameters is given by $s_{k} \defeq (1 - s)^{k}$, for $k \in \mathbb N_0$ and for some fixed $s \in (0,1)$. In this case, when a new mutation is acquired, the fitness of a chromosome always decreases by a factor $(1-s)$. In other words, the effect of each mutation is independent of the genetic background. Biologically, we say there is no \emph{epistasis}. As explained in Section~\ref{Paper02_introduction}, in the case $s, \mu \ll 1$, i.e.~under a weak selection--low mutation regime, Foutel-Rodier and Etheridge derive non-rigorously that the stationary solution of the system of PDEs~\eqref{Paper02_conjectured_PDE_heuristics} can be approximated by~\eqref{Paper02_approximating_equilibrium_weak_selection_low_mutation_regime}. To formalise the weak selection--low mutation regime, let $(s^{(n)})_{n \in \mathbb{N}}, (\mu^{(n)})_{n \in \mathbb{N}} \subset (0,\infty)$ and $s, \mu \in (0, \infty)$ be such that
\begin{equation} \label{Paper02_weak_selection_low_mutation_regime}
    \lim_{n \rightarrow \infty} ns^{(n)} = s \quad \textrm{and} \quad  \lim_{n \rightarrow \infty} n\mu^{(n)} = \mu.
\end{equation}
Then, for each $n \in \mathbb{N}$ and $k \in \mathbb{N}_{0}$, let $s^{(n)}_k \defeq (1 -s^{(n)})^k$. For $n \in \mathbb{N}$ sufficiently large that $s^{(n)}, \mu^{(n)} \in (0,1)$, we can let $s_k = s^{(n)}_k \; \forall \, k \in \mathbb N_0$ and $\mu = \mu^{(n)}$ determine the sequence of fitness parameters and the mutation probability for the system of PDEs~\eqref{Paper02_PDE_scaling_limit}. In this case, the (rescaled) stationary solution of~\eqref{Paper02_PDE_scaling_limit} defined in~\eqref{Paper02_definition_alpha_k} is given by $\alpha^{(n)} = (\alpha^{(n)}_k)_{k \in \mathbb{N}_{0}}$, where $\alpha^{(n)}_0 \defeq 1$ and
\begin{equation} \label{Paper02_sequence_proportions_weak_selection_low_mutation}
    \alpha^{(n)}_k \defeq \left(\frac{\mu^{(n)}}{1 - \mu^{(n)}}\right)^{k} \prod_{i = 1}^{k} \frac{(1 - s^{(n)})^{i-1}}{1 - (1 - s^{(n)})^{i}} \quad \forall \, k \in \mathbb{N}.
\end{equation}
If the polynomials $q_+$ and $q_-$ satisfy Assumption~\ref{Paper02_assumption_polynomials} and Definition~\ref{Paper02_assumption_monostable_condition}(i)-(v), under suitable assumptions on the initial condition,  Theorem~\ref{Paper02_prop_control_proportions} shows for each $n \in \mathbb{N}$ sufficiently large that the ratio of the population density carrying $k$ mutations to the population density carrying $0$ mutations is uniformly bounded between $\alpha^{(n)}_{k}(\frac{\mathfrak{Q}_{\min}}{\mathfrak{Q}_{\max}})^{k}$ and $\alpha^{(n)}_{k}(\frac{\mathfrak{Q}_{\max}}{\mathfrak{Q}_{\min}})^k$.
In particular, as noted in Remark~\ref{Paper02_very_good_control_proportions_FKPP}(ii), in the special case where $q_+$ is constant, the ratio converges uniformly to $\alpha^{(n)}_{k}$. We will prove in Lemma~\ref{Paper02_app_lem_weak_selec_low_mut} in the appendix that $(\alpha^{(n)}_k)_{k \in \mathbb{N}_0}$ converges as $n \rightarrow \infty$
in $\ell_1$ to $(\hat{\alpha}_k)_{k \in \mathbb{N}_0}$ given by
    \begin{equation} \label{Paper02_lim_Poisson_equil}
        \hat{\alpha}_k \defeq \frac{(\mu/s)^{k}}{k!} \quad  \forall \, k \in \mathbb{N}_0.
    \end{equation}
    In particular, for every $k \in \mathbb{N}_0$,
    \begin{equation} \label{Paper02_lim_Poisson_equil_ell_1}
        \lim_{n \rightarrow \infty} \, \frac{\alpha^{(n)}_k}{\| \alpha^{(n)}\|_{\ell_1}} = \exp\left(- \frac{\mu}{s}\right)\frac{(\mu / s)^k}{k!},
    \end{equation}
    which agrees with the right-hand side of~\eqref{Paper02_approximating_equilibrium_weak_selection_low_mutation_regime}. In other words, 
    in the weak selection--low mutation regime, our Theorem~\ref{Paper02_prop_control_proportions} shows that the approximate stationary solution given in~\cite{foutel2020spatial} describes 
    the proportions of the population carrying different numbers of mutations at large times in solutions of~\eqref{Paper02_PDE_scaling_limit}.
    
    Our Assumption~\ref{Paper02_assumption_fitness_sequence} also allows for other choices of fitness parameters. For instance, one could take $s_{k} \defeq \frac{1}{k+1} \; \forall k \in \mathbb{N}_0$. In this case, the addition of another mutation to a genotype with $k$ mutations leads to a reduction of fitness by a factor $s_{k+1}/s_{k} = (k+1)/(k+2)$, 
    which depends on the previous number of mutations $k$. In biological terms, we say that the model allows some degree of epistasis, which can be important in some biological applications~\cite{gros2009evolution}.

\section{Overview of the proofs} \label{Paper02_subsection_heuristics}

Before introducing the main ideas behind the proofs, we highlight the main challenges arising in the analysis of the system of PDEs~\eqref{Paper02_PDE_scaling_limit}. Since~\eqref{Paper02_PDE_scaling_limit} is an infinite system of nonlinear reaction--diffusion equations whose components are coupled through the dependence of the per-capita reproduction and death rates on the total local population density, it can be viewed as a non-local system in type space, with types indexed by~$\mathbb N_0$. A major consequence of this non-local interaction is that the standard comparison principle is not applicable. Many classical tools used in the analysis of reaction--diffusion equations (see e.g.~\cite{aronson1975nonlinear,fife1981comparison}) therefore cannot be applied directly, making it considerably more difficult to study the long-term behaviour and spreading properties of solutions to~\eqref{Paper02_PDE_scaling_limit}. As we explain below, we overcome this difficulty by deriving a Feynman--Kac representation for each coordinate~$u_k$ of the solution. This representation provides quantitative control on the evolution of the mutational profile and forms the basis for our analysis of the spreading behaviour of the solution to~\eqref{Paper02_PDE_scaling_limit}.

The proof of the main results can be roughly divided into two parts. The well-posedness and regularity properties of solutions to~\eqref{Paper02_PDE_scaling_limit} are established in Section~\ref{Paper02_uniqueness_weak_solutions_section}, while the asymptotic behaviour of solutions is characterised in Section~\ref{Paper02_section_asymptotic_behaviour_PDE}. Existence of mild solutions follows directly from our companion article~\cite[Proposition~6.2]{madeira2025LLN}, where we prove that, under an appropriate scaling regime, a sequence of interacting particle systems converges (up to a subsequence) to a mild solution of~\eqref{Paper02_PDE_scaling_limit}. To establish uniqueness, we first show that any non-negative mild solution admits a continuous version which is uniformly bounded on finite time intervals. We then prove a local Lipschitz property of the reaction term $F=(F_k)_{k\in\mathbb N_0}:\ell_1^+\to\ell_1$ defined in~\eqref{Paper02_reaction_term_PDE}. This allows us to apply a standard Gr\"onwall argument. Once uniqueness has been established, regularity properties follow from the smoothing effect of the heat semigroup together with classical Schauder estimates for linear parabolic equations.

The regularity properties of solutions allow us to express each coordinate~$u_k$ through a Feynman--Kac representation.
This representation is the key ingredient in the proof of the asymptotic behaviour results. In Section~\ref{Paper02_section_asymptotic_behaviour_PDE}, assuming that the reaction term is monostable in the sense of Definition~\ref{Paper02_assumption_monostable_condition}, we combine the Feynman--Kac representation with an induction argument to establish Theorem~\ref{Paper02_prop_control_proportions}. This theorem provides quantitative bounds on the ratios between the densities of individuals carrying different numbers of mutations and shows that the mutational profile remains well controlled during the range expansion.

To determine the spreading speed, we adapt an argument from~\cite{penington2018spreading} (which in turn is based on ideas in~\cite{bramson1983convergence}). In~\cite{penington2018spreading}, the spreading speed of a non-local Fisher--KPP equation is obtained using a Feynman--Kac representation together with Brownian motion estimates. The argument exploits the fact that regions of low population density correspond to high per-capita growth rates, while regions of high population density have already been colonised by the population.

In our setting, we analyse the dynamics of the mutation-free coordinate~$u_0$ through the total population density~$\|u\|_{\ell_1}$. The crucial observation is that Theorem~\ref{Paper02_prop_control_proportions} links these two quantities, i.e.~by Theorem~\ref{Paper02_prop_control_proportions}, there exists $C > 0$ such that for all~$T \in [0, \infty)$ and~$x \in \mathbb R$,
\[
u_0(T,x) \leq \|u(T,x)\|_{\ell_1} \leq C u_0(T,x).
\]
Hence, when~$\|u\|_{\ell_1}$ is small, the Fisher--KPP condition guarantees a large per-capita growth rate for~$u_0$. Conversely, when~$\|u\|_{\ell_1}$ is large, then~$u_0$ is uniformly bounded away from zero. Combining this dichotomy with the Brownian motion estimates of~\cite{penington2018spreading} yields Theorem~\ref{Paper02_spreading_speed_Fisher_KPP}, which determines the spreading speed of the population in the Fisher--KPP regime. Theorem~\ref{Paper02_thm_tracer_dynamics_no_gene_surfing}, i.e.~the characterisation of the asymptotic behaviour of the tracer, also follows from the Feynman--Kac representation and the control on the evolution of the mutational profile. Essentially, we show that the density of descendants of the subpopulation carrying mutations decays exponentially fast due to the competition with the fraction without mutations and its descendants. As a consequence, labelled mutants do not persist in the expanding wave, thereby ruling out gene surfing of deleterious mutations in the deterministic monostable regime.

\section{Partial differential equations in $\ell_{1}$} \label{Paper02_PDES_sequence_spaces}

In this section, we define the meaning of weak and mild $\ell_1$-valued solutions to partial differential equations, and state the existence result that we borrow from our companion article~\cite{madeira2025LLN}. Since solutions take values in~$\ell_1$, we first introduce the function spaces needed to formulate the PDE rigorously. Let $(S, \Sigma, \nu)$ be a measure space. We say a function $g = (g_k)_{k \in \mathbb N_0}: S \rightarrow \ell_{1}$ is (Bochner) measurable with respect to $(S, \Sigma, \nu)$ if $g_{k}: S \rightarrow \mathbb{R}$ is measurable for every $k \in \mathbb{N}_{0}$. Note that in this case, $\| g \|_{\ell_{1}}: S \rightarrow [0, \infty)$ is the pointwise limit of a sequence of measurable functions, i.e.~$\| g(x) \|_{\ell_{1}} = \lim_{n \rightarrow \infty} \sum_{k = 0}^{n} \vert g_{k}(x) \vert \; \forall \, x \in S$, and so $\| g \|_{\ell_{1}}$ is also measurable. Following \cite[Chapters 1 and 2]{hytonen2016analysis}, for $r \in [1, \infty]$, we let
\begin{equation*}
    L_r(S; \ell_{1}) \defeq \left\{g: S \rightarrow \ell_{1}: \; g \textrm{ is measurable and } \| g \|_{L_r(S;\ell_1)} \defeq \Big \| \, \| g \|_{\ell_{1}} \Big\|_{L_r(\nu)} < \infty\right\}.
\end{equation*}
We consider two functions to define the same element of~$L_r(S; \ell_{1})$ if they are equal $\nu$-almost everywhere in~$S$. Note that~$L_r(S; \ell_1)$ is a Banach space (see e.g.~\cite[Proposition~1.2.29]{hytonen2016analysis}). For $d \in \mathbb N$, for the Lebesgue measure space $(\mathbb{R}^{d}, \mathscr{R}^{(d)}, \lambda)$, for $S \in \mathscr{R}^{(d)}$, we say that a measurable function $g: S \rightarrow \ell_{1}$ is an element of $L_{1,\textrm{loc}}(S; \ell_{1})$ if for all compact sets $\mathcal{K} \subseteq S$, we have
\begin{equation*}
    \int_{\mathcal{K}}  \| g(x) \|_{\ell_{1}} \; dx < \infty.
\end{equation*}

We are now ready to properly define the meaning of a weak solution to the PDE~\eqref{Paper02_PDE_scaling_limit}. We will be thinking of weak solutions in the sense of distributions, rather than in the sense of elements of Sobolev spaces (see~\cite[Chapter~2]{hytonen2016analysis} for the difference between these concepts).

\begin{definition}[Weak solution] \label{Paper02_weak_solution_distributional_sense}
    We say that a measurable function $u \in L_{1, \textrm{loc}}([0, \infty) \times \mathbb{R}; \ell_{1})$ is a weak solution to the system of PDEs~\eqref{Paper02_PDE_scaling_limit} if the following conditions are satisfied:
    \begin{enumerate}[(i)]
        \item $u(0,x) = f(x)$ for $\lambda$-almost every $x \in \mathbb{R}$.
        \item $u(T, \cdot) \in L_{1,\textrm{loc}}(\mathbb{R}; \ell_{1})$ for all $T \geq 0$.
        \item $F(u): [0,\infty) \times \mathbb{R} \rightarrow \ell_{1}$ is an element of $L_{1, \textrm{loc}}([0,\infty) \times \mathbb{R}; \ell_{1})$, where $F = (F_{k})_{k \in \mathbb{N}_{0}}$ is the reaction term defined in~\eqref{Paper02_reaction_term_PDE}.
        \item For any $\varphi \in \mathscr{C}^{1,2}_c([0,\infty) \times \mathbb R; \mathbb R)$, i.e.~for any continuous function $\varphi: [0,\infty) \times \mathbb{R} \rightarrow \mathbb{R}$ with compact support, continuously differentiable in time and twice continuously differentiable in space, and any $T > 0$,
    \begin{equation} \label{Paper02_coordinate_wise_integral}
    \begin{aligned}
        \int_{\mathbb{R}} u(T,x) \varphi(T,x) \, dx = & \int_{\mathbb{R}} u(0,x) \varphi(0,x) \, dx + \int_{0}^{T} \int_{\mathbb{R}} u(t,x) \partial_{t} \varphi(t,x) \, dx \, dt \\ & + \int_{0}^{T} \int_{\mathbb{R}} \frac{m}{2} u(t,x) \Laplace \varphi(t,x) \, dx \, dt + \int_{0}^{T} \int_{\mathbb{R}} F(u(t,x)) \varphi(t,x) \, dx \, dt.
    \end{aligned}
    \end{equation}
    \end{enumerate}
\end{definition}

Note that~\eqref{Paper02_coordinate_wise_integral} holds if and only if it holds in a coordinate-wise manner, i.e.~if and only if for all $k \in \mathbb{N}_{0}$,
\begin{equation} \label{Paper02_coordinate_wise_weak_sol}
\begin{aligned}
     \int_{\mathbb{R}} u_{k}(T,x) \varphi(T,x) \, dx = & \int_{\mathbb{R}} u_{k}(0,x) \varphi(0,x) \, dx + \int_{0}^{T} \int_{\mathbb{R}} u_{k}(t,x) \partial_{t} \varphi(t,x) \, dx \, dt \\ & + \int_{0}^{T} \int_{\mathbb{R}} \frac{m}{2} u_{k}(t,x) \Laplace \varphi(t,x) \, dx \, dt + \int_{0}^{T} \int_{\mathbb{R}} F_{k}(u(t,x)) \varphi(t,x) \, dx \, dt.
\end{aligned}
\end{equation}
We refer the reader to~\cite[Chapter~1]{hytonen2016analysis} for a proof of this equivalence.

We will also need the definition of mild solutions to the system of PDEs~\eqref{Paper02_PDE_scaling_limit}. Let $\{P_{t}\}_{t \geq 0}$ denote the semigroup of Brownian motion run at speed $m$. The following notation will be useful: for $t > 0$ and $x \in \mathbb R$, we let
\begin{equation}
\label{Paper02_Gaussian_kernel}
p(t,x) \defeq \frac{1}{\sqrt{2\pi mt}} e^{-x^2/(2mt)}.
\end{equation}
Analogously to the definition of the action of $\{P_{t}\}_{t \geq 0}$ on real-valued functions, with a slight abuse of notation we can define its action on $\ell_{1}$-valued functions as follows. Let $v: \mathbb{R} \rightarrow \ell_{1}$ be a measurable function. For all $(t,x) \in [0, \infty) \times \mathbb R$ such that $(P_{t}\| v \|_{\ell_{1}}) (x) < \infty$,
we can define
\begin{equation}
\label{Paper02_semigroup_BM_action_ell_1_functions}
    \left(P_{t}v\right)(x) \defeq \left\{\begin{array}{lc}
        \displaystyle \int_{\mathbb{R}} p(t,x-y) v(y) \; dy & \textrm{ if } t > 0, \\
        v(x) & \textrm{ if } t = 0.
    \end{array}\right.
\end{equation}

\begin{definition}[Mild solution] \label{Paper02_definition_mild_solution}
    We say that a measurable function $u: [0,\infty) \times \mathbb{R} \rightarrow \ell_{1}$ is a (global) mild solution to the system of PDEs~\eqref{Paper02_PDE_scaling_limit} if and only if the following conditions are satisfied:
    \begin{enumerate}[(i)]
    \item For $\lambda$-almost every $(T,x) \in [0,\infty) \times \mathbb{R}$, $\left(P_{T-t}\| F(u(t,\cdot)) \|_{\ell_{1}}\right) (x) < \infty$ for $\lambda$-almost every $t \in [0,T]$.
    \item For $\lambda$-almost every $(T,x) \in [0,\infty) \times \mathbb{R}$,
    \begin{equation} \label{Paper02_mild_solution_weaker_version}
        u(T,x) = \left(P_{T} f\right)(x) + \int_{0}^{T} \Big(P_{T-t} F(u(t,\cdot))\Big)(x) \, dt.
    \end{equation}
    \end{enumerate}
\end{definition}

\begin{remark} \label{Paper02_remark_general_mild_sol}
    We collect here some observations regarding Definition~\ref{Paper02_definition_mild_solution}.
    \begin{enumerate}[(a)]
        \item Although our definition of a mild solution is compatible with the integral formulations used in the $L_p$-regularity literature (see e.g.~\cite[Chapter 5]{pruss2016moving}), classical semigroup approaches to deterministic PDEs typically require mild solutions to be continuous in time and to satisfy the variation-of-constants formula~\eqref{Paper02_mild_solution_weaker_version} pointwise in time (see e.g.~\cite[Section 4.1]{pazy2012semigroups}). In contrast, we work with a weaker formulation in which~\eqref{Paper02_mild_solution_weaker_version} holds almost everywhere in time and in space. This choice reflects the fact that we use the characterisation of solutions to the system of PDEs~\eqref{Paper02_PDE_scaling_limit} in our companion article~\cite{madeira2025LLN}, where time and space continuity of solutions is not available \textit{a priori} since they are constructed directly as the scaling limit of an interacting particle system.

        \item Let $T' > 0$ be fixed. If a measurable function $u: [0, T'] \times \mathbb R \rightarrow \ell_1$ satisfies Definition~\ref{Paper02_definition_mild_solution}(i) and~(ii) for $\lambda$-almost every $(T,x) \in [0, T'] \times \mathbb R$, we say $u$ is a (local) mild solution to the system of PDEs~\eqref{Paper02_PDE_scaling_limit}.
    \end{enumerate}
    
\end{remark}

Our next result shows that, under mild additional assumptions, a mild solution
is also a weak solution.

\begin{lemma} \label{Paper02_lem_equiv_mild_weak_sol}
Let $u: [0, \infty) \times \mathbb R \rightarrow \ell_1^+$ denote a measurable function. Suppose that $u$ is a mild solution to the system of PDEs~\eqref{Paper02_PDE_scaling_limit} in the sense of Definition~\ref{Paper02_definition_mild_solution}, and that it satisfies the following conditions:
\begin{enumerate}[(i)]
    \item $u(0,x) = f(x)$ for $\lambda$-almost every $x \in \mathbb R$.
    \item For all $T > 0$, $u \big|_{[0,T] \times \mathbb R} \in L_{\infty}([0,T] \times \mathbb R; \ell_1)$,
    where $ u \big|_{[0,T] \times \mathbb R}$ denotes the restriction of $u$ to $[0,T] \times \mathbb R$.
    \item For all $T > 0$, $u$ satisfies identity~\eqref{Paper02_mild_solution_weaker_version} for $\lambda$-almost every $x \in \mathbb R$.
\end{enumerate}
Then $u$ is a weak solution to the system of PDEs~\eqref{Paper02_PDE_scaling_limit} in the sense of Definition~\ref{Paper02_weak_solution_distributional_sense}.
\end{lemma}

Although the equivalence between mild and weak solutions is classical under higher regularity assumptions, we did not find a result in the literature at the level of generality required in this article. For this reason, and for the sake of completeness, we include the proof of Lemma~\ref{Paper02_lem_equiv_mild_weak_sol} in Section~\ref{Paper02_appendix_section_proof_auxiliary_no_technical_lemmas} of the appendix.

We now precisely introduce the function spaces where we will establish well-posedness of solutions to the system of PDEs~\eqref{Paper02_PDE_scaling_limit}. Recall that~$\lambda$ denotes the Lebesgue measure on~$\mathbb{R}^{2}$, and, for any $T \geq 0$, let $\hat{\lambda}$ denote the measure on $[0,T] \times \mathbb{R}$ given by
\begin{equation} \label{Paper02_measure_time_space_box_i}
    \hat{\lambda}(dt \; dx) \defeq \frac{\mathds{1}_{\{t \in [0,T]\}}}{1 + \vert x \vert^{2}} \lambda (dt \; dx).
\end{equation}
For~$T >0$ and~$r \in [1,\infty)$, we introduce the function space
\begin{equation} \label{Paper02_functional_space_L1_l1}
\begin{aligned}
    & {L}_{r}([0,T] \times \mathbb{R}, \hat{\lambda}; \ell_{1})
    \\ & \quad \defeq \Bigg\{v: [0,T] \times \mathbb{R} \rightarrow \ell_{1} \; \textrm{such that}\; v \; \textrm{is measurable and}
    \\[-3mm]
    & \hspace{3cm} \| v \|_{{L}_{r}([0,T] \times \mathbb{R}, \hat{\lambda}; \ell_{1})} \defeq \Bigg(\int_{0}^{T} \int_{\mathbb{R}} \frac{\| v(t,x) \|_{\ell_{1}}^{r}}{1 + \vert x \vert^{2}} \, dx \, dt \Bigg)^{1/r} < \infty\Bigg\}.
\end{aligned}
\end{equation}

\begin{remark}
    Observe that, by~\eqref{Paper02_functional_space_L1_l1}, the space~$L_{r}([0,T] \times \mathbb{R}, \hat{\lambda}; \ell_{1})$ differs from the space
    \begin{equation*}
    \begin{aligned}
        & {L}_{r}([0,T] \times \mathbb{R} \times \mathbb{N}_{0}, \hat{\lambda}; \mathbb{R}) 
        \\ & \quad \defeq \Bigg\{v: [0,T] \times \mathbb{R} \times \mathbb{N}_{0} \rightarrow \mathbb{R} \; \textrm{measurable such that} \;  \sum_{k \in \mathbb N_0} \, \int_{0}^{T} \int_{\mathbb{R}} \frac{\vert v_{k}(t,x) \vert^{r}}{1 + \vert x \vert^{2}} \, dx \, dt < \infty\Bigg\}.
    \end{aligned}
    \end{equation*}
    In particular, the advantage of embedding the sequence of density processes in~$L_r([0,T] \times \mathbb{R}, \hat{\lambda}; \ell_1)$ rather than in~${L}_{r}([0,T] \times \mathbb{R} \times \mathbb{N}_{0}, \hat{\lambda}; \mathbb{R})$ is because of the structure of the reaction term~$F = (F_{k})_{k \in \mathbb{N}_{0}}$ defined in~\eqref{Paper02_reaction_term_PDE}, since for $u \in \ell_1^+$ and $k \in \mathbb{N}_{0}$, the expression defining $F_k(u)$ contains polynomials of~$\| u \|_{\ell_{1}}$.
\end{remark}

Since~$\ell_{1}$ is a complete and separable Banach space, for every $r \in [1, \infty)$, $L_{r}([0,T] \times \mathbb{R}, \hat{\lambda}; \ell_{1})$ is also a complete and separable Banach space when equipped with the norm~$\| \cdot \|_{L_{r}([0,T] \times \mathbb{R}, \hat{\lambda}; \ell_{1})}$ defined in~\eqref{Paper02_functional_space_L1_l1} (see~\cite[Proposition~1.2.29]{hytonen2016analysis} for a proof of this fact).

Next, we state the existence of mild solutions (in the sense of Definition~\ref{Paper02_definition_mild_solution}) in $L_{4\deg q_-}([0,T] \times \mathbb R, \hat \lambda; \ell_1)$ for all $T > 0$. The result is an immediate consequence of~\cite[Proposition~6.2]{madeira2025LLN}, where we show that, under an appropriate scaling, 
along  a subsequence, the rescaled particle numbers in the corresponding interacting particle system
converge in distribution to a mild solution of~\eqref{Paper02_PDE_scaling_limit}. In order to precisely state the convergence result, we would need to define the interacting particle system, which requires a fair amount of notation. We refer the interested reader to~\cite{madeira2025existence} for the construction and uniqueness in law of the interacting particle system and to~\cite{madeira2025LLN} for the proof of~\cite[Proposition~6.2]{madeira2025LLN}. Since the aim of this article is to study the PDE system~\eqref{Paper02_PDE_scaling_limit} directly, we only state the existence of mild solutions result.
Note that although results from this article are used at the end of the proof of the law of large numbers in~\cite{madeira2025LLN}, no results from this article are used in the proof of~\cite[Proposition~6.2]{madeira2025LLN}.

\begin{proposition} \label{Paper02_absolutely_continuity_limiting_process}
    Suppose $m > 0$, $\mu \in [0,1]$, $(s_k)_{k \in \mathbb N_0}$ satisfies Assumption~\ref{Paper02_assumption_fitness_sequence}, $q_+,q_-: [0, \infty) \rightarrow [0, \infty)$ satisfy Assumption~\ref{Paper02_assumption_polynomials}, and $f \in L_{\infty}(\mathbb R; \ell_1)$ satisfies Assumption~\ref{Paper02_assumption_initial_condition}. Then, for all $T > 0$, there exists $v = (v_k)_{k \in \mathbb N_0} \in L_{4\deg q_-}([0, T] \times \mathbb R, \hat \lambda; \ell_1)$ such that~$v$ is a mild solution to the system of PDEs~\eqref{Paper02_PDE_scaling_limit} in the sense of Definition~\ref{Paper02_definition_mild_solution}. Moreover, for every $k \in \mathbb N_0$, $v_k(t,x) \geq 0$ for Lebesgue-almost every $(t,x) \in [0,T] \times \mathbb R$.
\end{proposition}

\section{Uniqueness and regularity of mild solutions} \label{Paper02_uniqueness_weak_solutions_section}

Existence of $\lambda$-almost everywhere non-negative mild solutions to the system of PDEs~\eqref{Paper02_PDE_scaling_limit} follows from Proposition~\ref{Paper02_absolutely_continuity_limiting_process}. In this section, we will establish uniqueness and regularity properties of mild solutions to the system of PDEs~\eqref{Paper02_PDE_scaling_limit}. For~$T > 0$ and~$r \in [1, \infty)$, recall the function space~$
{L}_{r}([0,T] \times \mathbb{R}, \hat \lambda; \ell_{1})$ defined in~\eqref{Paper02_functional_space_L1_l1}. For~$v^{(1)},v^{(2)} \in L_{r}([0,T] \times \mathbb{R}, \hat{\lambda}; \ell_{1})$, we call~$v^{(2)}$ a \emph{version} of~$v^{(1)}$ in~$L_{r}([0,T] \times \mathbb{R}, \hat{\lambda}; \ell_{1})$ if
\begin{equation} \label{Paper02_modification_version_L_p}
   v^{(1)}(t,x) = {v}^{(2)}(t,x) \textrm{ for } \lambda\textrm{-almost every } (t,x) \in [0,T] \times \mathbb{R},
\end{equation}
where $\lambda$ denotes the Lebesgue measure on~$\mathbb{R}^{2}$. Recall the definition of mild solutions to the system of PDEs~\eqref{Paper02_PDE_scaling_limit} in Definition~\ref{Paper02_definition_mild_solution}. Our first main result of this section concerns uniqueness of mild solutions to the system of PDEs~\eqref{Paper02_PDE_scaling_limit}.

\begin{proposition} \label{Paper02_thm_uniqueness_weak_solutions}
    Suppose $m > 0$, $\mu \in [0,1]$, $(s_k)_{k \in \mathbb N_0}$ satisfies Assumption~\ref{Paper02_assumption_fitness_sequence}, $q_+,q_-: [0, \infty) \rightarrow [0, \infty)$ satisfy Assumption~\ref{Paper02_assumption_polynomials}, and $f \in L_{\infty}(\mathbb R; \ell_1)$ satisfies Assumption~\ref{Paper02_assumption_initial_condition}. For~$T > 0$, let $v^{(1)}, v^{(2)} \in L_{4\deg q_-}([0,T] \times \mathbb{R}, \hat{\lambda}; \ell_{1})$ be $\lambda$-almost everywhere non-negative mild solutions to the system of PDEs~\eqref{Paper02_PDE_scaling_limit}. Then~$v^{(2)}$ is a version of~$v^{(1)}$ in~$L_{4\deg q_-}([0,T] \times \mathbb{R}, \hat{\lambda}; \ell_{1})$, i.e.~$v^{(1)}$ and $v^{(2)}$ satisfy~\eqref{Paper02_modification_version_L_p}.
\end{proposition}

We will also establish the following regularity property of mild solutions.

\begin{proposition} \label{Paper02_smoothness_mild_solution}
    Under the assumptions of Proposition~\ref{Paper02_thm_uniqueness_weak_solutions}, there exists a unique function $v = (v_k)_{k \in \mathbb N_0}: [0, \infty) \times \mathbb R \rightarrow \ell_1^+$ such that $v \in \mathscr C((0,\infty) \times \mathbb  R; \ell_1^+)$ and $v$ satisfies the following conditions:
    \begin{enumerate}[(i)]
       \item $v = (v_k)_{k \in \mathbb N_0}$ is a global non-negative mild solution to the system of PDEs~\eqref{Paper02_PDE_scaling_limit} in the sense of Definition~\ref{Paper02_definition_mild_solution}, and it satisfies~\eqref{Paper02_mild_solution_weaker_version} for all $(t,x) \in [0,\infty) \times \mathbb R$, i.e.~for any $k \in \mathbb N_0$ and~$(t,x) \in [0, \infty) \times \mathbb R$,
       \begin{equation} \label{Paper02_coordinate_wise_strong_duhamel}
           v_k(t,x) = (P_tf_k)(x) + \int_0^t \Big(P_{t-\tau} F_k(v(\tau, \cdot))\Big)(x) \, d\tau.
       \end{equation}
        \item For all $T > 0$,
        \[
        v \vert_{[0,T] \times \mathbb R} \in L_{4\deg q_-}([0,T] \times \mathbb R, \hat \lambda; \ell_1) \cap L_\infty([0,T] \times \mathbb R; \ell_1).
        \]
        \item For all $T > 0$, $\sup_{t \in [0,T]} \; \| {v}(t, \cdot) \|_{L_{\infty}(\mathbb{R}; \ell_{1})} < \infty$.
        \end{enumerate}
        Moreover, the unique function $v = (v_k)_{k \in \mathbb N_0}: [0, \infty) \times \mathbb R \rightarrow \ell_1^+$ satisfies the following conditions:
        \begin{enumerate}
        \item[(iv)] $v_k \in \mathscr{C}^{1,2}((0, \infty) \times \mathbb R; \mathbb R)$ for every $k \in \mathbb N_0$.
        \item[(v)] The map $(0, \infty) \times \mathbb R \ni (t,x) \mapsto \| v(t,x) \|_{\ell_1}$ is in $\mathscr C^{1,2}((0,\infty) \times \mathbb R; \mathbb R)$.
    \end{enumerate}
\end{proposition}

We are now ready to establish Theorem~\ref{Paper02_deterministic_scaling_foutel_rodier_etheridge}.

\begin{proof}[Proof of Theorem~\ref{Paper02_deterministic_scaling_foutel_rodier_etheridge}]
Existence of a global non-negative mild solution $u: [0, \infty) \times \mathbb R \rightarrow \ell_1^+$ to~\eqref{Paper02_PDE_scaling_limit} satisfying conditions~(i)--(iii) follows from Proposition~\ref{Paper02_smoothness_mild_solution}. In order to establish uniqueness, let~$u^{(1)},u^{(2)}: [0,\infty) \times \mathbb R \rightarrow \ell_1^+$ be global non-negative mild solutions to the system of PDEs~\eqref{Paper02_PDE_scaling_limit} satisfying~(i)--(iii). By~(ii), the maps~$u^{(1)}_k, u^{(2)}_k: [0,\infty) \times \mathbb R \rightarrow [0, \infty)$ are measurable with respect to the Lebesgue measure for every~$k \in \mathbb N_0$. Hence, as explained at the beginning of Section~\ref{Paper02_PDES_sequence_spaces}, the maps~$u^{(1)},u^{(2)}: [0,\infty) \times \mathbb R \rightarrow \ell_1^+$ are both (Bochner) measurable. Therefore, condition~(i) together with~\eqref{Paper02_functional_space_L1_l1} imply that, for all $T \in (0, \infty)$,
\begin{equation} \label{eq:measurability_candidate_solutions}
u^{(1)}\vert_{[0,T] \times \mathbb R},u^{(2)}\vert_{[0,T] \times \mathbb R} \in L_{4\deg q_-}([0,T] \times \mathbb R, \hat \lambda;\ell_1).
\end{equation}
Since~\eqref{eq:measurability_candidate_solutions} holds for all $T \in (0, \infty)$, it follows from Proposition~\ref{Paper02_thm_uniqueness_weak_solutions} that
\begin{equation} \label{eq:lebesgue_ae_uniqueness_final}
u^{(1)}(t,x) = {u}^{(2)}(t,x) \textrm{ for } \lambda\textrm{-almost every } (t,x) \in [0,\infty) \times \mathbb{R}.
\end{equation}
Finally, combining condition~(ii) with~\eqref{eq:lebesgue_ae_uniqueness_final}, we conclude that $u^{(1)}(t,x) = u^{(2)}(t,x)$ for all $t \in (0, \infty)$ and~$x \in \mathbb R$, which completes the proof.
\end{proof}

In the remainder of this section, we will prove Propositions~\ref{Paper02_thm_uniqueness_weak_solutions} and~\ref{Paper02_smoothness_mild_solution}.

\subsection{Proof of Proposition~\ref{Paper02_thm_uniqueness_weak_solutions}} \label{Paper02_uniqueness_mild_solutions_subsection}

Our first step towards the proof of Proposition~\ref{Paper02_thm_uniqueness_weak_solutions} will be to prove that any non-negative mild solution to the system of PDEs~\eqref{Paper02_PDE_scaling_limit} must be essentially bounded. Although we do not follow any particular reference, similar strategies to establish regularity properties of weak or mild solutions to partial differential equations are common in the PDE literature (e.g.~\cite[Proof of Theorem~1]{brezis1979uniqueness} and~\cite[Section~4.2]{pazy2012semigroups}).

\begin{lemma} \label{Paper02_uniform_bound_limiting_process_L_infty_norm}
    Suppose the assumptions of Proposition~\ref{Paper02_thm_uniqueness_weak_solutions} hold. For $T > 0$, let $v \in L_{4\deg q_-}([0,T] \times \mathbb{R}, \hat{\lambda}; \ell_{1})$ be a $\lambda$-almost everywhere~non-negative mild solution to the system of PDEs~\eqref{Paper02_PDE_scaling_limit}. Then $v \in L_{\infty}([0,T] \times \mathbb{R}; \ell_{1})$. Moreover, there exists a version~$\tilde{v} \in L_{4\deg q_-}([0,T] \times \mathbb{R}, \hat{\lambda}; \ell_{1})$ of~$v$ satisfying the following conditions:
    \begin{enumerate}[(i)]
        \item For every~$k \in \mathbb{N}_{0}$ and any~$(t,x) \in [0,T] \times \mathbb{R}$,
        \begin{equation} \label{Paper02_stronger_formulation_mild_solution}
        \tilde{v}_{k}(t,x) = (P_{t}f_{k})(x) + \int_{0}^{t} \Big(P_{t - \tau}F_{k}(\tilde{v}(\tau, \cdot))\Big)(x) \, d\tau.
    \end{equation} 
        \item $\tilde{v} \in \mathscr{C}((0,T] \times \mathbb{R}; \ell_{1}^{+})$.
        \item $\sup_{t \in [0,T]} \; \| \tilde{v}(t, \cdot) \|_{L_{\infty}(\mathbb{R}; \ell_{1})} < \infty.$
    \end{enumerate}
\end{lemma}

\begin{proof}
  By the definition of the reaction term~$F = (F_{k})_{k \in \mathbb{N}_{0}}$ in~\eqref{Paper02_reaction_term_PDE}, the fact that $s_k \leq 1$ for all~$k \in \mathbb N_0$ by Assumption~\ref{Paper02_assumption_fitness_sequence}(i) and~(iii), and using that $q_+$ and $q_-$ are non-negative by Assumption~\ref{Paper02_assumption_polynomials}, it follows that for all~$z \in \ell_{1}^{+}$,
    \begin{equation}
    \label{Paper02_trivial_ell_1_bound_reaction_term}
        \| F(z) \|_{\ell_{1}} \leq \| z \|_{\ell_{1}}(q_{+}(\| z \|_{\ell_{1}}) + q_{-}(\| z \|_{\ell_{1}})), 
    \end{equation}
    and that for all $J \in \mathbb N_0$,
    \begin{equation}
    \label{Paper02_trivial_ell_1_bound_reaction_term_ii}
        \sum_{k = 0}^J F_k(z)  \leq \sum_{k = 0}^J z_k (q_{+}(\| z \|_{\ell_{1}}) - q_{-}(\| z \|_{\ell_{1}})) \leq \sup_{U \geq 0} \Big(U(q_+(U) - q_-(U))\Big).
    \end{equation}
    
    Since~$v$ is a  mild solution to the system of PDEs~\eqref{Paper02_PDE_scaling_limit}, by Definition~\ref{Paper02_definition_mild_solution}(ii) we have that for $\lambda$-almost every $(t,x) \in [0,T] \times \mathbb R$,
      \[
      v_k(t,x) = (P_tf_k)(x) + \int_0^t \Big(P_{t-\tau} F_k(v(\tau, \cdot))\Big)(x) \, d\tau \quad \forall \, k \in \mathbb N_0,
      \]
      where $\{P_t\}_{t \geq 0}$ is the semigroup of Brownian motion run at speed $m$ defined before~\eqref{Paper02_Gaussian_kernel}. Hence, by using Assumption~\ref{Paper02_assumption_initial_condition}(ii) and~\eqref{Paper02_trivial_ell_1_bound_reaction_term_ii}, we conclude that for $\lambda$-almost every $(t,x) \in [0, T] \times \mathbb R$, the following estimate holds for every $J \in \mathbb N_0$:
      \begin{equation} \label{Paper02_trivial_ell_1_bound_reaction_term_iii}
      \begin{aligned}
          \sum_{k = 0}^J v_k(t,x) & = \left(P_t \sum_{k = 0}^J f_k\right)(x) + \int_0^t \left(P_{t-\tau} \sum_{k = 0}^J F_k(v(\tau, \cdot))\right)(x) \, d\tau\\ & \leq  \| f \|_{L_{\infty}(\mathbb{R}; \ell_{1})} + T \sup_{U \geq 0} \Big(U(q_{+}(U) - q_{-}(U))\Big).
      \end{aligned}
      \end{equation}
      Moreover, since $v$ is $\lambda$-almost everywhere non-negative, we have that for $\lambda$-almost every $(t,x) \in [0,T] \times \mathbb R$, $ \lim_{J \rightarrow \infty} \; \sum_{k = 0}^J v_k(t,x) = \| v(t,x) \|_{\ell_1}$,
      and therefore, by taking the limit as $J \rightarrow \infty$ on the left-hand side of~\eqref{Paper02_trivial_ell_1_bound_reaction_term_iii}, we conclude that for $\lambda$-almost every $(t,x) \in [0,T] \times \mathbb{R}$,
    \begin{equation} \label{Paper02_step_i_uniform_bound_mild_solution}
    \begin{aligned}
        \| {v} (t,x) \|_{\ell_{1}} \leq \| f \|_{L_{\infty}(\mathbb{R}; \ell_{1})} + T \sup_{U \geq 0} \Big( U(q_{+}(U) - q_{-}(U))\Big).
    \end{aligned}
    \end{equation}
    Since, by Assumption~\ref{Paper02_assumption_polynomials}, $q_+$ and $q_-$ are non-negative polynomials with $0 \leq \deg q_{+} < \deg q_{-}$, the term on the right-hand side of~\eqref{Paper02_step_i_uniform_bound_mild_solution} is finite, and therefore~$v \in L_{\infty}([0,T] \times \mathbb{R}; \ell_{1})$.

    Now, let $\tilde{v} = (\tilde{v}_k)_{k \in \mathbb{N}_0}: [0,T] \times \mathbb{R} \rightarrow \ell_1$ be given by, for every $k \in \mathbb{N}_0$ and any $(t,x) \in [0,T]\times \mathbb R$, 
    \begin{equation} \label{Paper02:trivial_construction_mild_solution_not_so_trivial}
        \tilde{v}_k(t,x) = (P_{t}f_{k})(x) + \int_{0}^{t} \Big(P_{t - \tau}F_{k}({v}(\tau, \cdot))\Big)(x) \, d\tau.
    \end{equation}
    Then, by Definition~\ref{Paper02_definition_mild_solution}(ii), $\tilde{v}$ is a version of $v$, and so $\tilde v = (\tilde v_k)_{k \in \mathbb N_0}$ is equal to $v = (v_k)_{k \in \mathbb N_0}$ $\lambda$-almost everywhere in $[0,T] \times \mathbb R$. Hence, identity~\eqref{Paper02:trivial_construction_mild_solution_not_so_trivial} implies that for all $k \in \mathbb N_0$ and $(t,x) \in [0,T] \times \mathbb R$,
    \begin{equation*}
    \begin{aligned}
        \tilde{v}_k(t,x) & = (P_tf_k)(x) + \int_0^t \int_\mathbb R p(t-\tau, x-y) F_k(v(\tau,y)) \, dy \, d\tau \\ & = (P_tf_k)(x) + \int_0^t \int_\mathbb R p(t-\tau, x-y) F_k(\tilde v(\tau,y)) \, dy \, d\tau,
    \end{aligned}
    \end{equation*}
    where $p$ is the Gaussian kernel function defined in~\eqref{Paper02_Gaussian_kernel}. Hence, we conclude that $\tilde v = (\tilde v_k)_{k \in \mathbb N_0}$ satisfies~\eqref{Paper02_stronger_formulation_mild_solution} for all $k \in \mathbb N_0$ and $(t,x) \in [0, T] \times \mathbb R$. Moreover, since $v \in L_{\infty}([0,T] \times \mathbb R; \ell_1)$ and $\tilde v$ is a version of $v$, we have that $\tilde v \in L_{\infty}([0,T] \times \mathbb R; \ell_1)$. Hence, by~\eqref{Paper02_trivial_ell_1_bound_reaction_term}, the map
    \begin{equation} \label{Paper02_F(v)_in_L_infty}
    [0,T] \times \mathbb R \ni (t,x) \mapsto F(\tilde v(t,x)) \textrm{ is in } L_{\infty}([0,T] \times \mathbb R; \ell_1).
    \end{equation}
    Combining~\eqref{Paper02_F(v)_in_L_infty} and~\eqref{Paper02_stronger_formulation_mild_solution}, and then applying the smoothing properties of the Gaussian kernel (see Lemma~\ref{Paper02lem:heat_continuity_ell1} in the appendix), it follows that $\tilde v = (\tilde v_k)_{k \in \mathbb N_0} \in \mathscr C((0,T] \times \mathbb R; \ell_1)$. By continuity and the fact that $\tilde v = (\tilde v_k)_{k \in \mathbb N_0}$ is almost everywhere non-negative, we conclude that $\tilde{v}_k(t,x) \geq 0 \; \forall \, (k,t,x) \in \mathbb N_0 \times (0,T] \times \mathbb R$, i.e.~that condition~(ii) of the lemma holds.
    
    Moreover, condition~(iii) of the lemma follows from the fact that $\tilde v \in L_{\infty}([0,T] \times \mathbb R; \ell_1)$ and that $\tilde v \in \mathscr C((0,T] \times \mathbb  R; \ell_1)$, together with the observation that by~\eqref{Paper02_stronger_formulation_mild_solution}, $\tilde v(0, \cdot) = f(\cdot)$, and that $f \in L_{\infty}(\mathbb R; \ell_1)$ by Assumption~\ref{Paper02_assumption_initial_condition}(ii). This completes the proof.
\end{proof}

Next, we establish a local Lipschitz condition for the reaction term $F = (F_{k})_{k \in \mathbb{N}_{0}}$, which, together with Lemma~\ref{Paper02_uniform_bound_limiting_process_L_infty_norm}, will enable us to apply Gr\"{o}nwall's inequality in the proof of Proposition~\ref{Paper02_thm_uniqueness_weak_solutions}.

\begin{lemma} \label{Paper02_lipschitz_continuity_reaction_term_PDE}
    Suppose the assumptions of Proposition~\ref{Paper02_thm_uniqueness_weak_solutions} hold. Then the reaction term $F = (F_{k})_{k \in \mathbb{N}_{0}}: \ell_{1} \rightarrow \ell_{1}$ defined in~\eqref{Paper02_reaction_term_PDE} is locally Lipschitz continuous on~$\ell_{1}^{+}$, i.e.~$F = (F_{k})_{k \in \mathbb{N
    }_{0}}$ is  Lipschitz continuous on bounded subsets of $\ell_{1}^{+}$.
\end{lemma}

\begin{proof}
    Take $A > 0$, and then take $z, \tilde{z} \in \ell_{1}^{+}$ such that $\| z \|_{\ell_{1}} \vee \| \tilde{z} \|_{\ell_{1}} \leq A$. Recalling the definition of $F = (F_{k})_{k \in \mathbb{N}_{0}}$ in~\eqref{Paper02_reaction_term_PDE}, and then using the triangle inequality and the fact that by Assumption~\ref{Paper02_assumption_fitness_sequence}, $s_k \leq 1 \; \forall \, k \in \mathbb N_0$, we conclude that
    \begin{equation} \label{Paper02_intermediate_step_local_lipschitz_condition}
    \begin{aligned}
        & \| F(z) - F(\tilde{z}) \|_{\ell_{1}} \\ & \quad \leq \| z - \tilde{z} \|_{\ell_{1}} \Big(q_{+}(\| z \|_{\ell_{1}}) + q_{-}(\| z \|_{\ell_{1}})\Big) \\ & \quad \quad \quad + \| \tilde{z} \|_{\ell_{1}}\Big\vert q_{+}(\| z \|_{\ell_{1}}) - q_{+}(\| \tilde{z} \|_{\ell_{1}})\Big\vert +  \| \tilde{z} \|_{\ell_{1}}\Big\vert q_{-}(\| z \|_{\ell_{1}}) - q_{-}(\| \tilde{z} \|_{\ell_{1}})\Big\vert \\ & \quad \lesssim_{A} \| z - \tilde{z} \|_{\ell_{1}} + \Big\vert q_{+}(\| z \|_{\ell_{1}}) - q_{+}(\| \tilde{z} \|_{\ell_{1}})\Big\vert + \Big\vert q_{-}(\| z \|_{\ell_{1}}) - q_{-}(\| \tilde{z} \|_{\ell_{1}})\Big\vert,
    \end{aligned}
    \end{equation}
    where for the last inequality we used the fact that we chose $z,\tilde z$ such that $\| z \|_{\ell_{1}} \vee \| \tilde{z} \|_{\ell_{1}} \leq A$. Observe that by the triangle inequality and the fact that $q_{+}$ and $q_{-}$ are polynomials, there exist polynomials $Q_{+}, Q_{-}: [0, \infty) \times [0, \infty) \rightarrow [0, \infty)$ with non-negative coefficients such that for any $z, \tilde{z} \in \ell_{1}^{+}$,
    \begin{equation} \label{Paper02_intermediate_step_local_lipschitz_condition_i}
    \begin{aligned}
        & \Big\vert q_{+}(\| z \|_{\ell_{1}}) - q_{+}(\| \tilde{z} \|_{\ell_{1}})\Big\vert \leq \| z - \tilde{z} \|_{\ell_{1}}Q_{+}\left(\| z \|_{\ell_{1}}, \| \tilde{z} \|_{\ell_{1}}\right), \\ \textrm{ and } & \Big\vert q_{-}(\| z \|_{\ell_{1}}) - q_{-}(\| \tilde{z} \|_{\ell_{1}})\Big\vert \leq \| z - \tilde{z} \|_{\ell_{1}}Q_{-}\left(\| z \|_{\ell_{1}}, \| \tilde{z} \|_{\ell_{1}}\right).
    \end{aligned}
    \end{equation}
    Hence, applying~\eqref{Paper02_intermediate_step_local_lipschitz_condition_i} to~\eqref{Paper02_intermediate_step_local_lipschitz_condition}, and then recalling that $\| z \|_{\ell_{1}} \vee \| \tilde{z} \|_{\ell_{1}} \leq A$, it follows that for any $A > 0$ and any $z, \tilde{z} \in \ell_{1}^{+}$ such that $\| z \|_{\ell_{1}} \vee \| \tilde{z} \|_{\ell_{1}} \leq A$, we have
    \begin{equation*}
         \| F(z) - F(\tilde{z}) \|_{\ell_{1}} \lesssim_{A} \| z - \tilde{z} \|_{\ell_{1}},
    \end{equation*}
    which completes the proof.
\end{proof}

We are now ready to prove Proposition~\ref{Paper02_thm_uniqueness_weak_solutions}.

\begin{proof}[Proof of Proposition~\ref{Paper02_thm_uniqueness_weak_solutions}]
  By Lemma~\ref{Paper02_uniform_bound_limiting_process_L_infty_norm}, we can assume without loss of generality that~$v^{(1)}$ and~$v^{(2)}$ satisfy conditions~(i)-(iii) of Lemma~\ref{Paper02_uniform_bound_limiting_process_L_infty_norm}. Also by Lemma~\ref{Paper02_uniform_bound_limiting_process_L_infty_norm}, we can take $A > 0$ sufficiently large that
  \begin{equation} \label{Paper02_bound_total_local_mass_finite_time_allowed_PDE}
      \| v^{(1)} \|_{L_{\infty}([0,T] \times \mathbb{R}; \ell_{1})} \vee \| v^{(2)} \|_{L_{\infty}([0,T] \times \mathbb{R}; \ell_{1})} \leq A.
  \end{equation}
  By Lemma~\ref{Paper02_uniform_bound_limiting_process_L_infty_norm}(i) and the triangle inequality, we have that for all~$k \in \mathbb{N}_{0}$, $t \in [0,T]$ and~$x \in \mathbb{R}$,
  \begin{equation} \label{Paper02_proof_uniqueness_intermediate_step_ii}
  \begin{aligned}
      \vert v^{(1)}_{k}(t,x) - v^{(2)}_{k}(t,x)\vert \leq \int_{0}^{t} \Big(P_{t - \tau} \vert F_{k}({v}^{(1)}(\tau, \cdot)) - F_{k}({v}^{(2)}(\tau, \cdot)) \vert\Big)(x) \, d\tau.
  \end{aligned}
  \end{equation}
  By summing both sides of~\eqref{Paper02_proof_uniqueness_intermediate_step_ii} over~$k \in \mathbb{N}_{0}$, and then using~\eqref{Paper02_bound_total_local_mass_finite_time_allowed_PDE} and the fact that by Lemma~\ref{Paper02_lipschitz_continuity_reaction_term_PDE}, the reaction term~$F = (F_{k})_{k \in \mathbb{N}_{0}}$ is locally Lipschitz, it follows that for every~$(t,x) \in [0,T] \times \mathbb{R}$,
  \begin{equation*}
        \| v^{(1)}(t,x) - v^{(2)}(t,x)\|_{\ell_{1}} \lesssim_{A} \int_{0}^{t} \sup_{\theta \leq \tau}  \| v^{(1)}(\theta,\cdot) - v^{(2)}(\theta,\cdot)\|_{L_{\infty}(\mathbb{R};\ell_{1})} \, d\tau.
  \end{equation*}
  Hence, for every~$t \in [0,T]$,
  \begin{equation} \label{Paper02_proof_uniqueness_intermediate_step_iii}
        \sup_{\theta \leq t} \| v^{(1)}(\theta,\cdot) - v^{(2)}(\theta,\cdot)\|_{L_{\infty}(\mathbb{R};\ell_{1})} \lesssim_{A} \int_{0}^{t} \sup_{\theta \leq \tau}  \| v^{(1)}(\theta,\cdot) - v^{(2)}(\theta,\cdot)\|_{L_{\infty}(\mathbb{R};\ell_{1})} \, d\tau.
  \end{equation}
  Applying Lemma~\ref{Paper02_uniform_bound_limiting_process_L_infty_norm}(iii) and Gr\"{o}nwall's lemma to~\eqref{Paper02_proof_uniqueness_intermediate_step_iii}, we conclude that
  \begin{equation*}
      \sup_{\theta \leq T} \| v^{(1)}(\theta,\cdot) - v^{(2)}(\theta,\cdot)\|_{L_{\infty}(\mathbb{R};\ell_{1})} = 0.
  \end{equation*}
  Therefore, $v^{(2)}$ is a version of~$v^{(1)}$, and the proof is complete.
\end{proof}

\subsection{Proof of Proposition~\ref{Paper02_smoothness_mild_solution}} \label{Paper02_section_regularity_solutions}

Our first step towards the proof of Proposition~\ref{Paper02_smoothness_mild_solution} will be to derive a weak formulation of H\"{o}lder continuity for a mild solution $v \in L_{4\deg q_-}([0,T] \times \mathbb{R}, \hat{\lambda}; \ell_{1})$ to the system of PDEs~\eqref{Paper02_PDE_scaling_limit}.

\begin{lemma} \label{Paper02_uniform_Holder_continuity}
    Suppose the assumptions of Proposition~\ref{Paper02_thm_uniqueness_weak_solutions} hold. Let $T > 0$ and~$v \in L_{4\deg q_-}([0,T] \times \mathbb{R}, \hat{\lambda}; \ell_{1})$ be a version of the non-negative mild solution to the system of PDEs~\eqref{Paper02_PDE_scaling_limit} satisfying Lemma~\ref{Paper02_uniform_bound_limiting_process_L_infty_norm}(i)-(iii). For any~$\mathcal{I} \subseteq \mathbb{N}_{0}$, let
    \[
    V_{\mathcal{I}}(t,x) \defeq \sum_{k \in \mathcal{I}} v_{k}(t,x) \quad \forall \, (t,x) \in [0,T] \times \mathbb R.
    \]
    Then, for all~$p \in \mathbb{N}_{0}$ and~$\delta \in (0,T)$, there exists~$C_{\delta,p,T} > 0$ such that for any~$\mathcal{I} \subseteq \mathbb{N}_{0}$, any~$t_{1}, t_{2} \in [\delta,T]$ and any~$x_{1},x_{2} \in \mathbb{R}$,
    \begin{equation*}
        \Big\vert V_{\mathcal{I}}(t_{1},x_{1}) \| v(t_{1},x_{1}) \|_{\ell_1}^{p} - V_{\mathcal{I}}(t_{2},x_{2}) \| v(t_{2},x_{2}) \|_{\ell_1}^{p} \Big\vert \leq C_{\delta,p,T} \Big(\vert x_1 - x_2 \vert + \vert t_1 - t_2 \vert^{1/2} \Big).
    \end{equation*}
\end{lemma}

\begin{proof}
    First observe that from Lemma~\ref{Paper02_uniform_bound_limiting_process_L_infty_norm}(ii) and~(iii), we have that
    \begin{equation} \label{Paper02_auxiliary_holder_i}
    \sup_{(t,x) \in [\delta,T] \times \mathbb R} \| v(t,x) \|_{\ell_1} < \infty.
    \end{equation}
    Hence, applying the triangle inequality,  then the elementary inequality
    \[
    \vert a^{p} - b^{p} \vert \leq p (a \vee b)^{p-1} \vert a - b\vert \leq p (a + b)^{p-1} \vert a - b\vert \quad \forall \, p \in \mathbb N,
    \]
    and finally~\eqref{Paper02_auxiliary_holder_i}, we conclude that for any $\mathcal I \subseteq \mathbb N_0$, $t_1,t_2 \in [\delta,T]$ and~$x_1,x_2 \in \mathbb R$,
    \begin{equation} \label{Paper02_Holder_continuity_step_i}
    \begin{aligned}
        & \Big\vert V_{\mathcal{I}}(t_{1},x_{1}) \| v(t_{1},x_{1}) \|^{p}_{\ell_1} - V_{\mathcal{I}}(t_{2},x_{2}) \| v(t_{2},x_{2}) \|^{p}_{\ell_1} \Big\vert \\ & \quad \leq \Big\vert V_{\mathcal{I}}(t_{1},x_{1}) - V_{\mathcal{I}}(t_{2},x_{2}) \Big\vert \| v(t_{1},x_{1}) \|^{p}_{\ell_1} + V_{\mathcal{I}}(t_{2},x_{2})\Big\vert \| v(t_{1},x_{1})\|^{p}_{\ell_1} - \| v(t_{2},x_{2})\|^{p}_{\ell_1} \Big\vert \\ & \quad \lesssim_{p,T} \Big\vert V_{\mathcal{I}}(t_{1},x_{1}) - V_{\mathcal{I}}(t_{2},x_{2}) \Big\vert + \Big\vert \| v(t_{1},x_{1}) \|_{\ell_1} - \| v(t_{2},x_{2}) \|_{\ell_1}\Big\vert \\
        & \quad \leq 2 \| v(t_1,x_1) - v(t_2,x_2) \|_{\ell_1},
    \end{aligned}
    \end{equation}
    where the last inequality follows from the triangle inequality. By Lemma~\ref{Paper02_uniform_bound_limiting_process_L_infty_norm}, $v \in L_{\infty}([0,T] \times \mathbb R; \ell_1)$, and therefore~\eqref{Paper02_trivial_ell_1_bound_reaction_term} implies that the map $[0,T] \times \mathbb R \ni (t,x) \mapsto F(v(t,x))$
    is in $L_{\infty}([0,T] \times \mathbb R; \ell_1)$, i.e.~that the source term in~\eqref{Paper02_stronger_formulation_mild_solution} is essentially bounded. Hence, by Assumption~\ref{Paper02_assumption_initial_condition}(ii) and Lemma~\ref{Paper02_uniform_bound_limiting_process_L_infty_norm}(i), we can apply estimate~\eqref{Paper02_pre_step_holder} from Lemma~\ref{Paper02lem:heat_continuity_ell1} in the appendix to bound the term on the right-hand side of~\eqref{Paper02_Holder_continuity_step_i}, which completes the proof.
\end{proof}

We are now ready to prove Proposition~\ref{Paper02_smoothness_mild_solution}.

\begin{proof}[Proof of Proposition~\ref{Paper02_smoothness_mild_solution}]
    We will first establish the existence of a unique function $v = (v_k)_{k \in \mathbb N_0} \in \mathscr C((0, \infty) \times \mathbb R; \ell_1)$ which is a global mild solution to the system of PDEs~\eqref{Paper02_PDE_scaling_limit}. For $T_2 > T_1 > 0$ and functions $v^{(1)} \in L_{4\deg q_-}([0,T_1] \times \mathbb R, \hat \lambda; \ell_1)$ and $v^{(2)} \in L_{4\deg q_-}([0,T_2] \times \mathbb R, \hat \lambda; \ell_1)$ which are both local mild solutions to the system of PDEs~\eqref{Paper02_PDE_scaling_limit} in the sense of Definition~\ref{Paper02_definition_mild_solution}, Proposition~\ref{Paper02_thm_uniqueness_weak_solutions} implies that $v^{(1)}(t,x) = v^{(2)}(t,x)$ for $\lambda$-almost every $(t,x) \in [0,T_1] \times \mathbb R$. Since for any $T > 0$, Proposition~\ref{Paper02_absolutely_continuity_limiting_process} implies the existence of a local mild solution to the system of PDEs~\eqref{Paper02_PDE_scaling_limit} up to time $T$, Proposition~\ref{Paper02_thm_uniqueness_weak_solutions} and Lemma~\ref{Paper02_uniform_bound_limiting_process_L_infty_norm} imply the existence of a unique function $v = (v_k)_{k \in \mathbb N_0} : [0, \infty) \times \mathbb R \rightarrow \ell_1^+$ satisfying conditions (i)-(iii) of Proposition~\ref{Paper02_smoothness_mild_solution}.
    
    To complete the proof, it will suffice then to establish that $v = (v_k)_{k \in \mathbb N_0}$ satisfies conditions~(iv) and~(v). Take $\delta > 0$. By the semigroup property of the heat kernel~$\{P_{t}\}_{t \geq 0}$ defined in~\eqref{Paper02_semigroup_BM_action_ell_1_functions} and by condition~(i) of this proposition, observe that
    for all~$t > \delta$ and~$x \in \mathbb{R}$, we have
    \begin{equation} \label{Paper02_aux_global_mild_holder_i}
        v(t,x) = (P_{t - \delta}v(\delta, \cdot))(x) + \int_{\delta}^{t} \Big(P_{t - \tau}F(v(\tau, \cdot))\Big)(x) \, d\tau.
    \end{equation}
    By the definition of the reaction term $F = (F_{k})_{k \in \mathbb{N}_{0}}$ in~\eqref{Paper02_reaction_term_PDE} and the fact that $q_{+},q_{-}: [0,\infty) \rightarrow [0,\infty)$ are polynomials, combined with Lemma~\ref{Paper02_uniform_Holder_continuity}, it follows that for every~$k \in \mathbb{N}_{0}$, the map
    \begin{equation*}
        [\delta, \infty) \times \mathbb R \ni (t,x) \mapsto F_{k}(v(t,x))
    \end{equation*}
    is locally H\"{o}lder continuous. Standard theory for linear parabolic PDEs (see e.g.~\cite[Theorem~8.10.1]{krylov1996lectures}) then implies that $v_{k} \in \mathscr{C}^{1,2}((\delta, \infty) \times \mathbb{R}; \mathbb{R})$, for all~$\delta > 0$. It follows that $v_{k} \in \mathscr{C}^{1,2}((0, \infty) \times \mathbb{R}; \mathbb{R})$, for every~$k \in \mathbb{N}_{0}$, i.e.~that $v = (v_k)_{k \in \mathbb N_0}$ satisfies condition~(iv) of this proposition. Finally, by~\eqref{Paper02_F(v)_in_L_infty} stated in the proof of Lemma~\ref{Paper02_uniform_bound_limiting_process_L_infty_norm}, we can combine identity~\eqref{Paper02_aux_global_mild_holder_i} and Fubini's theorem to conclude that for $\delta > 0$, for all~$t > \delta$ and~$x \in \mathbb{R}$,
    \begin{equation} \label{Paper02_F(v)_in_L_infty_aux_func_g_ii}
        \| v(t,x)  \|_{\ell_1} = (P_{t - \delta} \| v(\delta, \cdot)\|_{\ell_1})(x) + \int_{\delta}^{t} \Big(P_{t - \tau}g(\tau, \cdot)\Big)(x) \, d\tau,
    \end{equation}
    where the map $g: [0, \infty) \times \mathbb R \rightarrow \mathbb R$ is given by
    \begin{equation} \label{Paper02_F(v)_in_L_infty_aux_func_g}
        g(t,x) \defeq \sum_{k = 0}^{\infty} F_k(v(t,x)) \quad \forall \, (t,x) \in [0, \infty) \times \mathbb R.
    \end{equation}
    Then, by applying the triangle inequality to~\eqref{Paper02_F(v)_in_L_infty_aux_func_g}, and then using Lemma~\ref{Paper02_uniform_bound_limiting_process_L_infty_norm}(iii) and the fact that by Lemma~\ref{Paper02_lipschitz_continuity_reaction_term_PDE}, the reaction term $F = (F_k)_{k \in \mathbb N_0}: \ell_1^+ \rightarrow \ell_1$ is locally Lipschitz continuous, we have that for all $T > 0$, $\delta \in (0,T)$, $t_1,t_2 \in [\delta,T]$ and $x_1,x_2 \in \mathbb R$,
    \begin{equation*}
    \begin{aligned}
        \vert g(t_1,x_1) - g(t_2,x_2) \vert & \leq  \| F(v(t_1,x_1)) - F(v(t_2,x_2)) \|_{\ell_1} \\
        & \lesssim_{T} \| v(t_1,x_1) - v(t_2,x_2) \|_{\ell_1}
        \\ & \lesssim_{\delta, T} \vert x_1 - x_2 \vert + \vert t_1 - t_2 \vert^{1/2},
    \end{aligned}
    \end{equation*}
    where for the last estimate we used estimate~\eqref{Paper02_pre_step_holder} from Lemma~\ref{Paper02lem:heat_continuity_ell1} in the appendix (by the same argument as at the end of the proof of Lemma~\ref{Paper02_uniform_Holder_continuity}). Hence, the map $g: [0, \infty) \times \mathbb R \rightarrow \mathbb R$ defined in~\eqref{Paper02_F(v)_in_L_infty_aux_func_g} is locally H\"older continuous on $(0, \infty) \times \mathbb R$. In particular,~\eqref{Paper02_F(v)_in_L_infty_aux_func_g_ii} shows that the scalar function \( V_{\mathbb N_0}(t,x) \defeq \|v(t,x)\|_{\ell_1}\) is a mild solution to a linear parabolic problem with locally H\"older continuous forcing term \(g\) on $(\delta,\infty)\times\mathbb R $, for any $\delta > 0$. Therefore, again applying standard regularity properties of linear PDEs (see~\cite[Theorem~8.10.1]{krylov1996lectures}) to~\eqref{Paper02_F(v)_in_L_infty_aux_func_g_ii}, we conclude that condition~(v) of this proposition holds, which completes the proof.
\end{proof}

Next, we establish a bound on the local mass that holds uniformly over time and space, for the continuous non-negative mild solution~$v$ given in Proposition~\ref{Paper02_smoothness_mild_solution}, with the extra assumption that the reaction term~$F = (F_{k})_{k \in \mathbb{N}_{0}}$ is monostable (see Definition~\ref{Paper02_assumption_monostable_condition}), and that the initial condition~$f \in L_{\infty}(\mathbb{R};\ell_{1})$ satisfies Assumption~\ref{Paper02_assumption_initial_condition_modified} in addition to Assumption~\ref{Paper02_assumption_initial_condition}. Although this result will be used only in Section~\ref{Paper02_section_asymptotic_behaviour_PDE} to address the asymptotic behaviour of the system of PDEs~\eqref{Paper02_PDE_scaling_limit}, we will state and prove it here since its proof follows from results in this section.

\begin{lemma} \label{Paper02_uniform_bound_total_mass}
     Suppose $m > 0$, $\mu \in [0,1]$, $(s_k)_{k \in \mathbb N_0}$ satisfies Assumption~\ref{Paper02_assumption_fitness_sequence}, $q_+,q_-: [0, \infty) \rightarrow [0, \infty)$ satisfy Assumption~\ref{Paper02_assumption_polynomials}, and $f \in L_{\infty}(\mathbb R; \ell_1)$ satisfies Assumption~\ref{Paper02_assumption_initial_condition} and~\ref{Paper02_assumption_initial_condition_modified}.
     Suppose also that the reaction term~$F = (F_{k})_{k \in \mathbb{N}_{0}}$ is monostable in the sense of Definition~\ref{Paper02_assumption_monostable_condition}. Let $v = (v_k)_{k \in \mathbb N_0}: [0, \infty) \times \mathbb R \rightarrow \ell_1^+$ be the version of the mild solution to the system of PDEs~\eqref{Paper02_PDE_scaling_limit} given in Proposition~\ref{Paper02_smoothness_mild_solution}. Then,
    \begin{equation*}
        \sup_{T \geq 0} \; \| v(T, \cdot) \|_{L_{\infty}(\mathbb{R}; \ell_{1})} \leq 1.
    \end{equation*}
\end{lemma}

\begin{proof}
    Observe that since $s_k \leq 1 \; \forall \, k \in \mathbb N_0$ by Assumption~\ref{Paper02_assumption_fitness_sequence}, and by the definition of the reaction term in~\eqref{Paper02_reaction_term_PDE}, for any $t \geq 0$ and $x \in \mathbb R$ we have
    \begin{equation} \label{Paper02_pre_step_comparison_principle}
        \sum_{k = 0}^{\infty} F_{k}(v(t,x)) \leq \| v(t,x) \|_{\ell_1} \Big(q_{+}(\| v(t,x) \|_{\ell_1}) - q_{-}(\| v(t,x) \|_{\ell_1})\Big).
    \end{equation}
    Let $w: [0, \infty) \times \mathbb R \rightarrow \mathbb R$ be the unique non-negative function which is a classical solution to the reaction-diffusion PDE
    \begin{equation} \label{Paper02_classical_monostable_PDE}
    \left\{\begin{array}{lc}
         \partial_{t}w  = \frac{m}{2} \Laplace w + w\Big(q_{+}(w) - q_{-}(w)\Big), & t > 0,\\
         w(0, \cdot) = \| f(\cdot) \|_{\ell_1}. & 
    \end{array}
    \right.
    \end{equation}
    By conditions~(iii) and~(iv) from Definition~\ref{Paper02_assumption_monostable_condition}, and by Assumption~\ref{Paper02_assumption_initial_condition_modified}(i), we have (see e.g.~\cite[Proposition~2.1]{aronson1975nonlinear})
    \begin{equation}  \label{Paper02_uniform_bound_classical_sol}
        \sup_{T \geq 0} \| w(T,\cdot) \|_{L_{\infty}(\mathbb R; \mathbb R)} \leq 1.
    \end{equation}
    Moreover, by Proposition~\ref{Paper02_smoothness_mild_solution}(v), identities~\eqref{Paper02_F(v)_in_L_infty_aux_func_g_ii} and~\eqref{Paper02_F(v)_in_L_infty_aux_func_g}, and by~\eqref{Paper02_pre_step_comparison_principle} and~\eqref{Paper02_classical_monostable_PDE}, we have that for all $(t,x) \in (0, \infty) \times \mathbb R$,
    \begin{equation} \label{Paper02_pre_step_comparison_principle_ii}
    \begin{aligned}
          & \partial_t \| v(t,x) \|_{\ell_{1}} - \frac{m}{2} \Laplace \| v(t,x) \|_{\ell_{1}} - \| v(t,x)\|_{\ell_1} \Big(q_{+}(\| v(t,x) \|_{\ell_1}) - q_{-}(\| v(t,x) \|_{\ell_1})\Big)
          \\ & \quad \leq \partial_{t}w(t,x) - \frac{m}{2} \Laplace w(t,x) - w(t,x)\Big(q_{+}(w(t,x)) - q_{-}(w(t,x))\Big).
    \end{aligned}
    \end{equation}
    By Proposition~\ref{Paper02_smoothness_mild_solution}(iii) and  by~\eqref{Paper02_pre_step_comparison_principle_ii}, we can apply classical comparison theorems for parabolic PDEs (see e.g.~\cite[Proposition~2.1]{aronson1975nonlinear}), and conclude from~\eqref{Paper02_uniform_bound_classical_sol} that
    \begin{equation*}
        \sup_{(T,x) \in (0,\infty) \times \mathbb R}  \| v(T,x) \|_{\ell_{1}} \leq   \sup_{(T,x) \in (0,\infty) \times \mathbb R} w(T,x) \leq 1,
    \end{equation*}
    which (together with Assumption~\ref{Paper02_assumption_initial_condition_modified}(i)) yields the desired uniform bound and completes the proof.
\end{proof}

\section{Asymptotic behaviour of the system of PDEs} \label{Paper02_section_asymptotic_behaviour_PDE}

In this section, we will investigate the asymptotic behaviour of solutions of the system of PDEs~\eqref{Paper02_PDE_scaling_limit}, and prove Theorem~\ref{Paper02_prop_control_proportions}, as well as Theorems~\ref{Paper02_spreading_speed_Fisher_KPP} and~\ref{Paper02_thm_tracer_dynamics_no_gene_surfing}. In what follows, we denote the version of the mild solution of the system of PDEs~\eqref{Paper02_PDE_scaling_limit} given in Proposition~\ref{Paper02_smoothness_mild_solution} by $$u = (u_{k}(t,x))_{k \in \mathbb{N}_{0}, \, t \geq 0, \, x \in \mathbb{R}}.$$
We will refer to this solution as the continuous mild solution of~\eqref{Paper02_PDE_scaling_limit}.

We will assume that the reaction term~$F = (F_{k})_{k \in \mathbb{N}_{0}}$ defined in~\eqref{Paper02_reaction_term_PDE} is monostable, i.e.~that $F = (F_{k})_{k \in \mathbb{N}_{0}}$ satisfies conditions (i)-(v) of Definition~\ref{Paper02_assumption_monostable_condition}, and that the initial condition~$f = (f_{k})_{k \in \mathbb{N}_{0}}$ of the system of PDEs~\eqref{Paper02_PDE_scaling_limit} satisfies both Assumptions~\ref{Paper02_assumption_initial_condition} and~\ref{Paper02_assumption_initial_condition_modified}. Our strategy to determine the long-term behaviour of~$u$ will be to represent each of the functions~$u_{k}: [0,T] \times \mathbb{R} \rightarrow \mathbb{R}$ in terms of a Feynman--Kac formula. We divide the remainder of this section into four subsections. In Section~\ref{Paper02_subsection_feynman_kac}, we state and prove the Feynman--Kac representation. We prove Theorem~\ref{Paper02_prop_control_proportions} in Section~\ref{Paper02_subsection_evol_proportion_mutations}. The control over the evolution of the prevalence of mutations gained in Section~\ref{Paper02_subsection_evol_proportion_mutations}, together with ideas adapted from~\cite{penington2018spreading}, will be used to prove Theorem~\ref{Paper02_spreading_speed_Fisher_KPP} in Section~\ref{Paper02_fkpp_case}. Finally, we will prove Theorem~\ref{Paper02_thm_tracer_dynamics_no_gene_surfing} in Section~\ref{Paper02_tracer_dynamics_section}.

\subsection{Feynman--Kac representation formula} \label{Paper02_subsection_feynman_kac}

Fix $m \in (0, \infty)$, let $(W(t))_{t \geq 0}$ denote a Brownian motion run at speed~$m$, and for every $x \in \mathbb R$, let $\mathbb{P}_{x}$ denote the probability measure associated to $(W(t))_{t \geq 0}$ started from~$x$, and let $\mathbb E_x$ denote the corresponding expectation. In this section, we will extensively use a Feynman--Kac formula. For ease of reference, we quote a version stated and proved in~\cite[Proposition~3.1]{berestycki2019global}, which does not require continuity of the solution at time $0$.

\begin{proposition}[Feynman--Kac formula] \label{Paper02_general_version_feynman_kac}
    Suppose $G,H: (0, \infty) \times \mathbb{R} \rightarrow \mathbb{R}$ are continuous, and uniformly bounded on compact time intervals. Suppose that $v \in \mathscr{C}^{1,2}((0,\infty) \times \mathbb{R})$ is bounded on compact time intervals, and satisfies
    \begin{equation*}
    \begin{aligned}
        \partial_{t} v & = \frac{m}{2} \Laplace v + Gv + H \quad \forall (t,x) \in (0,\infty) \times \mathbb{R},
        \quad v(0, \cdot) = f(\cdot),
    \end{aligned}
    \end{equation*}
    and $v(t,x) \rightarrow f(x)$ as $t \rightarrow 0$ for every $x \in \mathbb R$ such that $f$ is continuous at $x$, where~$f \in L_{\infty}(\mathbb R)$ is continuous at $\lambda$-almost every $x \in \mathbb{R}$. Then, for any $T \geq 0$, $t \in [0,T]$ and $x \in \mathbb{R}$,
    \begin{equation*}
    \begin{aligned}
        v(T,x) = & \mathbb{E}_{x}\Bigg[v(T-t,W(t))\exp\left(\int_{0}^{t} G(T-\tau, W(\tau)) d\tau \right) \\ & \quad \quad + \int_{0}^{t}H(T-\tau, W(\tau))\exp\left(\int_{0}^{\tau} G(T-{\theta}, W(\theta)) \,  d\theta \right) d\tau\Bigg].
    \end{aligned}
    \end{equation*}
\end{proposition}

Our next result provides a Feynman–Kac representation of $u_{k}$, for each~$k \in \mathbb{N}_{0}$.

\begin{lemma}[Feynman--Kac representation for solutions of~\eqref{Paper02_PDE_scaling_limit}] \label{Paper02_feynman_kac_representation_lemma_each_u_k}
    Suppose $m > 0$, $\mu \in [0,1]$, $(s_k)_{k \in \mathbb N_0}$ satisfies Assumption~\ref{Paper02_assumption_fitness_sequence}, $q_+,q_-: [0, \infty) \rightarrow [0, \infty)$ satisfy Assumption~\ref{Paper02_assumption_polynomials}, and $f \in L_{\infty}(\mathbb R; \ell_1)$ satisfies Assumption~\ref{Paper02_assumption_initial_condition}. Let $u = (u_k(t,x))_{k \in \mathbb N_0, \, t \geq 0, \, x \in \mathbb R}$ denote the continuous mild solution to the system of PDEs~\eqref{Paper02_PDE_scaling_limit} given in Proposition~\ref{Paper02_smoothness_mild_solution}. Then, recalling that $u_k(0,x) = f_k(x) \; \forall k \in \mathbb N_0, \, x \in \mathbb R$, for any~$T \geq 0$, $t \in [0,T]$ and~$x \in \mathbb{R}$,
    \begin{equation} \label{Paper02_feynman_kac_no_mutations}
    \begin{aligned}
    u_{0}(T,x) = \mathbb{E}_{x}\left[u_{0}(T-t, W(t)) \exp\left(\int_{0}^{t} \Big((1 - \mu)q_{+} - q_{-}\Big)\Big(\| u(T-\tau, W(\tau)) \|_{\ell_{1}}\Big) \, d\tau\right)\right].
    \end{aligned}
    \end{equation}
    Moreover, for any~$k \in \mathbb{N}$, $T \geq 0$ and $x \in \mathbb R$,
    \begin{equation} \label{Paper02_feynman_kac_with_mutations}
\begin{aligned}
    u_{k}(T,x) = & \mathbb{E}_{x}\left[f_{k}(W(T))\exp\left(\int_{0}^{T} \Big(s_{k} (1-\mu)q_{+} - q_{-}\Big) \Big(\| u(T-t,W(t))\|_{\ell_{1}}\Big) dt\right)\right] \\ & + \mathbb{E}_{x}\Bigg[\int_{0}^{T} \Big(s_{k-1}\mu u_{k-1}q_{+}(\| u \|_{\ell_{1}})\Big)(T-t, W(t)) \\ & \quad \quad \quad \quad\quad \quad\cdot \exp\left(\int_{0}^{t} \Big(s_{k} (1-\mu)q_{+} - q_{-}\Big) \Big(\| u(T-\tau,W(\tau))\|_{\ell_{1}}\Big) d\tau\right) \; dt\Bigg].
\end{aligned}
\end{equation}
\end{lemma}

\begin{remark} \label{Paper02_remark_difference_between_feynman_kac_representation_u_0_u_k}
We collect some observations concerning the Feynman--Kac representation below.

\begin{enumerate}[(a)]
    \item We observe that the Feynman--Kac formula for~$u_{0}$ differs from the formula for~$u_{k}$ with~$k \in \mathbb N$ since, by the reaction term $F = (F_{k})_{k \in \mathbb{N}_{0}}$ defined in~\eqref{Paper02_reaction_term_PDE}, only the subpopulation without mutations can be a source for~$u_{0}$, while for~$k \in \mathbb N$, the subpopulations with~$k-1$ or $k$ mutations can both be sources for~$u_{k}$.
    \item In the special case~$t = T$ of~\eqref{Paper02_feynman_kac_no_mutations}, we have
    \begin{equation} \label{Paper02_feynman_kac_no_mutations_from_begining}
    u_{0}(T,x) = \mathbb{E}_{x}\left[f_{0}(W(T)) \exp\left(\int_{0}^{T} \Big((1 - \mu)q_{+} - q_{-}\Big)\Big(\| u(T-\tau, W(\tau)) \|_{\ell_{1}}\Big) \, d\tau\right)\right].
\end{equation}
\end{enumerate}
\end{remark}

\begin{proof}[Proof of Lemma~\ref{Paper02_feynman_kac_representation_lemma_each_u_k}]
    By the properties of $u$ given  in Proposition~\ref{Paper02_smoothness_mild_solution}, for every~$k \in \mathbb{N}_{0}$ and $T > 0$, $u_{k} \in \mathscr{C}^{1,2}((0, T] \times \mathbb{R}; \mathbb{R})$ is uniformly bounded on~$[0,T] \times \mathbb R$, and~$u_{k}$ satisfies~\eqref{Paper02_coordinate_wise_strong_duhamel}. Moreover, the functions
    \begin{equation*}
    \begin{aligned}
        (t,x) & \mapsto G^{(k)}(t,x) \defeq \Big(s_{k}(1-\mu)q_{+} - q_{-}\Big)\Big(\| u(t,x) \|_{\ell_{1}}\Big),\\ \textrm{and } (t,x) & \mapsto H^{(k)}(t,x) \defeq \Big(\mathds{1}_{\{k \geq 1\}}s_{k-1}\mu u_{k-1}q_{+}(\| u \|_{\ell_{1}})\Big)(t,x)
    \end{aligned}
    \end{equation*}
    are bounded and continuous on~$(0,T] \times \mathbb{R}$. Since, by Assumption~\ref{Paper02_assumption_initial_condition}, $f_{k} \in L_{\infty}(\mathbb R)$ is continuous at $\lambda$-almost every $x \in \mathbb{R}$, identities~\eqref{Paper02_feynman_kac_no_mutations} and~\eqref{Paper02_feynman_kac_with_mutations} follow directly from Proposition~\ref{Paper02_general_version_feynman_kac}.
\end{proof}

\subsection{Evolution of the prevalence of mutations} \label{Paper02_subsection_evol_proportion_mutations}

In this subsection, we will prove Theorem~\ref{Paper02_prop_control_proportions}. Throughout this subsection, we assume that $F = (F_{k})_{k \in \mathbb N_0}$ is monostable in the sense of Definition~\ref{Paper02_assumption_monostable_condition}, and so in particular~$(s_{k})_{k \in \mathbb{N}_{0}}$ satisfies Assumption~\ref{Paper02_assumption_fitness_sequence} and is strictly decreasing, and~$\mu \in (0,1)$. Recall that the sequence~$(\alpha_{k})_{k \in \mathbb{N}_{0}}$ defined in~\eqref{Paper02_definition_alpha_k} is given by $\alpha_0 = 1$ and
\begin{equation} \label{Paper02_redefining_alpha_k}
    \alpha_{k} = \prod_{i=1}^{k} \frac{\mu s_{i-1}}{(1-\mu) (1 - s_{i})} \quad \forall \, k \in \mathbb N.
\end{equation}
Observe that since, by Assumption~\ref{Paper02_assumption_fitness_sequence}(iv), $s_{k} \rightarrow 0$ as $k \rightarrow \infty$, we have $(\alpha_{k})_{k\in \mathbb{N}_0} \in \ell_{1}^{+}$. Moreover, recall from~\eqref{Paper02_definition_max_min_birth_polynomial} that
\[
\mathfrak{Q}_{\min} = \min_{U \in [0,1]} q_+(U) \quad \textrm{and} \quad \mathfrak{Q}_{\max} = \max_{U \in [0,1]} q_+(U).
\]
Our first step towards the proof of Theorem~\ref{Paper02_prop_control_proportions} will be to define a sequence of functions~$(\underline{\pi}_{k})_{k \in \mathbb{N}_{0}}$ that we will use to obtain a lower bound on the ratio of~$u_{k}$ and~$u_{0}$, for each~$k \in \mathbb{N}_{0}$.

\begin{lemma} \label{Paper02_evolution_proportions_characterisation}
   Suppose $F = (F_k)_{k \in \mathbb N_0}$ is monostable in the sense of Definition~\ref{Paper02_assumption_monostable_condition}. Define a sequence $(\underline{\pi}_k)_{k \in \mathbb{N}_0}$ of non-negative real-valued functions on $[0, \infty)$  with the following recursive formula: for all $T \geq 0$, let
    \begin{equation} \label{Paper02_evolution_proportion_time}
    \underline{\pi}_k(T) \defeq \left\{\begin{array}{lc}
       1 & \textrm{if } k = 0, \\ \mu \mathfrak{Q}_{\min}\int_{0}^{T} \exp\Big(-t\mathfrak{Q}_{\max}(1 - s_{1})(1 - \mu)\Big) \, dt  & \textrm{if } k =1, \\
        \mu s_{k-1} \mathfrak{Q}_{\min} \int_{0}^{T} \underline{\pi}_{k-1} (T-t) \exp\Big(- t\mathfrak{Q}_{\max}(1 - s_{k})(1 - \mu)\Big) \, dt & \textrm{for } k \geq 2.
    \end{array} \right.
\end{equation}
Then, for $k \in \mathbb{N}$, $\underline{\pi}_{k}$ is strictly increasing on $[0, \infty)$ and strictly positive on $(0, \infty)$. Moreover, for all $k \in \mathbb{N}$,
    \begin{equation} \label{Paper02_limit_time_lower_bound_prevalence_mutations}
        \lim_{T \rightarrow \infty} \underline{\pi}_{k}(T) = \alpha_{k}\left(\frac{\mathfrak{Q}_{\min}}{\mathfrak{Q}_{\max}}\right)^{k} > 0.
    \end{equation}
\end{lemma}

By our observation after~\eqref{Paper02_redefining_alpha_k} that $(\alpha_k)_{k \in \mathbb N_0} \in \ell_1^+$, and the fact that by~\eqref{Paper02_definition_max_min_birth_polynomial} and by Definition~\ref{Paper02_assumption_monostable_condition}(ii), $0 < \mathfrak Q_{\min} \leq \mathfrak Q_{\max}$, we have that $\left(\alpha_{k}\left(\frac{\mathfrak{Q}_{\min}}{\mathfrak{Q}_{\max}}\right)^{k}\right)_{k \in \mathbb{N}_{0}} \in \ell_{1}^{+}$. 

\begin{proof}
    We will use an induction argument on~$k \in \mathbb{N}$. Observe that by the definition of~$\underline{\pi}_{1}$ in~\eqref{Paper02_evolution_proportion_time}, we have for all $T > 0$,
    \begin{equation} \label{Paper02_derivative_lower_bound_proportion_0}
        \frac{d}{dT}\underline{\pi}_{1}(T) = \mu \mathfrak{Q}_{\min} \exp\Big(-T\mathfrak{Q}_{\max}(1- s_{1})(1 - \mu)\Big) > 0,
    \end{equation}
    since by Definition~\ref{Paper02_assumption_monostable_condition}(ii), the fact that~$F=(F_{k})_{k \in \mathbb{N}_{0}}$ is monostable implies that~$\mathfrak{Q}_{\min} > 0$. It follows from~\eqref{Paper02_derivative_lower_bound_proportion_0} that~$\underline{\pi}_{1}$ is strictly increasing on $[0,\infty)$. By taking the limit in~\eqref{Paper02_evolution_proportion_time} as~$T \rightarrow \infty$, we have
    \begin{equation} \label{Paper02_limit_on_pi_0_lower}
    \begin{aligned}
        \lim_{T \rightarrow \infty} \underline{\pi}_{1}(T) & = \mu \mathfrak{Q}_{\min} \int_{0}^{\infty} \exp \Big(-t \mathfrak{Q}_{\max}(1 - s_{1})(1-\mu)\Big) dt \\ & = \frac{\mu \mathfrak{Q}_{\min}}{(1 - \mu)(1 - s_{1})\mathfrak{Q}_{\max}} \\ & = \alpha_{1} \frac{\mathfrak{Q}_{\min}}{\mathfrak{Q}_{\max}},
    \end{aligned}
    \end{equation}
    where for the second equality we used our assumption that~$F = (F_{k})_{k \in \mathbb{N}_{0}}$ is monostable and so $(1 - s_{1})(1-\mu) > 0$ by Definition~\ref{Paper02_assumption_monostable_condition}, and for the last line we used the definition of~$\alpha_{1}$ in~\eqref{Paper02_definition_alpha_k}.
    
    Suppose now that for some~$k \in \mathbb{N}$, $\underline{\pi}_{k}$ satisfies the following conditions:
    \begin{enumerate}[(i)]
        \item $\underline{\pi}_{k} \in \mathscr C^{1}((0,\infty))$ with bounded derivative, and for all $T >0$,
        \begin{equation} \label{Paper02_intermediate_step_control_lower_bound_prop_mut_ii}
            \frac{d}{dT}\underline{\pi}_{k}(T) > 0.
        \end{equation}
        \item The limit in~\eqref{Paper02_limit_time_lower_bound_prevalence_mutations} holds for $\underline{\pi}_{k}$.
    \end{enumerate}
    Under conditions~(i) and~(ii), by differentiating~$\underline{\pi}_{k+1}(T)$ with respect to~$T$, using~\eqref{Paper02_evolution_proportion_time}, dominated convergence and the fact that~\eqref{Paper02_evolution_proportion_time} implies that~$\underline{\pi}_{k}(0) = 0$, we conclude that for $T > 0$,
    \begin{equation} \label{Paper02_intermediate_step_lower_bound_proportions_i}
    \begin{aligned}
          \frac{d}{dT}\underline{\pi}_{k+1}(T) = \mu s_{k} \mathfrak{Q}_{\min} \int_{0}^{T} \frac{d}{dT} \underline{\pi}_{k} (T-t) \exp\Big(- t\mathfrak{Q}_{\max}(1 - s_{k+1})(1 - \mu)\Big) \, dt > 0,
    \end{aligned}
    \end{equation}
    where the inequality follows from~\eqref{Paper02_intermediate_step_control_lower_bound_prop_mut_ii}. Observe that~\eqref{Paper02_intermediate_step_lower_bound_proportions_i} implies, in particular, that~$\underline{\pi}_{k+1}$ is strictly increasing on $[0,\infty)$. Hence, to complete the induction argument, it remains to show that the limit in~\eqref{Paper02_limit_time_lower_bound_prevalence_mutations} holds with~$k$ replaced by~$k+1$. By taking the limit in~\eqref{Paper02_evolution_proportion_time} as~$T \rightarrow \infty$, and then using monotone convergence and our assumption that~\eqref{Paper02_limit_time_lower_bound_prevalence_mutations} holds for $\underline \pi_k$,
    \begin{equation*}
    \begin{aligned}
        \lim_{T \rightarrow \infty} \underline{\pi}_{k+1}(T) & = \mu s_{k} \mathfrak{Q}_{\min} \lim_{T \rightarrow \infty}  \int_{0}^{\infty} \mathds{1}_{\{t \leq T\}}\underline{\pi}_{k} (T-t) \exp\Big(- t\mathfrak{Q}_{\max}(1 - s_{k+1})(1 - \mu)\Big) \, dt \\ & = \mu s_{k} \mathfrak{Q}_{\min} \int_{0}^{\infty} \alpha_{k} \left(\frac{\mathfrak{Q}_{\min}}{\mathfrak{Q}_{\max}}\right)^{k} \exp\Big(- t\mathfrak{Q}_{\max}(1 - s_{k+1})(1 - \mu)\Big) \, dt \\ & = \alpha_{k+1}\left(\frac{\mathfrak{Q}_{\min}}{\mathfrak{Q}_{\max}}\right)^{k+1},
    \end{aligned}
    \end{equation*}
    where we used the definition of~$\alpha_{k}$ in~\eqref{Paper02_definition_alpha_k} for the last identity. This completes the proof, by induction on~$k$.
\end{proof}

We now define a sequence of auxiliary functions that we will use for upper bounds on the ratio of~$u_k$ and~$u_0$. Suppose that $f = (f_{k})_{k \in \mathbb N_0}$ satisfies Assumptions~\ref{Paper02_assumption_initial_condition} and~\ref{Paper02_assumption_initial_condition_modified}, and recall that $\left(\hat{\pi}_{k}\right)_{k \in \mathbb{N}_{0}} \in \ell_1^+$ is defined in Assumption~\ref{Paper02_assumption_initial_condition_modified}(iii).

\begin{lemma} \label{Paper02_evolution_proportions_upper_bound_characterisation}
   Suppose $F = (F_k)_{k \in \mathbb N_0}$ is monostable in the sense of Definition~\ref{Paper02_assumption_monostable_condition}. Define $\phi = (\phi_k)_{k \in \mathbb N}$, where $\phi_k: [0,\infty) \rightarrow [0, \infty)$ for every $k \in \mathbb N$, as follows. Let $\phi_{1} \equiv 0$ and, for each~$k \geq 2$ and $T \geq 0$, let
\begin{equation} \label{Paper02_evolution_proportion_time_upper_bound}
   {\phi_{k}}(T) \defeq \sum_{i = 1}^{k - 1} \frac{(\mu \mathfrak{Q}_{\max}T)^{k-i}}{(k - i)!}  \Bigg(\prod_{j = i}^{k-1} s_{j}\Bigg)\hat{\pi}_{i}\exp\Big(-T\mathfrak{Q}_{\min}(1-s_{i})(1 - \mu)\Big).
\end{equation}
Then~$\Phi(T) \defeq \|  {\phi(T)} \|_{\ell_{1}} < \infty$ for every~$T \geq 0$, with $\sup_{T \geq 0} \Phi(T) < \infty$ and
    \begin{equation} \label{Paper02_vanishing_sum_error_term}
        \lim_{T \rightarrow \infty} \Phi(T) = 0.
    \end{equation}
\end{lemma}

\begin{proof}
    
    We start by noticing that by~\eqref{Paper02_evolution_proportion_time_upper_bound} and Fubini's theorem, $\Phi(T) = \| \phi (T) \|_{\ell_{1}}$ is given by
    \begin{equation} \label{Paper02_rewriting_Theta_T}
    \begin{aligned}
        \Phi(T) & = \sum_{k = 2}^{\infty} \; \sum_{i = 1}^{k - 1} \frac{(\mu \mathfrak{Q}_{\max}T)^{k-i}}{(k - i)!}  \Bigg(\prod_{j = i}^{k-1} s_{j}\Bigg)\hat{\pi}_{i}\exp\Big(-T \mathfrak{Q}_{\min}(1 - s_{i})(1 - \mu)\Big) \\ & = \sum_{i = 1}^{\infty} \; \sum_{k = i + 1}^{\infty} \frac{(\mu \mathfrak{Q}_{\max}T)^{k-i}}{(k - i)!}  \Bigg(\prod_{j = i}^{k-1} s_{j}\Bigg)\hat{\pi}_{i}\exp\Big(-T \mathfrak{Q}_{\min}(1 - s_{i})(1 - \mu)\Big)
        \\ & = \sum_{i = 1}^{\infty} \hat{\pi}_{i}\exp\Big(-T \mathfrak{Q}_{\min}(1 - s_{i})(1 - \mu)\Big)  \sum_{k = i + 1}^{\infty} \frac{(\mu \mathfrak{Q}_{\max}T)^{k-i}}{(k - i)!}  \Bigg(\prod_{j = i}^{k-1} s_{j}\Bigg).
    \end{aligned}
    \end{equation}
    Since, by Definition~\ref{Paper02_assumption_monostable_condition}, $(s_k)_{k \in \mathbb N_0}$ is strictly decreasing, $\mu \in (0,1)$ and $\mathfrak{Q}_{\min} > 0$, and by Assumption~\ref{Paper02_assumption_fitness_sequence}(iv), $s_{k} \rightarrow 0$ as $k \rightarrow \infty$, there exists $J \in \mathbb{N}$ such that
    \begin{equation} \label{Paper02_definition_J_bound_capital_Theta}
    0 < s_{k} < \frac{(1-s_{1})(1 - \mu)}{2\mu} \cdot \frac{\mathfrak{Q}_{\min}}{\mathfrak{Q}_{\max}} \quad \forall k \geq J.
    \end{equation}
For any $i \in \mathbb{N}$,
\begin{equation} \label{Paper02_splitting_sum_small_large_fitness_parameters}
\begin{aligned}
   & \sum_{k = i + 1}^{\infty} \frac{(\mu \mathfrak{Q}_{\max}T)^{k-i}}{(k - i)!}  \Bigg(\prod_{j = i}^{k-1} s_{j}\Bigg)  \\ & = \sum_{k = i + 1}^{\infty} \frac{(\mu \mathfrak{Q}_{\max}T)^{k-i}}{(k - i)!}  \Bigg(\prod_{j = i}^{k-1} s_{j}\Bigg) \mathds{1}_{\{k \leq J\}} + \sum_{k = i + 1}^{\infty} \frac{(\mu \mathfrak{Q}_{\max}T)^{k-i}}{(k - i)!}  \Bigg(\prod_{j = i}^{k-1} s_{j}\Bigg) \mathds{1}_{\{k > J\}}.
\end{aligned}
\end{equation}
We will bound each of the sums on the right-hand side of~\eqref{Paper02_splitting_sum_small_large_fitness_parameters} separately. For the first sum, we note that since $s_j \leq 1 \; \forall \, j \in \mathbb N_0$ by Assumption~\ref{Paper02_assumption_fitness_sequence}, for any $i \in \mathbb N$, we have
\begin{equation} \label{Paper02_first_bound_capital_Theta_T}
\begin{aligned}
    \sum_{k = i + 1}^{\infty} \frac{(\mu \mathfrak{Q}_{\max}T)^{k-i}}{(k - i)!}  \Bigg(\prod_{j = i}^{k-1} s_{j}\Bigg) \mathds{1}_{\{k \leq J\}} \leq \sum_{k = 1}^{J-1} \frac{(\mu\mathfrak{Q}_{\max}T)^{k}}{k!}.
\end{aligned}
\end{equation}
For the second sum on the right-hand side of~\eqref{Paper02_splitting_sum_small_large_fitness_parameters}, we note that by letting $k' = k - (J \vee i)$ and using that~$(s_{k})_{k \in \mathbb N_0}$ is monotonically decreasing by Assumption~\ref{Paper02_assumption_fitness_sequence}(iii), we have for all $i \in \mathbb{N}$,
\begin{equation} \label{Paper02_second_bound_capital_Theta_T}
\begin{aligned}
    \sum_{k = i + 1}^{\infty}  \frac{(\mu \mathfrak{Q}_{\max}T)^{k-i}}{(k - i)!}  \Bigg(\prod_{j = i}^{k-1} s_{j}\Bigg) \cdot \mathds{1}_{\{k > J\}} & \leq \Big(1 \vee (\mu \mathfrak{Q}_{\max}T)^{J}\Big) \cdot \sum_{k' = 0}^{\infty} \frac{(\mu\mathfrak{Q}_{\max}s_{J}T)^{k'}}{(k')!} \\ & = \Big(1 \vee (\mu \mathfrak{Q}_{\max}T)^{J}\Big) \exp\Big(\mu\mathfrak{Q}_{\max}s_{J}T\Big) \\ & \leq \Big(1 \vee (\mu \mathfrak{Q}_{\max}T)^{J}\Big) \exp\left(\frac{ \mathfrak{Q}_{\min}(1 - s_{1})(1 - \mu)T}{2}\right),
\end{aligned}
\end{equation}
where for the last inequality we used~\eqref{Paper02_definition_J_bound_capital_Theta}. Finally, applying estimates~\eqref{Paper02_splitting_sum_small_large_fitness_parameters},~\eqref{Paper02_first_bound_capital_Theta_T} and~\eqref{Paper02_second_bound_capital_Theta_T} to~\eqref{Paper02_rewriting_Theta_T}, and then using the fact that~$(s_{k})_{k \in \mathbb{N}_{0}}$ is monotonically decreasing again, we conclude that for all $T \geq 0$,
\begin{equation} \label{Paper02_upper_bound_sum_theta_T}
\begin{aligned}
    \Phi(T) & \leq \sum_{k = 1}^{J-1} \frac{(\mu\mathfrak{Q}_{\max}T)^{k}}{k!} \cdot \sum_{i = 1}^{\infty} \hat{\pi}_{i}\exp\Big(-T\mathfrak{Q}_{\min}(1-s_{i})(1 - \mu)\Big)\\ & \quad \; + \Big(1 \vee (\mu \mathfrak{Q}_{\max}T)^{J}\Big) \exp\left(\frac{T \mathfrak{Q}_{\min}(1 - s_{1})(1 - \mu)}{2}\right)\sum_{i = 1}^{\infty} \hat{\pi}_{i}\exp\Big(- T\mathfrak{Q}_{\min}(1 - s_{i})(1 - \mu)\Big) \\ & \leq \exp\left(- \frac{T\mathfrak{Q}_{\min}(1 -s_{1})(1 - \mu)}{2}\right) \| \hat{\pi} \|_{\ell_{1}} \Bigg(\sum_{k = 1}^{J-1} \frac{(\mu\mathfrak{Q}_{\max}T)^{k}}{k!} + \Big(1 \vee (\mu \mathfrak{Q}_{\max}T)^{J}\Big)\Bigg).
\end{aligned}
\end{equation}
Hence, for any $T \geq 0$, we have $\Phi(T) < \infty$. Moreover, by Definition~\ref{Paper02_assumption_monostable_condition} we have $s_1, \mu < 1$ and $\mathfrak{Q}_{\min} > 0$, and so $\mathfrak{Q}_{\min}(1 - s_{1})(1-\mu) > 0$.  Hence, recalling that~$J \in \mathbb{N}$ is a fixed parameter satisfying~\eqref{Paper02_definition_J_bound_capital_Theta}, we see that $\Phi(T)$ is bounded, and by taking~$T \rightarrow \infty$, the right-hand side of~\eqref{Paper02_upper_bound_sum_theta_T} vanishes, which completes the proof.
\end{proof}

Our next result combines the Feynman–Kac representation from Lemma~\ref{Paper02_feynman_kac_representation_lemma_each_u_k} with Lemmas~\ref{Paper02_evolution_proportions_characterisation} and~\ref{Paper02_evolution_proportions_upper_bound_characterisation} to control the evolution of the ratio of~$u_k$ and~$u_0$ over time. We recall the sequences of non-negative real numbers~$(\alpha_{k})_{k \in \mathbb{N}_{0}}$ and~$(\hat{\pi}_{k})_{k \in \mathbb{N}_{0}}$ defined in~\eqref{Paper02_definition_alpha_k} and in Assumption~\ref{Paper02_assumption_initial_condition_modified}(iii). Also, recall the sequences of non-negative functions~$(\underline{\pi}_{k})_{k \in \mathbb{N}_0}$ and~$(\phi_{k})_{k \in \mathbb N}$ defined in~\eqref{Paper02_evolution_proportion_time} and~\eqref{Paper02_evolution_proportion_time_upper_bound}.

\begin{lemma} \label{Paper02_uniform_evolution_proportion_over_space_fkpp_case}
      Suppose that the reaction term $F = (F_{k})_{k \in \mathbb{N}_{0}}$ is monostable in the sense of Definition~\ref{Paper02_assumption_monostable_condition}, and that $f$ satisfies Assumptions~\ref{Paper02_assumption_initial_condition} and~\ref{Paper02_assumption_initial_condition_modified}, and let $u = (u_{k})_{k \in \mathbb{N}_{0}}$ be the continuous mild solution to the system of PDEs~\eqref{Paper02_PDE_scaling_limit} given in Proposition~\ref{Paper02_smoothness_mild_solution}. Then, for any $T > 0$, $x \in \mathbb{R}$ and every~$k \in \mathbb{N}$,
    \begin{equation} \label{Paper02_not_so_narrow_control_over_the_proportions_PTW_case_k_1}
      \underline{\pi}_{k}(T)u_{0}(T,x) \leq u_{k}(T,x) \leq \Bigg(\hat{\pi}_{k}\exp\Big(-T\mathfrak{Q}_{\min}(1 - s_{k})(1 - \mu)\Big) + \phi_{k}(T) + \alpha_{k}\left(\frac{\mathfrak{Q}_{\max}}{\mathfrak{Q}_{\min}}\right)^{k}\Bigg)u_{0}(T,x).
    \end{equation}
\end{lemma}

\begin{proof}
    We start by using an induction argument to establish the lower bound in~\eqref{Paper02_not_so_narrow_control_over_the_proportions_PTW_case_k_1}, i.e.~that for all $k \in \mathbb{N}$,
    \begin{equation} \label{Paper02_lower_bound_proportions_uniform_space}
        \underline{\pi}_{k}(T)u_{0}(T,x) \leq u_{k}(T,x) \; \forall \, T > 0 \textrm{ and } \forall \, x \in \mathbb{R}.
    \end{equation}
    Observe that by Definition~\ref{Paper02_assumption_monostable_condition},
    \begin{equation} \label{Paper02_s_k_lower_than_1}
        s_{k} < s_{0} = 1 \quad \forall \, k \in \mathbb N.
    \end{equation}
    Moreover, by Lemma~\ref{Paper02_uniform_bound_total_mass} and~\eqref{Paper02_definition_max_min_birth_polynomial},
    \begin{equation} \label{Paper02_trivial_conclusion_Q_max_Q_min}
        \inf_{t \geq 0} \; \inf_{x \in \mathbb{R}} q_{+}(\| u(t,x) \|_{\ell_{1}}) \geq \mathfrak{Q}_{\min} \quad \textrm{and} \quad \sup_{t \geq 0} \; \sup_{x \in \mathbb{R}} q_{+}(\| u(t,x) \|_{\ell_{1}}) \leq \mathfrak{Q}_{\max}.
    \end{equation}
    
    By using~\eqref{Paper02_feynman_kac_with_mutations} from Lemma~\ref{Paper02_feynman_kac_representation_lemma_each_u_k} with $k = 1$, and using the fact that $f_{1}$ is non-negative by Assumption~\ref{Paper02_assumption_initial_condition}, together with identity~\eqref{Paper02_s_k_lower_than_1}, and then using estimate~\eqref{Paper02_trivial_conclusion_Q_max_Q_min}, we have that for any $T \geq 0$ and $x \in \mathbb{R}$,
    \begin{equation*}
    \begin{aligned}
        & u_{1}(T,x) \\ & \quad \geq \mathbb{E}_{x}\Bigg[\int_{0}^{T} \Big(\mu u_{0}q_{+}(\| u \|_{\ell_{1}})\Big)(T-t, W(t)) \\ & \quad \quad \quad \quad\quad \quad \quad \quad \cdot \exp\left(\int_{0}^{t} \Big(s_{1} (1-\mu)q_{+} - q_{-}\Big) \Big(\| u(T-\tau,W(\tau))\|_{\ell_{1}}\Big) d\tau\right) \; dt\Bigg] \\ & \quad \geq \int_{0}^{T} \mu \mathfrak{Q}_{\min}\exp\Big(-t\mathfrak{Q}_{\max}(1 - s_{1})(1-\mu)\Big)  \\ & \quad \quad \quad \quad \quad \quad \cdot \mathbb{E}_{x}\Bigg[u_{0}(T-t, W(t))\exp\left(\int_{0}^{t} \Big((1-\mu)q_{+} - q_{-}\Big) \Big(\| u(T-\tau,W(\tau))\|_{\ell_{1}}\Big) d\tau\right) \Bigg] dt
      \\ & \quad = \underline{\pi}_{1}(T)u_{0}(T,x),
    \end{aligned}
    \end{equation*}
    where the final equality follows from the Feynman--Kac formula given in~\eqref{Paper02_feynman_kac_no_mutations} and the definition of $\underline{\pi}_{1}$ in~\eqref{Paper02_evolution_proportion_time}.
    
    Now take $k \in \mathbb{N}$ and suppose that~\eqref{Paper02_lower_bound_proportions_uniform_space} holds. Then, by using the Feynman--Kac representation in~\eqref{Paper02_feynman_kac_with_mutations} for~$u_{k+1}$ and the fact that $f_{k+1} \geq 0$ by Assumption~\ref{Paper02_assumption_initial_condition}, and then by our assumption~\eqref{Paper02_lower_bound_proportions_uniform_space}, and by~\eqref{Paper02_s_k_lower_than_1} and~\eqref{Paper02_trivial_conclusion_Q_max_Q_min}, and then Fubini's theorem, we have that for any $T \geq 0$ and $x \in \mathbb{R}$,
    \begin{equation*}
    \begin{aligned}
         & u_{k+1}(T,x) \\ & \quad \geq \mathbb{E}_{x}\Bigg[\int_{0}^{T} \Big(s_{k}\mu u_{k}q_{+}(\| u \|_{\ell_{1}})\Big)(T-t, W(t)) \\ & \quad \quad \quad \quad\quad \quad \quad\quad \quad\cdot \exp\left(\int_{0}^{t} \Big(s_{k+1} (1-\mu)q_{+} - q_{-}\Big) \Big(\| u(T-\tau,W(\tau))\|_{\ell_{1}}\Big) d\tau\right) \; dt\Bigg] \\ & \quad \geq \int_{0}^{T} \mu s_{k} \mathfrak{Q}_{\min} \underline{\pi}_{k}(T-t) \exp\Big(-t\mathfrak{Q}_{\max}(1-s_{k+1})(1-\mu)\Big) \\ & \quad \quad \quad \quad\quad \quad \cdot \mathbb{E}_{x}\Bigg[u_{0}(T-t, W(t))\exp\left(\int_{0}^{t} \Big((1-\mu)q_{+} - q_{-}\Big) \Big(\| u(T-\tau,W(\tau))\|_{\ell_{1}}\Big) d\tau\right)\Bigg] \; dt \\ & \quad = \underline{\pi}_{k+1}(T) u_{0}(T,x),
    \end{aligned}
    \end{equation*}
    where we used~\eqref{Paper02_evolution_proportion_time} and~\eqref{Paper02_feynman_kac_no_mutations} for the final equality. Hence, by induction,~\eqref{Paper02_lower_bound_proportions_uniform_space} holds for every~$k \in \mathbb{N}$.
    
    It remains to establish the upper bound in~\eqref{Paper02_not_so_narrow_control_over_the_proportions_PTW_case_k_1}, i.e.~that for every~$k \in \mathbb{N}$, $T > 0$ and~$x \in \mathbb{R}$,
    \begin{equation} \label{Paper02_definition_upper_bound_prevalence_mutations}
        u_{k}(T,x) \leq \Bigg(\hat{\pi}_{k}\exp\Big(-T\mathfrak{Q}_{\min}(1 - s_{k})(1 - \mu)\Big) + \phi_{k}(T) + \alpha_{k}\left(\frac{\mathfrak{Q}_{\max}}{\mathfrak{Q}_{\min}}\right)^{k}\Bigg)u_{0}(T,x).
    \end{equation}
    We will also use an induction argument to establish~\eqref{Paper02_definition_upper_bound_prevalence_mutations}. For $k = 1$, by using~\eqref{Paper02_feynman_kac_with_mutations}, Assumption~\ref{Paper02_assumption_initial_condition_modified}(iii),~\eqref{Paper02_trivial_conclusion_Q_max_Q_min} and that~$s_0 = 1$ by Assumption~\ref{Paper02_assumption_fitness_sequence}(i), we have that for $T > 0$ and $x \in \mathbb R$,
    \begin{equation} \label{Paper02_step_i_upper_bound_prevalence_k_=_1}
    \begin{aligned}
        u_{1}(T,x) & \leq \hat{\pi}_{1} \mathbb{E}_{x}\left[f_{0}(W(T))\exp\left(\int_{0}^{T} \Big(s_{1} (1-\mu)q_{+} - q_{-}\Big) \Big(\| u(T-t,W(t))\|_{\ell_{1}}\Big) dt\right)\right] \\ & \quad + \mu\mathfrak{Q}_{\max} \mathbb{E}_{x}\Bigg[\int_{0}^{T}u_{0}(T-t, W(t)) \\ & \quad \quad \quad \quad \quad \quad \quad \quad \quad \cdot \exp\left(\int_{0}^{t} \Big(s_{1} (1-\mu)q_{+} - q_{-}\Big) \Big(\| u(T-\tau,W(\tau))\|_{\ell_{1}}\Big) d\tau\right) \; dt\Bigg].
    \end{aligned}
    \end{equation}
    We now bound the terms on the right-hand side of~\eqref{Paper02_step_i_upper_bound_prevalence_k_=_1} separately. For the first term, by using~\eqref{Paper02_s_k_lower_than_1} and~\eqref{Paper02_trivial_conclusion_Q_max_Q_min}, we have
    \begin{equation} \label{Paper02_first_step_control_proportion_density_one_mutation_PTW}
    \begin{aligned}
        & \hat{\pi}_{1} \mathbb{E}_{x}\left[f_{0}(W(T))\exp\left(\int_{0}^{T} \Big(s_{1} (1-\mu)q_{+} - q_{-}\Big) \Big(\| u(T-t,W(t))\|_{\ell_{1}}\Big) dt\right)\right] \\ & \quad \leq \hat{\pi}_{1}\exp\Big(-T\mathfrak{Q}_{\min}(1 -s_{1})(1 - \mu)\Big)  \\ & \quad \quad \quad \cdot \mathbb{E}_{x}\left[f_{0}(W(T))\exp\left(\int_{0}^{T} \Big((1-\mu)q_{+} - q_{-}\Big) \Big(\| u(T-t,W(t))\|_{\ell_{1}}\Big) dt\right)\right] \\ & \quad = \hat{\pi}_{1}\exp\Big(-T\mathfrak{Q}_{\min}(1 -s_{1})(1 - \mu)\Big) u_{0}(T,x),
    \end{aligned}
    \end{equation}
    where for the last equality we used~\eqref{Paper02_feynman_kac_no_mutations_from_begining}. To bound the second term on the right-hand side of~\eqref{Paper02_step_i_upper_bound_prevalence_k_=_1}, we use~\eqref{Paper02_s_k_lower_than_1},~\eqref{Paper02_trivial_conclusion_Q_max_Q_min} and Fubini's theorem, and then~\eqref{Paper02_feynman_kac_no_mutations}, obtaining
    \begin{equation} \label{Paper02_second_step_control_proportion_density_one_mutation_PTW}
    \begin{aligned}
        & \mu\mathfrak{Q}_{\max} \mathbb{E}_{x}\Bigg[\int_{0}^{T}u_{0}(T-t, W(t)) \exp\left(\int_{0}^{t} \Big(s_{1} (1-\mu)q_{+} - q_{-}\Big) \Big(\| u(T-\tau,W(\tau))\|_{\ell_{1}}\Big) d\tau\right) \; dt\Bigg] \\
        & \quad \leq  \int_{0}^{T} \mu\mathfrak{Q}_{\max} \exp\Big(-t\mathfrak{Q}_{\min}(1 - s_{1})(1 - \mu)\Big) \\ & \quad \quad \quad \quad \cdot \mathbb{E}_{x}\Bigg[u_{0}(T-t, W(t)) \exp\left(\int_{0}^{t} \Big((1-\mu)q_{+} - q_{-}\Big) \Big(\| u(T-\tau,W(\tau))\|_{\ell_{1}}\Big) d\tau\right)\Bigg] \, dt \\ & \quad = \mu \mathfrak{Q}_{\max} u_{0}(T,x) \int_{0}^{T} \exp\Big(-t\mathfrak{Q}_{\min}(1 - s_{1})(1 - \mu)\Big) \, dt \\ & \quad \leq \mu \mathfrak{Q}_{\max} u_{0}(T,x) \int_{0}^{\infty} \exp\Big(-t\mathfrak{Q}_{\min}(1 - s_{1})(1 - \mu)\Big) \, dt \\ & \quad = \alpha_{1} \frac{\mathfrak{Q}_{\max}}{\mathfrak{Q}_{\min}}u_{0}(T,x),
    \end{aligned}
    \end{equation}
    where for the last equality we used~\eqref{Paper02_definition_alpha_k} and the fact that $\mathfrak Q_{\min} (1-s_1)(1-\mu) > 0$ by Definition~\ref{Paper02_assumption_monostable_condition}. Applying~\eqref{Paper02_first_step_control_proportion_density_one_mutation_PTW} and~\eqref{Paper02_second_step_control_proportion_density_one_mutation_PTW} to~\eqref{Paper02_step_i_upper_bound_prevalence_k_=_1}, and recalling that we defined~$\phi_{1} \equiv 0$ in the statement of Lemma~\ref{Paper02_evolution_proportions_upper_bound_characterisation}, we conclude that~\eqref{Paper02_definition_upper_bound_prevalence_mutations} holds for all~$T > 0$ and~$x \in \mathbb R$ for~$k = 1$.
    
    Suppose now that for some~$k \in \mathbb{N}$,~\eqref{Paper02_definition_upper_bound_prevalence_mutations} holds for all $T > 0$ and $x \in \mathbb R$. By applying~\eqref{Paper02_feynman_kac_with_mutations} for $u_{k+1}$, and using Assumption~\ref{Paper02_assumption_initial_condition_modified}(iii), the induction hypothesis~\eqref{Paper02_definition_upper_bound_prevalence_mutations} for~$k$,~\eqref{Paper02_trivial_conclusion_Q_max_Q_min}, and~\eqref{Paper02_feynman_kac_no_mutations}, for $T > 0$ and $x \in \mathbb R$ we have
    \begin{equation} \label{Paper02_intermediate_step_upper_bound_prevalence_mutations_ii}
    \begin{aligned}
        & u_{k+1}(T,x) \\ & \; \leq \hat{\pi}_{k+1} \mathbb{E}_{x}\left[f_{0}(W(T))\exp\left(\int_{0}^{T} \Big(s_{k+1} (1-\mu)q_{+} - q_{-}\Big) \Big(\| u(T-t,W(t))\|_{\ell_{1}}\Big) dt\right)\right] \\ & \quad \; + \mu \mathfrak{Q}_{\max} s_{k} \hat{\pi}_{k}u_{0}(T,x)\int_{0}^{T} \exp\Big(-t\mathfrak{Q}_{\min}(1 - s_{k+1})(1 - \mu) - (T-t)\mathfrak{Q}_{\min}(1 - s_{k})(1 - \mu)\Big) \, dt \\ & \quad \; + \mu \mathfrak{Q}_{\max} s_{k} u_{0}(T,x)\int_{0}^{T} \exp\Big(-t\mathfrak{Q}_{\min}(1 - s_{k+1})(1 - \mu)\Big) \phi_{k}(T-t) \, dt \\ & \quad \; +\mu \mathfrak{Q}_{\max} s_{k} u_{0}(T,x)\int_{0}^{T} \exp\Big(-t\mathfrak{Q}_{\min}(1 - s_{k+1})(1 - \mu)\Big) \alpha_{k}\left(\frac{\mathfrak{Q}_{\max}}{\mathfrak{Q}_{\min}}\right)^{k} \, dt.
    \end{aligned}
    \end{equation}
    We will bound each of the terms on the right-hand side of~\eqref{Paper02_intermediate_step_upper_bound_prevalence_mutations_ii} separately. For the first term, by using~\eqref{Paper02_s_k_lower_than_1},~\eqref{Paper02_trivial_conclusion_Q_max_Q_min}, and~\eqref{Paper02_feynman_kac_no_mutations_from_begining}, we have
    \begin{equation} \label{Paper02_intermediate_step_upper_bound_prevalence_mutations_iii}
    \begin{aligned}
        & \hat{\pi}_{k+1} \mathbb{E}_{x}\left[f_{0}(W(T))\exp\left(\int_{0}^{T} \Big(s_{k+1} (1-\mu)q_{+} - q_{-}\Big) \Big(\| u(T-t,W(t))\|_{\ell_{1}}\Big) dt\right)\right] \\ & \quad \leq  \hat{\pi}_{k+1} \exp\Big(- T \mathfrak{Q}_{\min}(1 - s_{k+1})(1 - \mu) \Big) u_{0}(T,x).
    \end{aligned}
    \end{equation}
    For the second term on the right-hand side of~\eqref{Paper02_intermediate_step_upper_bound_prevalence_mutations_ii}, we use that $1 - s_{k+1} \geq 1 - s_{k}$ by Assumption~\ref{Paper02_assumption_fitness_sequence}(iii), and therefore
    \begin{equation} \label{Paper02_intermediate_step_upper_bound_prevalence_mutations_iv}
    \begin{aligned}
         & \mu \mathfrak{Q}_{\max} s_{k} \hat{\pi}_{k}u_{0}(T,x)\int_{0}^{T} \exp\Big(-t\mathfrak{Q}_{\min}(1 - s_{k+1})(1 - \mu) - (T-t)\mathfrak{Q}_{\min}(1 - s_{k})(1 - \mu)\Big) \, dt \\ & \quad \leq \mu \mathfrak{Q}_{\max} s_{k} \hat{\pi}_{k}u_{0}(T,x)\int_{0}^{T} \exp\Big(-T\mathfrak{Q}_{\min}(1 - s_{k})(1 - \mu)\Big) \, dt \\ & \quad = \mu \mathfrak{Q}_{\max}T s_{k} \hat{\pi}_{k} \exp \Big(-T \mathfrak{Q}_{\min}(1 - s_{k})(1 - \mu)\Big)u_{0}(T,x).
    \end{aligned}
    \end{equation}
    To bound the third term on the right-hand side of~\eqref{Paper02_intermediate_step_upper_bound_prevalence_mutations_ii}, we use the definition of~$\phi_{k}$ in~\eqref{Paper02_evolution_proportion_time_upper_bound}, and then that $1 - s_{i} \leq 1 - s_{k+1}$ for every~$i \in \mathbb{N} \cap \{1, \ldots, k\}$  by Assumption~\ref{Paper02_assumption_fitness_sequence}(iii), to deduce that
    \begin{equation} \label{Paper02_intermediate_step_upper_bound_prevalence_mutations_v}
    \begin{aligned}
        & \int_{0}^{T} \exp\Big(-t\mathfrak{Q}_{\min}(1 - s_{k+1})(1 - \mu)\Big) \phi_{k}(T-t) \, dt \\ & \quad = \int_{0}^{T} \exp\Big(-t\mathfrak{Q}_{\min}(1 - s_{k+1})(1 - \mu)\Big) \\ & \quad \quad \quad \quad \quad \cdot \sum_{i = 1}^{k - 1} \frac{(\mu \mathfrak{Q}_{\max}(T- t))^{k-i}}{(k - i)!}  \Bigg(\prod_{j = i}^{k-1} s_{j}\Bigg)\hat{\pi}_{i}\exp\Big(-(T-t)\mathfrak{Q}_{\min}(1-s_{i})(1 - \mu)\Big) \, dt \\ & \quad \leq \sum_{i = 1}^{k-1} \Bigg(\int_{0}^{T} \frac{(\mu \mathfrak{Q}_{\max}(T- t))^{k-i}}{(k - i)!} \, dt\Bigg) \Bigg(\prod_{j = i}^{k-1} s_{j}\Bigg)\hat{\pi}_{i}\exp\Big(-T\mathfrak{Q}_{\min}(1-s_{i})(1 - \mu)\Big) \\ & \quad = \sum_{i = 1}^{k - 1} \frac{\mu^{k-i} \mathfrak{Q}_{\max}^{k-i}T^{k+1-i}}{(k + 1 - i)!}  \Bigg(\prod_{j = i}^{k-1} s_{j}\Bigg)\hat{\pi}_{i}\exp\Big(-T\mathfrak{Q}_{\min}(1-s_{i})(1 - \mu)\Big).
    \end{aligned}
    \end{equation}    Applying~\eqref{Paper02_intermediate_step_upper_bound_prevalence_mutations_v} to the third term on the right-hand side of~\eqref{Paper02_intermediate_step_upper_bound_prevalence_mutations_ii}, we get
    \begin{equation} \label{Paper02_intermediate_step_upper_bound_prevalence_mutations_vi}
    \begin{aligned}
        & \mu \mathfrak{Q}_{\max} s_{k} u_{0}(T,x)\int_{0}^{T} \exp\Big(-t\mathfrak{Q}_{\min}(1 - s_{k+1})(1 - \mu)\Big) \phi_{k}(T-t) \, dt \\ & \quad \leq \sum_{i = 1}^{k - 1} \frac{(\mu \mathfrak{Q}_{\max}T)^{k+1-i}}{(k + 1 - i)!}  \Bigg(\prod_{j = i}^{k} s_{j}\Bigg)\hat{\pi}_{i}\exp\Big(-T\mathfrak{Q}_{\min}(1-s_{i})(1 - \mu)\Big) u_0(T,x).
    \end{aligned}
    \end{equation}
    Combining~\eqref{Paper02_intermediate_step_upper_bound_prevalence_mutations_iv},~\eqref{Paper02_intermediate_step_upper_bound_prevalence_mutations_vi} and the definition of~$\phi_{k+1}$ in~\eqref{Paper02_evolution_proportion_time_upper_bound}, we conclude that the second and the third terms on the right-hand side of~\eqref{Paper02_intermediate_step_upper_bound_prevalence_mutations_ii} satisfy
    \begin{equation} \label{Paper02_intermediate_step_upper_bound_prevalence_mutations_vii}
    \begin{aligned}
        & \mu \mathfrak{Q}_{\max} s_{k} \hat{\pi}_{k}u_{0}(T,x)\int_{0}^{T} \exp\Big(-t\mathfrak{Q}_{\min}(1 - s_{k+1})(1 - \mu) - (T-t)\mathfrak{Q}_{\min}(1 - s_{k})(1 - \mu)\Big) \, dt \\ & \quad + \mu \mathfrak{Q}_{\max} s_{k} u_{0}(T,x)\int_{0}^{T} \exp\Big(-t\mathfrak{Q}_{\min}(1 - s_{k+1})(1 - \mu)\Big) \phi_{k}(T-t) \, dt \\ & \quad \leq \phi_{k+1}(T) u_{0}(T,x).
    \end{aligned}
    \end{equation}
    Finally, for the fourth term on the right-hand side of~\eqref{Paper02_intermediate_step_upper_bound_prevalence_mutations_ii}, we observe that
    \begin{equation} \label{Paper02_intermediate_step_upper_bound_prevalence_mutations_viii}
    \begin{aligned}
        & \mu \mathfrak{Q}_{\max} s_{k} u_{0}(T,x)\int_{0}^{T} \exp\Big(-t\mathfrak{Q}_{\min}(1 - s_{k+1})(1 - \mu)\Big) \alpha_{k}\left(\frac{\mathfrak{Q}_{\max}}{\mathfrak{Q}_{\min}}\right)^{k} \, dt \\ & \quad \leq \mu \mathfrak{Q}_{\max} s_{k} \alpha_{k}\left(\frac{\mathfrak{Q}_{\max}}{\mathfrak{Q}_{\min}}\right)^{k} u_{0}(T,x)\int_{0}^{\infty} \exp\Big(-t\mathfrak{Q}_{\min}(1 - s_{k+1})(1 - \mu)\Big) \, dt \\ & \quad = \alpha_{k+1}\left(\frac{\mathfrak{Q}_{\max}}{\mathfrak{Q}_{\min}}\right)^{k+1} u_{0}(T,x),
    \end{aligned}
    \end{equation}
    where we used~\eqref{Paper02_definition_alpha_k} and the fact that $\mathfrak Q_{\min} (1 -s_{k+1})(1-\mu) > 0$ by Definition~\ref{Paper02_assumption_monostable_condition} for the last equality. Applying~\eqref{Paper02_intermediate_step_upper_bound_prevalence_mutations_iii},~\eqref{Paper02_intermediate_step_upper_bound_prevalence_mutations_vii} and~\eqref{Paper02_intermediate_step_upper_bound_prevalence_mutations_viii} to~\eqref{Paper02_intermediate_step_upper_bound_prevalence_mutations_ii}, we conclude that~\eqref{Paper02_definition_upper_bound_prevalence_mutations} holds with~$k$ replaced by~$k+1$, which completes the induction argument. The proof is then complete.
\end{proof}

Before proving Theorem~\ref{Paper02_prop_control_proportions}, we state and prove a result that is a corollary of Lemmas~\ref{Paper02_evolution_proportions_upper_bound_characterisation} and~\ref{Paper02_uniform_evolution_proportion_over_space_fkpp_case}, and that will be used in bounding the spreading speed in Section~\ref{Paper02_fkpp_case}.

\begin{corollary} \label{Paper02_same_speed_u_0_total_mass}
    Suppose that $F = (F_{k})_{k \in \mathbb N_0}$ is monostable in the sense of Definition~\ref{Paper02_assumption_monostable_condition}, and that~$f$ satisfies Assumptions~\ref{Paper02_assumption_initial_condition} and~\ref{Paper02_assumption_initial_condition_modified}, and let $u = (u_k)_{k \in \mathbb N_0}$ be the continuous mild solution to the system of PDEs~\eqref{Paper02_PDE_scaling_limit} given in Proposition~\ref{Paper02_smoothness_mild_solution}. Then there exists~$A > 0$ such that for any $T > 0$ and any $x \in \mathbb{R}$,
    \begin{equation*}
        u_{0}(T,x) \leq \| u(T,x) \|_{\ell_{1}} \leq (1 + A) u_{0}(T,x).
    \end{equation*}
\end{corollary}

\begin{proof}
    The first inequality follows from the definition of the norm $\| \cdot \|_{\ell_{1}}$. For the second inequality, let
    \begin{equation} \label{Paper02_definition_maximal_proportion_mutants_in_general}
        A \defeq \sup_{t \geq 0} \; \sum_{k \in \mathbb N} \left( \hat{\pi}_{k}\exp\Big(-t\mathfrak{Q}_{\min}(1 - s_{k})(1 - \mu)\Big) + \phi_{k}(t) + \alpha_{k}\left(\frac{\mathfrak{Q}_{\max}}{\mathfrak{Q}_{\min}}\right)^{k}\right),
    \end{equation}
    where~$(\alpha_{k})_{k \in \mathbb{N}_{0}}$ and~$(\hat{\pi}_{k})_{k \in \mathbb{N}_{0}}$ are defined in~\eqref{Paper02_definition_alpha_k} and Assumption~\ref{Paper02_assumption_initial_condition_modified}(iii), and~$(\phi_{k})_{k \in \mathbb N}$ is defined in~\eqref{Paper02_evolution_proportion_time_upper_bound}.
    By Assumption~\ref{Paper02_assumption_initial_condition_modified}(iii), the sequence  $(\hat{\pi}_{k})_{k \in \mathbb{N}}$ is summable. Moreover, by Lemma~\ref{Paper02_evolution_proportions_upper_bound_characterisation}, we have $\sup_{t \geq 0} \Phi(t) < \infty$, where $  \Phi(t) = \sum_{k \in \mathbb N} \phi_k(t) \; \forall \, t \geq 0$. Also, as observed after Remark~\ref{Paper02_very_good_control_proportions_FKPP}, the sequence $\left(\alpha_{k}\left(\frac{\mathfrak{Q}_{\max}}{\mathfrak{Q}_{\min}}\right)^{k}\right)_{k \in \mathbb{N}_{0}}$ is summable. Therefore, the constant~$A$ defined in~\eqref{Paper02_definition_maximal_proportion_mutants_in_general} is finite, and the result follows from Lemma~\ref{Paper02_uniform_evolution_proportion_over_space_fkpp_case}.
\end{proof}

We are now ready to prove Theorem~\ref{Paper02_prop_control_proportions}.

\begin{proof}[Proof of Theorem~\ref{Paper02_prop_control_proportions}]
    For every~$k \in \mathbb{N}$, we define $\underline{\pi}_{k}: [0,\infty) \rightarrow [0,\infty)$ as in~\eqref{Paper02_evolution_proportion_time}, and the map~$\overline{\pi}_{k}: [0,\infty) \rightarrow [0,\infty)$ by
    \begin{equation*}
        \overline{\pi}_{k}(T) \defeq \hat{\pi}_{k}\exp\Big(-T\mathfrak{Q}_{\min}(1 - s_{k})(1 - \mu)\Big) + \phi_{k}(T) + \alpha_{k}\left(\frac{\mathfrak{Q}_{\max}}{\mathfrak{Q}_{\min}}\right)^{k} \quad \forall \, T \geq 0.
    \end{equation*}
    Note that $(\hat \pi_k)_{k \in \mathbb N_0} \in \ell_1^+$ by Assumption~\ref{Paper02_assumption_initial_condition_modified}(iii), and for all $k \in \mathbb N$,
    \[
    \mathfrak Q_{\min}(1-s_k)(1-\mu) \geq \mathfrak Q_{\min}(1-s_1)(1-\mu) > 0
    \]
    by Definition~\ref{Paper02_assumption_monostable_condition}. 
    Moreover, as observed after Remark~\ref{Paper02_very_good_control_proportions_FKPP}, the sequences
    $$\left(\alpha_{k}\left(\frac{\mathfrak{Q}_{\min}}{\mathfrak{Q}_{\max}}\right)^{k}\right)_{k \in \mathbb{N}_{0}}\quad \textrm{and} \quad \left(\alpha_{k}\left(\frac{\mathfrak{Q}_{\max}}{\mathfrak{Q}_{\min}}\right)^{k}\right)_{k \in \mathbb{N}_{0}}$$
    are elements of $\ell_{1}^{+}$.
    Therefore, Theorem~\ref{Paper02_prop_control_proportions} follows directly from Lemmas~\ref{Paper02_evolution_proportions_characterisation},~\ref{Paper02_evolution_proportions_upper_bound_characterisation} and~\ref{Paper02_uniform_evolution_proportion_over_space_fkpp_case}, noting that since $\underline \pi_k$ is increasing for each $k$ by Lemma~\ref{Paper02_evolution_proportions_characterisation}, the limit in~\eqref{Paper02_limit_time_lower_bound_prevalence_mutations} also holds in~$\ell_1$.
\end{proof}

\subsection{Spreading speed under Fisher-KPP dynamics} \label{Paper02_fkpp_case}

In this subsection, we will prove Theorem~\ref{Paper02_spreading_speed_Fisher_KPP} by computing the spreading speed of the solution of~\eqref{Paper02_PDE_scaling_limit} when the reaction term $F = (F_k)_{k \in \mathbb N_0}$ satisfies a Fisher-KPP condition. Throughout this subsection, we let~$u = (u_{k})_{k \in \mathbb{N}_{0}}: [0,\infty) \times \mathbb{R} \rightarrow \ell_{1}^{+}$ denote the continuous mild solution to the system of PDEs~\eqref{Paper02_PDE_scaling_limit} given in Proposition~\ref{Paper02_smoothness_mild_solution}. We suppose that $(s_k)_{k \in \mathbb N_0}$, $q_+$, $q_-$ and $f$ satisfy Assumptions~\ref{Paper02_assumption_fitness_sequence},~\ref{Paper02_assumption_polynomials},~\ref{Paper02_assumption_initial_condition} and~\ref{Paper02_assumption_initial_condition_modified}, and that the reaction term~$F = (F_{k})_{k \in \mathbb{N}_{0}}$ defined in~\eqref{Paper02_reaction_term_PDE} is monostable (in the sense of Definition~\ref{Paper02_assumption_monostable_condition}) and is of Fisher-KPP type, i.e.~we suppose that~\eqref{Paper02_fisher_kpp_condition_muller_ratchet} holds.

Our strategy will be to adapt the arguments used in the proof of~\cite[Theorem~1.1]{penington2018spreading}. We start by proving a H\"{o}lder estimate for~$u_{0}$ that holds uniformly in time, analogous to~\cite[Lemma~2.2]{penington2018spreading}. To simplify notation, let
\begin{equation} \label{Paper02_min_max_difference_polynomial}
    H_{\min} \defeq \min_{U \in [0,1]} \Big( (1-\mu)q_{+}(U) - q_{-}(U)\Big) \quad \textrm{and} \quad H_{\max} \defeq (1-\mu)q_{+}(0) - q_{-}(0),
\end{equation}
and observe that since we are assuming throughout this subsection that the reaction term~$F = (F_{k})_{k \in \mathbb{N}_{0}}$ is of Fisher-KPP type in the sense of Definition~\ref{Paper02_assumption_monostable_condition}, we have
\begin{equation} \label{Paper02_min_max_difference_polynomial_ii}
    H_{\max} = \max_{U \in [0,1]} \Big( (1-\mu)q_{+}(U) - q_{-}(U)\Big) > 0.
\end{equation}

\begin{lemma} \label{Paper02_coordinate-wise_uniform_continuity}
    Suppose that $(s_k)_{k \in \mathbb N_0}$, $q_+$ and $q_-$ satisfy Assumptions~\ref{Paper02_assumption_fitness_sequence} and~\ref{Paper02_assumption_polynomials}, that $\mu \in (0,1)$, that $m > 0$, that the reaction term $F = (F_{k})_{k \in \mathbb{N}_{0}}$ defined in~\eqref{Paper02_reaction_term_PDE} is of Fisher-KPP type in the sense of Definition~\ref{Paper02_assumption_monostable_condition},  and that $f$ satisfies Assumptions~\ref{Paper02_assumption_initial_condition} and~\ref{Paper02_assumption_initial_condition_modified}. Then, there exists $\varepsilon_0 > 0$ such that for any~$\varepsilon \in (0, \varepsilon_0)$, for any $t \geq 1$ and~$x,y \in \mathbb{R}$ with $\vert x - y \vert \leq \varepsilon^{3}$,
    \begin{equation*}
        \vert u_{0}(t,x) - u_{0}(t,y) \vert \leq \varepsilon.
    \end{equation*}
\end{lemma}

\begin{proof}
    The proof closely follows the arguments used in~\cite[Lemma~2.2]{penington2018spreading}. By Lemma~\ref{Paper02_uniform_bound_total_mass}, we have $\| u(t,x) \|_{\ell_1} \leq 1$ for all $t > 0$, $x \in \mathbb R$. Therefore, by the Feynman--Kac representation in~\eqref{Paper02_feynman_kac_no_mutations} and by~\eqref{Paper02_min_max_difference_polynomial} and~\eqref{Paper02_min_max_difference_polynomial_ii}, we have that for every~$z \in \mathbb{R}$, any~$t \geq 1$ and any $\varepsilon \in (0,1)$,
    \begin{equation} \label{Paper02_short_window_values_u_0}
        \exp(\varepsilon^{2}H_{\min}) \mathbb{E}_{z}[u_{0}(t - \varepsilon^{2}, W(\varepsilon^{2}))] \leq u_{0}(t,z) \leq \exp(\varepsilon^{2}H_{\max}) \mathbb{E}_{z}[u_{0}(t - \varepsilon^{2}, W(\varepsilon^{2}))].
    \end{equation}
    Hence, for any $t \geq 1$ and any~$x,y \in \mathbb{R}$,
    \begin{equation} \label{Paper02_intermediate_step_uniform_continuity_space_regardles_time_i}
    \begin{aligned}
        u_{0}(t,x) - u_{0}(t,y) & \leq \exp(\varepsilon^{2}H_{\max}) (\mathbb{E}_{x}[u_{0}(t - \varepsilon^{2}, W(\varepsilon^{2}))] - \mathbb{E}_{y}[u_{0}(t - \varepsilon^{2}, W(\varepsilon^{2}))]) \\ & \quad + \exp(\varepsilon^{2}H_{\max})\Big(1 - \exp(- \varepsilon^{2}(H_{\max} - H_{\min}))\Big)\mathbb{E}_{y}[u_{0}(t - \varepsilon^{2}, W(\varepsilon^{2}))].
    \end{aligned}
    \end{equation}
    We will bound each term on the right-hand side of~\eqref{Paper02_intermediate_step_uniform_continuity_space_regardles_time_i} separately. Recall the definition of $p(\tau,z)$ in~\eqref{Paper02_Gaussian_kernel}. By Lemma~\ref{Paper02_uniform_bound_total_mass}, we can write
    \begin{equation*}
    \begin{aligned}
        \mathbb{E}_{x}[u_{0}(t - \varepsilon^{2}, W(\varepsilon^{2}))] - \mathbb{E}_{y}[u_{0}(t - \varepsilon^{2}, W(\varepsilon^{2}))] & \leq \int_\mathbb R (p(\varepsilon^2,x-z) - p(\varepsilon^2,y-z))^+ \, dz \\ 
        & = \mathbb{P}_{0}\Big(\vert W(\varepsilon^{2}) \vert \leq \vert x - y \vert/2\Big).
    \end{aligned}
    \end{equation*}
    Since $p(\tau,z) \leq \frac{1}{\sqrt{2\pi m \tau}} \; \forall \, z \in \mathbb R, \, \tau > 0$, we have that
    \begin{equation} \label{Paper02_intermediate_step_uniform_continuity_space_regardles_time_ii}
    \begin{aligned}
        \mathbb{E}_{x}[u_{0}(t - \varepsilon^{2}, W(\varepsilon^{2}))] - \mathbb{E}_{y}[u_{0}(t - \varepsilon^{2}, W(\varepsilon^{2}))] \leq \frac{\vert x - y \vert}{\varepsilon\sqrt{2\pi m}}.
    \end{aligned}
    \end{equation}
    For the second term on the right-hand side of~\eqref{Paper02_intermediate_step_uniform_continuity_space_regardles_time_i}, we use Lemma~\ref{Paper02_uniform_bound_total_mass} and the fact that $1 - \exp(-r) \leq r$ for every~$r \geq 0$, obtaining
    \begin{equation} \label{Paper02_intermediate_step_uniform_continuity_space_regardles_time_iii}
    \begin{aligned}
        \Big(1 - \exp(- \varepsilon^{2}(H_{\max} - H_{\min}))\Big)\mathbb{E}_{y}[u_{0}(t - \varepsilon^{2}, W(\varepsilon^{2}))] \leq \varepsilon^{2}(H_{\max} - H_{\min}).
    \end{aligned}
    \end{equation}
    By applying~\eqref{Paper02_intermediate_step_uniform_continuity_space_regardles_time_ii} and~\eqref{Paper02_intermediate_step_uniform_continuity_space_regardles_time_iii} to~\eqref{Paper02_intermediate_step_uniform_continuity_space_regardles_time_i}, for $t \geq 1$, $\varepsilon \in (0,1)$ and $x,y \in \mathbb R$ with $\vert x - y \vert \leq \varepsilon^3$, we get
    \begin{equation} \label{Paper02_intermediate_step_uniform_continuity_space_regardles_time_iv}
        u_{0}(t,x) - u_{0}(t,y) \leq \exp(H_{\max})\left(\frac{1}{\sqrt{2 \pi m}} + H_{\max} - H_{\min}\right)\varepsilon^{2}.
    \end{equation}
    Therefore, by taking
    \begin{equation*}
        \varepsilon_0 \leq 1 \wedge \left(\exp(H_{\max})\left(\frac{1}{\sqrt{2 \pi m}} + H_{\max} - H_{\min}\right)\right)^{-1},
    \end{equation*}
    the result follows directly from~\eqref{Paper02_intermediate_step_uniform_continuity_space_regardles_time_iv}.
\end{proof}

In the remainder of this subsection, take~$A > 0$ as given in Corollary~\ref{Paper02_same_speed_u_0_total_mass}, and let
\begin{equation} \label{Paper02_equilibrium_value_general_PDE}
    U_{\textrm{eq}} \defeq \min \Big\{U \geq 0: \, (1 - \mu)q_{+}(U) - q_{-}(U) = 0\Big\}.
\end{equation}
Observe that, since the reaction term~$F = (F_{k})_{k \in \mathbb{N}_{0}}$ is assumed to be of Fisher-KPP type in the sense of Definition~\ref{Paper02_assumption_monostable_condition}, we have~$ U_{\textrm{eq}} \in (0,1)$ and $U_{\textrm{eq}}$ is the smallest strictly positive equilibrium state of the ODE:
\begin{equation*}
    \frac{d}{dt}U = U((1 - \mu)q_{+}(U) - q_{-}(U)).
\end{equation*}

We now prove that for any~$\delta > 0$, if~$u_{0}(T,x)$ is sufficiently small, then $u_{0}$ grows exponentially fast until at some time $T+t$, there is some~$y \in \mathbb{R}$ close to $x$ such that $u_{0}(T+t,y) \geq \frac{U_{\textrm{eq}} - \delta}{1 + A}$. This result is analogous to~\cite[Lemma~2.3]{penington2018spreading}.

\begin{lemma} \label{Paper02_exponential_growth}
    Suppose the assumptions of Lemma~\ref{Paper02_coordinate-wise_uniform_continuity} hold. Let $u = (u_{k})_{k \in \mathbb N_0}: [0, \infty) \times \mathbb R \rightarrow \ell_1^+$ denote the continuous mild solution to~\eqref{Paper02_PDE_scaling_limit} given in Proposition~\ref{Paper02_smoothness_mild_solution}, let $A > 0$ denote the constant given in Corollary~\ref{Paper02_same_speed_u_0_total_mass} and let $\delta \in (0,U_{\textrm{eq}})$. Then, there exist positive real numbers $\tau = \tau(\delta)$, $R = R(\delta)$ and $z^{*} = z^{*}(\delta)$ such that for $z \in (0, z^{*})$ and $x \in \mathbb R$, if $T \geq 1$ and $u_{0}(T,x) > z$, then there exists $t \in [T, T + \tau\log(1/z)]$ and $y \in [x-R,x+R]$ such that $u_{0}(t,y) \geq \frac{U_{\textrm{eq}} - \delta}{1 + A}$. 
\end{lemma}

\begin{proof}
    We will closely follow the proof of~\cite[Lemma~2.3]{penington2018spreading}. Let $R > 1$ and $\tau > 1$ be large constants depending on $\delta$ that we will later specify. Let $z^* \in (0, (2\varepsilon_0) \wedge e^{-1})$ to be specified later, where $\varepsilon_0$ is defined in Lemma~\ref{Paper02_coordinate-wise_uniform_continuity}, and suppose $z \in (0,z^*)$, $T \geq 1$ and $x \in \mathbb R$ with $u_0(T,x) > z$. Following~\cite{penington2018spreading}, we divide the proof into two cases:
    \begin{enumerate}[(1)]
        \item For each $t \in  [T, T + \tau\log(1/z)]$ and $y \in [x - R, x + R]$, we have
        \begin{equation} \label{Paper02_case_1_BM_around_region_total_mass_is_small}
            \| u(t,y) \|_{\ell_{1}} < U_{\textrm{eq}} - \delta.
        \end{equation}
        \item There exist $\tilde{t} \in [T, T + \tau\log(1/z)]$ and $\tilde{y} \in [x - R, x + R]$ such that
        \begin{equation} \label{Paper02_case_2_BM_around_region_total_mass_is_big}
            \| u(\tilde{t},\tilde{y}) \|_{\ell_{1}} \geq U_{\textrm{eq}} - \delta.
        \end{equation}
    \end{enumerate}

    \noindent \underline{Case $(1)$:} Since $u_{0}(T,x) > z$, and $z < z^* < 2\varepsilon_0$, by Lemma~\ref{Paper02_coordinate-wise_uniform_continuity} with $\varepsilon = z/2$, we have that
\begin{equation} \label{Paper02_initial_condition_lower_bound_u_0}
    u_{0}(T,y) > \tfrac{1}{2}z \quad \forall \, y \in \left[x - \tfrac{1}{8}z^3, x + \tfrac{1}{8}z^3\right].
\end{equation}
Let
\begin{equation} \label{Paper02_level_set_positive_difference_birth_death_rates}
    H_{\min,\delta} \defeq \min_{U \in [0, U_{\textrm{eq}} - \delta]} \, \Big((1 - \mu)q_{+}(U) - q_{-}(U)\Big) > 0,
\end{equation}
where the inequality holds by~\eqref{Paper02_equilibrium_value_general_PDE} and our assumption that~$F=(F_{k})_{k \in \mathbb{N}_{0}}$ is of Fisher-KPP type in the sense of Definition~\ref{Paper02_assumption_monostable_condition}. Then, by~\eqref{Paper02_case_1_BM_around_region_total_mass_is_small} and~\eqref{Paper02_level_set_positive_difference_birth_death_rates}, on the event that $\vert W(t) - x \vert \leq R$ for all $t \in [0, \tau\log(1/z)]$, we have
\begin{equation} \label{Paper02_estimate_integral_over_constrict_tube}
    \int_{0}^{\tau\log(1/z)} \Big((1 - \mu) q_{+} - q_{-}\Big) \Big(\| u(T + \tau\log(1/z) - t, W(t)) \|_{\ell_{1}}\Big) \, dt \geq \tau \log(1/z) H_{\min,\delta}.
\end{equation}
Hence, by combining the Feynman--Kac representation in~\eqref{Paper02_feynman_kac_no_mutations} with~\eqref{Paper02_initial_condition_lower_bound_u_0} and~\eqref{Paper02_estimate_integral_over_constrict_tube}, we have
\begin{equation} \label{Paper02_exponential_growth_solution_u_0_i}
\begin{aligned}
    & u_{0}(T + \tau\log(1/z), x) \\ & \, \geq \frac{z}{2} \exp\left(\tau \log(1/z)H_{\min,\delta}\right) \mathbb{P}_{x}\left(\vert W(t) - x\vert \leq R \; \forall t \leq \tau\log(1/z), \;\vert W(\tau\log(1/z)) - x\vert \leq \frac{z^{3}}{8}\right) \\ & \, = \frac{z}{2} \exp\left(\tau \log(1/z) H_{\min,\delta}\right) \mathbb{P}_{0}\left(\vert W(t) \vert \leq 1 \; \forall t \leq \tau R^{-2}\log(1/z), \;\vert W(\tau R^{-2}\log(1/z))\vert \leq \frac{z^{3}}{8R}\right),
\end{aligned}
\end{equation}
where the last line holds due to Brownian scaling. Let $C = C^{(1)}_1 > 0$ denote the constant given in Lemma~\ref{Paper02_aux_results_BM}(iii) in the appendix, and recall from the start of the proof that $\tau > 1$, $\log (1/z) > 1$ and $z^3/(8R) < 2/3$. Therefore, by applying the Brownian motion estimate~\eqref{Paper02_roberts_lemma} from Lemma~\ref{Paper02_aux_results_BM}(iii) to~\eqref{Paper02_exponential_growth_solution_u_0_i}, we conclude that
\begin{equation} \label{Paper02_exponential_growth_solution_u_0_ii}
\begin{aligned}
    u_{0}(T + \tau\log(1/z), x) \geq \frac{z}{2} \cdot \frac{z^{3}}{8R} \cdot C \cdot \exp\left(\tau \log(1/z) H_{\min,\delta} - \frac{\pi^{2}m\tau}{8R^{2}}\log(1/z)\right) = \frac{C}{16 R} z^{a},
\end{aligned}
\end{equation}
where
\begin{equation} \label{Paper02_growth_exponent}
    a \defeq \frac{-8\tau H_{\min,\delta} + \pi^{2}m\tau R^{-2} + 32} {8}.
\end{equation}
Hence, by taking $R > 1$,  $\tau > 1$ large enough that the following estimates hold:
    \begin{equation} \label{Paper02_estimates_constants_conditions_i}
        R^{2} > \frac{\pi^{2}m}{8H_{\min,\delta}} \textrm{ and } \tau > \frac{32}{8H_{\min,\delta} - \pi^{2}mR^{-2}},
    \end{equation}
    it follows from~\eqref{Paper02_growth_exponent} that~$a < 0$. Therefore, by~\eqref{Paper02_exponential_growth_solution_u_0_ii}, we can choose $z^* = z^*(\delta) > 0$ to be sufficiently small that for $z \in (0,z^*)$, we have    \begin{equation*}
        u_{0}(T+\tau\log(1/z), x) \geq \frac{U_{\textrm{eq}} -\delta}{1 + A},
    \end{equation*}
    which completes the analysis of case~(1).
    
    \medskip

    \noindent \underline{Case~$(2)$:} By Corollary~\ref{Paper02_same_speed_u_0_total_mass} and then by~\eqref{Paper02_case_2_BM_around_region_total_mass_is_big}, we have
\begin{equation*}
    u_{0}(\tilde{t}, \tilde{y}) \geq \frac{\| u(\tilde{t}, \tilde{y}) \|_{\ell_{1}}}{1 + A} \geq \frac{U_{\textrm{eq}} - \delta}{1 + A},
\end{equation*}
which completes the proof.
\end{proof}

As in the statement of Theorem~\ref{Paper02_spreading_speed_Fisher_KPP}, let
\begin{equation} \label{Paper02_speed_Fisher_KPP_case_specific_constant}
    c^{*} \defeq \sqrt{2m\Big((1-\mu)q_{+}(0) - q_{-}(0)\Big)} > 0.
\end{equation}
Our next result shows that for any~$\varepsilon >0$, $u_{0}$ spreads with speed at least $c^{*} - \varepsilon$. Its proof is an adaptation of~\cite[Lemma~2.4]{penington2018spreading}, with the modification that in our setting the non-local interaction occurs in the type space, i.e.~in terms of numbers of mutations, rather than in the spatial domain, and that the polynomials~$q_{+}$ and~$q_{-}$ must satisfy the Fisher-KPP condition~\eqref{Paper02_fisher_kpp_condition_muller_ratchet} in Definition~\ref{Paper02_assumption_monostable_condition}, which generalises the classical Fisher-KPP condition.

\begin{lemma} \label{Paper02_propagation_mass_u_0}
Suppose the assumptions of Lemma~\ref{Paper02_coordinate-wise_uniform_continuity} hold. Let $u = (u_{k})_{k \in \mathbb N_0}: [0, \infty) \times \mathbb R \rightarrow \ell_1^+$ denote the continuous mild solution to~\eqref{Paper02_PDE_scaling_limit} given in Proposition~\ref{Paper02_smoothness_mild_solution}, and let $A > 0$ denote the constant given in Corollary~\ref{Paper02_same_speed_u_0_total_mass}. For $c \in (0, c^{*})$, there exist ${\nu}^{*}_{0} = {\nu}^{*}_{0}(c) \in  \left(0, \frac{U_{\textrm{eq}}}{2(1+A)}\right)$ and ${t}^{*} = {t}^{*}(c) < \infty$ such that for $T \geq {t}^{*}$ and $t \geq 1$, if $x,x' \in \mathbb R$ with $u_{0}(t,x) \geq {\nu}^{*}_{0}$ and $\vert x - x'\vert \leq cT$, then $u_{0}(t + T, x') \geq {\nu}^{*}_{0}$.
\end{lemma}

\begin{proof}
 We start by defining certain constants; the reasons for the choice of constants will become clear later. Recall the definition of~$U_{\textrm{eq}}$ in~\eqref{Paper02_equilibrium_value_general_PDE}. Fix $c \in (0, c^{*})$, and let~$\delta = \delta(c) \in (0, U_{\textrm{eq}})$ be such that
\begin{equation} \label{Paper02_first_constant_spreading_speed_proof_modified_i}
    H_{\min,\delta} > \frac{c^{2}}{2m},
\end{equation}
where~$H_{\min,\delta}$ is defined in~\eqref{Paper02_level_set_positive_difference_birth_death_rates}. Indeed, by~\eqref{Paper02_speed_Fisher_KPP_case_specific_constant} we have
\begin{equation*}
    (1 - \mu)q_{+}(0) - q_{-}(0) = \frac{(c^{*})^{2}}{2m} > \frac{c^{2}}{2m},
\end{equation*}
and so it is always possible to choose~$\delta \in (0, U_{\textrm{eq}})$ such that~\eqref{Paper02_first_constant_spreading_speed_proof_modified_i} holds. Let $R > 2$ be sufficiently large that
\begin{equation} \label{Paper02_first_condition_spread_mass_constants}
        H_{\min,\delta} - \frac{c^{2}}{2m} - \frac{\pi^{2}m}{8(R - 1)^{2}} > 0.
    \end{equation}
By Lemma~\ref{Paper02_coordinate-wise_uniform_continuity}, there exists ${\varepsilon}^{(1)} \in (0,1)$ sufficiently small that for any $x,y \in \mathbb R$ with $\vert x - y \vert \leq {\varepsilon}^{(1)}$ and any $\tau \geq 1$, we have
\begin{equation} \label{Paper02_intermediate_step_propagation_speed_i}
    \vert u_{0}(\tau, x) - u_{0}(\tau, y) \vert \leq \frac{U_{\textrm{eq}}-\delta}{2(1 + A)}.
\end{equation}
We then take ${\nu}^{*}_{0} > 0$ small enough that
\begin{equation} \label{Paper02_second_condition_spread_mass_constants}
    {\nu}^{*}_{0} \leq \frac{(U_{\textrm{eq}}-\delta) {\varepsilon}^{(1)}}{(1 + A)\sqrt{6m\pi}} \exp\left(3H_{\min} -\frac{(3c + R + 1)^{2}}{2m}\right),
\end{equation}
where $H_{\min}$ is defined in~\eqref{Paper02_min_max_difference_polynomial}. Let $C^{(2)}_{R-1} = C^{(2)}_{R-1}(c^*) > 0$ be the constant defined in Lemma~\ref{Paper02_aux_results_BM}(iv) in the appendix for
estimate~\eqref{Paper02_tilted_tube_BM}. We will also assume that ${\nu}^{*}_{0}$ is sufficiently small that
\begin{equation} \label{Paper02_third_condition_spread_mass_constants}
     {\nu}^{*}_{0} \leq e^{2H_{\min}}\frac{(U_{\textrm{eq}}-\delta)C^{(2)}_{R-1}}{4(1 + A)}\mathbb{P}_{c+R+1}\Big(\vert W(1) \vert \leq {\varepsilon}^{(1)}\Big) \mathbb{P}_{0}\left(\vert W(1) + c\vert \leq \frac{1}{2} \right).
\end{equation}
Applying Lemma~\ref{Paper02_coordinate-wise_uniform_continuity} again, let $\varepsilon^{(2)} \in (0,1)$ be sufficiently small that for any $x,y \in \mathbb R$ with $\vert x - y \vert \leq {\varepsilon}^{(2)}$ and any $\tau \geq 1$, we have
\begin{equation} \label{Paper02_intermediate_step_propagation_speed_ii}
    \vert u_{0}(\tau, x) - u_{0}(\tau, y) \vert \leq{{\nu}^{*}_{0}}/{2}.
\end{equation}
Finally, by~\eqref{Paper02_first_condition_spread_mass_constants}, we can take ${t}^{*} > 2$ sufficiently large that
\begin{equation}
\label{Paper02_fourth_condition_spread_mass_constants}
    \frac{\varepsilon^{(2)}}{4} C^{(2)}_{R-1}\exp\left((t^{*} - 1)\left(H_{\min,\delta} - \frac{c^{2}}{2m} - \frac{\pi^{2}m}{8(R - 1)^{2}} \right)\right) > e^{-H_{\min}}\mathbb{P}_{0}\Big(\vert W(1) - c \vert \leq \varepsilon^{(2)}/2\Big)^{-1}.
\end{equation}
Observe that since the reaction term~$F = (F_{k})_{k \in \mathbb{N}_{0}}$ is monostable, Definition~\ref{Paper02_assumption_monostable_condition} and~\eqref{Paper02_min_max_difference_polynomial} imply that~$H_{\min} < 0$, and so $e^{-H_{\min}} > 1$.

We are finally ready to start the proof. Take $T \geq t^* > 2$, $t \geq 1$ and $x,x' \in \mathbb R$ with $\vert x - x' \vert \leq cT$, and let $a \defeq (x' - x)/T$. Suppose that $u_0(t,x) \geq \nu^*_0$. Following the proof of~\cite[Lemma~2.4]{penington2018spreading}, we will consider three different cases:
\begin{enumerate}[(1)]
    \item For each $\tau \in [0, T-1]$ and $y \in \mathbb R$ with $\vert y - (x + a\tau) \vert \leq R$, we have $\| u(t + \tau, y) \|_{\ell_{1}} < U_{\textrm{eq}}-\delta$.
    \item There exist $\tilde{\tau} \in [T-3,T-1]$ and $\tilde{y} \in \mathbb{R}$ such that $\vert \tilde{y} - (x + a\tilde{\tau}) \vert \leq R$ and $\| u(t + \tilde{\tau}, \tilde{y}) \|_{\ell_{1}} \geq U_{\textrm{eq}}-\delta$.
    \item There exist $\tilde{\tau} \in [0,T-3]$ and $\tilde{y} \in \mathbb{R}$ such that $\vert \tilde{y} - (x + a\tilde{\tau}) \vert \leq R$ and $\| u(t + \tilde{\tau}, \tilde{y}) \|_{\ell_{1}} \geq U_{\textrm{eq}}-\delta$, and such that for any $(\tau, y)$ with $\tau \in [\tilde{\tau} + 1, T-1]$ and $\vert y - (x + a\tau) \vert \leq R$, we have $\| u(t + \tau, y) \|_{\ell_{1}} < U_{\textrm{eq}}-\delta$.
\end{enumerate}

To see that these cases exhaust all possibilities, if both (1) and (2) do not hold, for (3) consider the supremum over all $\tau' \in [0,T-3]$ such that there exists ${y}' \in \mathbb{R}$ such that $\vert {y}' - (x + a{\tau}') \vert \leq R$ and $\| u(t + {\tau}', {y}') \|_{\ell_{1}} \geq U_{\textrm{eq}}-\delta$. Then, choose the required $\tilde\tau$ within distance $1/2$ from this supremum.
For simplicity, we assume that $x' \geq x$, and so $a \in [0,c]$, since the proof for $x' \leq x$ is analogous. We will address each of the cases listed separately.

\medskip

\noindent \underline{Case~$(1)$:} Take $\tau \in [0,T-1]$ and suppose $\vert W(\tau) - (x + a(T-1-\tau)) \vert \leq R$; then by the assumptions of case~$(1)$, we have $\| u(t + T -1 - \tau, W(\tau)) \|_{\ell_{1}} < U_{\textrm{eq}}-\delta$. Hence, by~\eqref{Paper02_level_set_positive_difference_birth_death_rates} we have
\begin{equation} \label{Paper02_trivial_consequence_level_set_small_z}
    ((1 - \mu)q_{+} - q_{-})(\| u(t + T -1 - \tau, W(\tau)) \|_{\ell_{1}}) \geq H_{\min,\delta}.
\end{equation}
Moreover, by~\eqref{Paper02_intermediate_step_propagation_speed_ii} and the fact that $t \geq 1$ and $u_{0}(t,x) \geq {\nu}^{*}_{0}$, we also have that
\begin{equation} \label{Paper02_intermediate_step_propagation_speed_iii}
    u_{0}(t,y) \geq \tfrac{1}{2}{\nu}^{*}_{0} \quad \forall y \in [x-\varepsilon^{(2)}, x + \varepsilon^{(2)}].
\end{equation}
By combining the Feynman--Kac representation in~\eqref{Paper02_feynman_kac_no_mutations} with~\eqref{Paper02_trivial_consequence_level_set_small_z} and~\eqref{Paper02_intermediate_step_propagation_speed_iii}, we conclude that for any $y \in [x-\varepsilon^{(2)}/2, x + \varepsilon^{(2)}/2]$, we have
\begin{equation} \label{Paper02_intermediate_step_propagation_speed_iv}
\begin{aligned}
    & u_{0}(t + T - 1, y + a(T-1)) \\ & \quad \geq \frac{{\nu}^{*}_{0}}{2}\exp\Big((T-1)H_{\min,\delta}\Big) \\ & \quad \quad \quad \cdot \mathbb{P}_{y + a(T-1)}\Big(\vert W(\tau) - (x + a(T-1 - \tau))\vert \leq R \; \forall \tau \leq T-1,\,  \vert W(T-1) - x \vert \leq \varepsilon^{(2)}\Big) \\ & \quad \geq \frac{{\nu}^{*}_{0}}{2}\exp\Big((T-1)H_{\min,\delta}\Big) \\ & \quad \quad \quad \cdot \mathbb{P}_{0}\left(\vert W(\tau) + a\tau\vert \leq R - 1 \; \forall \tau \leq T-1,\,  \vert W(T-1) + a(T-1) \vert \leq \tfrac{1}{2}\varepsilon^{(2)}\right),
\end{aligned}
\end{equation}
where for the second inequality we used the fact that $\vert x - y \vert \leq \varepsilon^{(2)}/2$ and $\varepsilon^{(2)} < 1$. Note that since $\vert x' - x \vert \leq cT$, we have $\vert a \vert \leq c$. Thus, applying estimate~\eqref{Paper02_tilted_tube_BM} from Lemma~\ref{Paper02_aux_results_BM}(iv) in the appendix to~\eqref{Paper02_intermediate_step_propagation_speed_iv}, we have
\begin{equation} \label{Paper02_intermediate_step_propagation_speed_v}
\begin{aligned}
    u_{0}(t + T - 1, y + a(T-1)) & \geq \frac{{\nu}^*_{0}\varepsilon^{(2)}}{4} C^{(2)}_{R-1}\exp\left((T-1)\left(H_{\min,\delta} - \frac{a^{2}}{2m} - \frac{\pi^{2}m}{8(R-1)^{2}}\right)\right) \\ & \geq \frac{{\nu}^{*}_{0}\varepsilon^{(2)}}{4} C^{(2)}_{R-1}\exp\left((t^{*}-1)\left(H_{\min,\delta} - \frac{c^{2}}{2m} - \frac{\pi^{2}m}{8(R-1)^{2}}\right)\right) \\ & > {\nu}^{*}_{0} e^{-H_{\min}} \mathbb{P}_{0}\Big(\vert W(1) - c \vert \leq \varepsilon^{(2)}/2\Big)^{-1},
\end{aligned}
\end{equation}
where in the second inequality we used the fact that $T \geq t^{*}$ and~\eqref{Paper02_first_condition_spread_mass_constants}, and in the third inequality we used~\eqref{Paper02_fourth_condition_spread_mass_constants}. Since~\eqref{Paper02_intermediate_step_propagation_speed_v} holds for any $y \in [x-\varepsilon^{(2)}/2, x + \varepsilon^{(2)}/2]$ and since $x' = x + aT$, by using the Feynman--Kac representation in~\eqref{Paper02_feynman_kac_no_mutations}, Lemma~\ref{Paper02_uniform_bound_total_mass} and the definition of~$H_{\min}$ in~\eqref{Paper02_min_max_difference_polynomial}, and then using~\eqref{Paper02_intermediate_step_propagation_speed_v}, we conclude that
\begin{equation} \label{Paper02_desired_result_case_(1)}
\begin{aligned}
    & u_{0}(t + T, x') \\ & \quad \geq e^{H_{\min}}\mathbb{P}_{x + aT}\Big(\vert W(1) - x - a(T-1)\vert \leq \varepsilon^{(2)}/2\Big) \inf_{\vert y - x\vert \leq \varepsilon^{(2)}/2} u_{0}(t + T - 1, y + a(T-1)) \\ & \quad \geq e^{H_{\min}}\mathbb{P}_{0}\Big(\vert W(1) + a\vert \leq \varepsilon^{(2)}/2\Big) {\nu}^{*}_{0}e^{- H_{\min}}\mathbb{P}_{0}\Big(\vert W(1) - c \vert \leq \varepsilon^{(2)}/2\Big)^{-1} \\ & \quad > {\nu}^{*}_{0},
\end{aligned}
\end{equation}
where for the last inequality we used the fact that $0 \leq a \leq c$, together with the fact that the density of~$W(1)$ under $\mathbb P_0$ is given by the symmetric function $p(1,\cdot)$ (recall~\eqref{Paper02_Gaussian_kernel}), and $p(1,y)$ is decreasing in $y$ on $[0,\infty)$. This completes the analysis of case~$(1)$.

\medskip

\noindent \underline{Case~$(2)$:} By the assumptions of case $(2)$ and by Corollary~\ref{Paper02_same_speed_u_0_total_mass}, we have
\begin{equation*}
    u_{0}(t + \tilde{\tau}, \tilde{y}) \geq \frac{U_{\textrm{eq}}-\delta}{1 + A}.
\end{equation*}
Hence, by~\eqref{Paper02_intermediate_step_propagation_speed_i} and since $t \geq 1$, we have
\begin{equation} \label{Paper02_intermediate_step_propagation_speed_vi}
    u_{0}(t + \tilde{\tau}, y) \geq \frac{U_{\textrm{eq}}-\delta}{2(1 + A)} \quad \forall y \in [\tilde{y} - {\varepsilon}^{(1)}, \tilde{y} + {\varepsilon}^{(1)}].
\end{equation}
Since, by the assumptions of case~(2), $T - \tilde{\tau} \leq 3$ and since $x' = x + aT$, by using the Feynman--Kac representation in~\eqref{Paper02_feynman_kac_no_mutations}, and then by applying~\eqref{Paper02_intermediate_step_propagation_speed_vi}, Lemma~\ref{Paper02_uniform_bound_total_mass}, the definition of $H_{\min}$ in~\eqref{Paper02_min_max_difference_polynomial}, and our observation after~\eqref{Paper02_fourth_condition_spread_mass_constants} that~$H_{\min} < 0$, we conclude that
\begin{equation} \label{Paper02_intermediate_step_feynman_kac_x}
\begin{aligned}
    u_{0}(t + T, x') & \geq e^{3H_{\min}}\frac{U_{\textrm{eq}}-\delta}{2(1+A)}\mathbb{P}_{x + aT}\Big(\vert W(T - \tilde{\tau}) - \tilde{y}\vert \leq {\varepsilon}^{(1)}\Big) \\ & = e^{3H_{\min}}\frac{U_{\textrm{eq}}-\delta}{2(1+A)}\mathbb{P}_{0}\Big(\vert W(T - \tilde{\tau}) - (\tilde{y} - x - aT)\vert \leq {\varepsilon}^{(1)}\Big).
\end{aligned}
\end{equation}
Then since, by the triangle inequality and the assumptions of case~(2), we have
\[
\vert \tilde y - x - aT \vert \leq \vert \tilde y - (x + a\tilde \tau) \vert + \vert a(\tilde \tau - T)\vert \leq R +3a,
\]
it follows from~\eqref{Paper02_intermediate_step_feynman_kac_x} that
\begin{equation} \label{Paper02_intermediate_step_propagation_speed_vii}
\begin{aligned}
     u_{0}(t + T, x') \geq e^{3H_{\min}} \frac{U_{\textrm{eq}}-\delta}{2(1 + A)}\mathbb{P}_{0}\Big(\vert W(T - \tilde{\tau}) + 3a + R \vert \leq {\varepsilon}^{(1)}\Big).
\end{aligned}
\end{equation}
Hence, by applying the standard Brownian motion estimate~\eqref{Paper02_elementary_gaussian_estimate} from Lemma~\ref{Paper02_aux_results_BM} in the appendix, and the fact that $1 \leq T - \tilde \tau \leq 3$, $0 \leq a \leq c$ and~$\varepsilon^{(1)} < 1$ to~\eqref{Paper02_intermediate_step_propagation_speed_vii}, it follows that
\begin{equation*}
\begin{aligned}
      u_{0}(t + T, x') \geq \frac{(U_{\textrm{eq}}-\delta) {\varepsilon}^{(1)}}{(1 + A)\sqrt{6m\pi}} \exp\left(3H_{\min} -\frac{(3c + R + 1)^{2}}{2m}\right) \geq {\nu}^{*}_{0},
\end{aligned}
\end{equation*}
where for the second inequality we used~\eqref{Paper02_second_condition_spread_mass_constants}. This completes the proof for case~$(2)$.

\medskip

\noindent \underline{Case~$(3)$:} Observe that by the same argument as in case (2),~\eqref{Paper02_intermediate_step_propagation_speed_vi} holds. Now take $y \in [x-1, x + 1]$. By the Feynman--Kac representation in~\eqref{Paper02_feynman_kac_no_mutations}, Lemma~\ref{Paper02_uniform_bound_total_mass}, the definition of $H_{\min}$ in~\eqref{Paper02_min_max_difference_polynomial}, and~\eqref{Paper02_intermediate_step_propagation_speed_vi}, we have
\begin{equation} \label{Paper02_intermediate_step_propagation_speed_viii}
\begin{aligned}
    u_{0}(t + \tilde{\tau} + 1, y + a(\tilde{\tau} + 1)) \geq e^{H_{\min}}\frac{U_{\textrm{eq}}-\delta}{2(1 + A)} \mathbb{P}_{y + a(\tilde{\tau} + 1)}\Big(\vert W(1) - \tilde{y} \vert \leq {\varepsilon}^{(1)}\Big).
\end{aligned}
\end{equation}
Since, by the assumptions of case~(3), $\vert \tilde{y} - (x + a\tilde{\tau}) \vert \leq R$, and since $0 \leq a \leq c$ and $\vert x - y \vert \leq 1$, we have
\begin{equation*}
   \vert \tilde{y} - (y + a(\tilde{\tau} + 1)) \vert \leq c + R + 1,
\end{equation*}
and we conclude from~\eqref{Paper02_intermediate_step_propagation_speed_viii} that
\begin{equation} \label{Paper02_intermediate_estimate_case_3}
    u_{0}(t + \tilde{\tau}+1, y + a(\tilde{\tau} + 1)) \geq e^{H_{\min}}\frac{U_{\textrm{eq}}-\delta}{2(1 + A)}\mathbb{P}_{c+R+1}\Big(\vert W(1) \vert \leq {\varepsilon}^{(1)}\Big).
\end{equation}
By the assumptions of case~(3), for $\tau \in [0, T - \tilde{\tau} - 2]$, if $\vert W(\tau) - (x + a(T-1-\tau)) \vert \leq R$, we have $\| u(t + T - 1 - \tau, W(\tau)) \|_{\ell_{1}} < U_{\textrm{eq}}-\delta$. Hence, by~\eqref{Paper02_level_set_positive_difference_birth_death_rates}, if $\vert W(\tau) - (x + a(T -1 - \tau)) \vert \leq R$ for some $\tau \in [0,T - \tilde{\tau}-2]$, we have
\begin{equation} \label{Paper02_trivial_consequence_level_set_small_z_ii}
    ((1 - \mu)q_{+} - q_{-})(\| u(t + T -1 - \tau, W(\tau)) \|_{\ell_{1}}) \geq H_{\min,\delta}.
\end{equation}
Hence, for $y' \in [x - 1/2, x + 1/2]$, by applying~\eqref{Paper02_trivial_consequence_level_set_small_z_ii} to~\eqref{Paper02_feynman_kac_no_mutations}, we have
\begin{equation} \label{Paper02_intermediate_step_propagation_speed_ix}
\begin{aligned}
    & u_{0}(t + T - 1, y' + a(T-1)) \\ & \quad \geq \inf_{y \in [x-1, x+1]} u_{0}(t + \tilde{\tau} + 1, y + a(\tilde{\tau} + 1)) \exp\Big(H_{\min,\delta}(T - \tilde{\tau} - 2)\Big) \\ & \quad \quad \quad \cdot \mathbb{P}_{y' + a(T -1)}\Big(\vert W(\tau) - (y' + a(T - 1 - \tau)) \vert \leq R -1 \; \forall \tau \leq T - \tilde{\tau} - 2, \\ & \quad \quad \quad \quad \quad \quad \quad \quad \quad \quad \quad \quad \quad \quad \quad \quad \quad \, \vert W(T - \tilde{\tau} - 2) - (y' + a(\tilde{\tau} + 1)) \vert \leq \tfrac{1}{2}\Big).
\end{aligned}
\end{equation}
Since $T - \tilde \tau - 2 \geq 1$ and $0 \leq a \leq c$, by applying the Brownian motion estimate~\eqref{Paper02_tilted_tube_BM} from Lemma~\ref{Paper02_aux_results_BM} in the appendix together with~\eqref{Paper02_intermediate_estimate_case_3} to~\eqref{Paper02_intermediate_step_propagation_speed_ix}, it follows that, for any $y' \in [x - 1/2, x + 1/2]$,
\begin{equation} \label{Paper02_intermediate_step_propagation_speed_x}
\begin{aligned}
    u_{0}(t + T -1, y' + a(T -1)) & \geq e^{H_{\min}}\frac{(U_{\textrm{eq}}-\delta)C^{(2)}_{R-1}}{4(1 + A)}\mathbb{P}_{c+R+1}\Big(\vert W(1) \vert \leq {\varepsilon}^{(1)}\Big) \\
    & \qquad \cdot \exp\left((T - \tilde{\tau} - 2)\left(H_{\min,\delta} - \frac{c^{2}}{2m} - \frac{\pi^{2}m}{8(R - 1)^{2}} \right)\right) \\ & \geq e^{H_{\min}}\frac{(U_{\textrm{eq}}-\delta)C^{(2)}_{R-1}}{4(1 + A)}\mathbb{P}_{c+R+1}\Big(\vert W(1) \vert \leq {\varepsilon}^{(1)}\Big),
\end{aligned}
\end{equation}
where for the second inequality we used the fact that $T - \tilde{\tau} - 2 > 0$ and~\eqref{Paper02_first_condition_spread_mass_constants}. Finally, by the Feynman--Kac representation in~\eqref{Paper02_feynman_kac_no_mutations}, again together with Lemma~\ref{Paper02_uniform_bound_total_mass} and~\eqref{Paper02_min_max_difference_polynomial}, and then using that $x' = x + aT$ and $0 \leq a \leq c$ and using~\eqref{Paper02_intermediate_step_propagation_speed_x}, we have
\begin{equation*}
\begin{aligned}
    & u_{0}(t + T, x') \\ & \quad  \geq e^{H_{\min}} \inf_{\vert y' - x \vert \leq 1/2} u_{0}(t + T -1, y' + a(T -1)) \cdot \mathbb{P}_{x'}\left(\vert W(1) - (x + a(T-1))\vert \leq \tfrac{1}{2} \right) \\ & \quad \geq e^{2H_{\min}}\frac{(U_{\textrm{eq}}-\delta)C^{(2)}_{R-1}}{4(1 + A)}\mathbb{P}_{c+R+1}\Big(\vert W(1) \vert \leq {\varepsilon}^{(1)}\Big) \mathbb{P}_{0}\left(\vert W(1) + c\vert \leq \tfrac{1}{2} \right) \\ & \quad \geq {\nu}^{*}_{0},
\end{aligned}
\end{equation*}
where for the last inequality we used~\eqref{Paper02_third_condition_spread_mass_constants}. This completes the proof in case~$(3)$.

Since in each of cases (1)-(3) we have~$u_{0}(t + T, x') \geq \nu^{*}_{0}$, the proof is complete.
\end{proof}

We are finally ready to prove Theorem~\ref{Paper02_spreading_speed_Fisher_KPP}. Our proof is similar to the proof of~\cite[Theorem~1.1]{penington2018spreading}.

\begin{proof}[Proof of Theorem~\ref{Paper02_spreading_speed_Fisher_KPP}]
    Recall the definition of~$c^{*}$ in~\eqref{Paper02_speed_Fisher_KPP_case_specific_constant}. We will first establish an upper bound on the propagation speed for $u_{0}$, i.e.~we will show that
    \begin{equation} \label{Paper02_upper_bound_speed_u_0}
        \lim_{T \rightarrow \infty} \, \sup_{x \geq c^{*}T} u_{0}(T,x) = 0.
    \end{equation}
     Observe that, by Corollary~\ref{Paper02_same_speed_u_0_total_mass}, if~\eqref{Paper02_upper_bound_speed_u_0} holds, then we must have
     \[
     \lim_{T \rightarrow \infty} \, \sup_{x \geq c^*T} \| u(T,x)\|_{\ell_{1}} = 0.
     \]
     In order to prove~\eqref{Paper02_upper_bound_speed_u_0}, recall from the assumptions of the theorem that there exists~$R > 0$ such that $\supp f_{0} \subseteq (-\infty, R]$, where $f = (f_{k})_{k \in \mathbb{N}_{0}} \in L_{\infty}(\mathbb{R};\ell_{1})$ is the initial condition to the system of PDEs~\eqref{Paper02_PDE_scaling_limit} satisfying Assumptions~\ref{Paper02_assumption_initial_condition} and~\ref{Paper02_assumption_initial_condition_modified}. Recall the definition of~$H_{\max}$ in~\eqref{Paper02_min_max_difference_polynomial}, and observe that
     \begin{equation} \label{Paper02_relation_c_star_H_max}
         c^{*} = ({2mH_{\max}})^{1/2}.
     \end{equation}
     Recall that since the reaction term $F = (F_k)_{k \in \mathbb N_0}$ is of Fisher-KPP type in the sense of Definition~\ref{Paper02_assumption_monostable_condition}, we have that~\eqref{Paper02_min_max_difference_polynomial_ii} holds. Therefore, by the Feynman-Kac representation in~\eqref{Paper02_feynman_kac_no_mutations_from_begining}, and by Lemma~\ref{Paper02_uniform_bound_total_mass}, we have that for any~$T \geq 0$ and any~$x \geq c^{*}T$,
     \begin{equation} \label{Paper02_upper_bound_speed_u_0_step_i}
    \begin{aligned}
        u_{0}\left(T, x\right) & \leq e^{H_{\max}{T}}\mathbb{E}_{x}\left[f_{0}(W(T))\right] \\ & \leq e^{H_{\max}{T}} \| f \|_{L_{\infty}(\mathbb{R};\ell_{1})} \mathbb{P}_{0}(W(T) \geq c^{*}T - R).
    \end{aligned}
     \end{equation}
     Observe that by the standard Gaussian tail estimate~\eqref{Paper02_elementary_gaussian_estimate_ii} from Lemma~\ref{Paper02_aux_results_BM} in the appendix, for $T > R/c^*$,
     \begin{equation*}
     \begin{aligned}
         \mathbb{P}_{0}(W(T) \geq c^{*}T - R) \leq \frac{(mT)^{1/2}}{(c^*T - R) \sqrt{2\pi}} \exp\left(- \frac{(c^*T-R)^2}{2mT}\right),
     \end{aligned}
     \end{equation*}
     and so for $T > R/c^*$,~\eqref{Paper02_upper_bound_speed_u_0_step_i} combined with~\eqref{Paper02_relation_c_star_H_max} yields
     \begin{equation} \label{Paper02_upper_bound_speed_u_0_step_ii}
    \begin{aligned}
        \sup_{x \geq c^*T} u_{0}\left(T, x\right) & \leq \|f\|_{L_\infty(\mathbb R; \ell_1)} \frac{(mT)^{1/2}}{(c^*T - R) \sqrt{2\pi}} \exp\left(\frac{(c^*)^2T}{2m} - \frac{(c^{*}T - R)^{2}}{2mT}\right) \\ & = \|f\|_{L_\infty(\mathbb R; \ell_1)} \frac{(mT)^{1/2}}{(c^*T - R)\sqrt{2\pi}} \exp\left(\frac{c^{*}R}{m} - \frac{R^{2}}{2mT}\right).
    \end{aligned}
     \end{equation}
     Taking $T \rightarrow \infty$ on both sides of~\eqref{Paper02_upper_bound_speed_u_0_step_ii}, we conclude that~\eqref{Paper02_upper_bound_speed_u_0} holds.

     For the lower bound on the spreading speed, by Theorem~\ref{Paper02_prop_control_proportions}, it is enough to establish that there exists~$\nu_{0} > 0$ such that for any~$\varepsilon \in (0,c^*)$, 
     \begin{equation} \label{Paper02_lower_bound_propagation_speed_u_0}
         \liminf_{T \rightarrow \infty} \; \inf_{x \in [0, (c^{*} - \varepsilon)T]} \; u_{0}(T,x) \geq \nu_{0}.
     \end{equation}
     To prove~\eqref{Paper02_lower_bound_propagation_speed_u_0}, we use the same argument as in the proof of the lower bound on the spreading speed in~\cite[Theorem~1.1]{penington2018spreading}, replacing Lemmas~2.3 and~2.4 in~\cite{penington2018spreading} by Lemmas~\ref{Paper02_exponential_growth} and~\ref{Paper02_propagation_mass_u_0} respectively. Since the arguments are the same, we omit the proof and refer the interested reader to~\cite{penington2018spreading}.
\end{proof}

\subsection{Tracer dynamics}
\label{Paper02_tracer_dynamics_section}

In this subsection, we will prove Theorem~\ref{Paper02_thm_tracer_dynamics_no_gene_surfing}. Throughout this subsection, let~$u$ denote the continuous mild solution to the system of PDEs~\eqref{Paper02_PDE_scaling_limit} given in Proposition~\ref{Paper02_smoothness_mild_solution}. We will assume that the reaction term~$F = (F_{k})_{k \in \mathbb{N}_{0}}$ defined in~\eqref{Paper02_reaction_term_PDE} is monostable in the sense of Definition~\ref{Paper02_assumption_monostable_condition}, and that the initial condition~$f = (f_{k})_{k \in \mathbb{N}_{0}}$ of the system of PDEs~\eqref{Paper02_PDE_scaling_limit} satisfies both Assumptions~\ref{Paper02_assumption_initial_condition} and~\ref{Paper02_assumption_initial_condition_modified}. Recall the system of PDEs for the labelled population~\eqref{Paper02_PDE_system_labelled_particles}, and suppose~$f^{*} = (f^{*}_{k})_{k \in \mathbb{N}_{0}} \in L_{\infty}(\mathbb{R}; \ell_{1})$ satisfies Assumption~\ref{Paper02_assumption_initial_condition_labelled_particles}, recalling that~$f^*$ is the initial condition of the system of PDEs~\eqref{Paper02_PDE_system_labelled_particles}. Recall that~$F^{*} = (F^{*}_{k})_{k \in \mathbb{N}_{0}}: \ell_{1}^{+} \times \ell_{1} \rightarrow \ell_{1}$ is the reaction term defined in~\eqref{Paper02_reaction_term_PDE_labelled}. We start by introducing the definition of mild solutions to the system of PDEs~\eqref{Paper02_PDE_system_labelled_particles} analogously to the definition of a mild solution of~\eqref{Paper02_PDE_scaling_limit} in Definition~\ref{Paper02_definition_mild_solution}.

\begin{definition}
    Suppose that $(s_k)_{k \in \mathbb N_0}$, $q_+$ and $q_-$ satisfy Assumptions~\ref{Paper02_assumption_fitness_sequence} and~\ref{Paper02_assumption_polynomials}, that $\mu \in (0,1)$ and $m > 0$, and that $f$ and $f^*$ satisfy Assumption~\ref{Paper02_assumption_initial_condition_labelled_particles}, and that the reaction term $F = (F_{k})_{k \in \mathbb{N}_{0}}$ defined in~\eqref{Paper02_reaction_term_PDE} is monostable in the sense of Definition~\ref{Paper02_assumption_monostable_condition}. Let $u = (u_k)_{k \in \mathbb N_0}: [0, \infty) \times \mathbb R \rightarrow \ell_1^+$ be the continuous mild solution to~\eqref{Paper02_PDE_scaling_limit} given in Proposition~\ref{Paper02_smoothness_mild_solution}. We say that a Lebesgue-measurable function $u^* = (u^*_k)_{k \in \mathbb N_0}: [0, \infty) \times \mathbb R \rightarrow \ell_1^+$ is a \emph{mild solution} to the system of PDEs~\eqref{Paper02_PDE_system_labelled_particles} if and only if the following conditions are satisfied:
    \begin{enumerate}[(i)]
        \item For $\lambda$-almost every $(T,x) \in [0,\infty) \times \mathbb R$, for $\lambda$-almost every $t \in [0,T]$,
        \[
        \Big(P_{T-t} \| F^*(u(t,\cdot), u^*(t,\cdot)) \|_{\ell_1}\Big)(x) < \infty.
        \]
        \item For $\lambda$-almost every $(T,x) \in [0, \infty) \times \mathbb R$,
        \[
        u^*(T,x) = (P_Tf^*)(x) + \int_0^T \Big(P_{T-t} F^*(u(t,\cdot), u^*(t,\cdot)) \Big)(x) \, dt.
        \]
    \end{enumerate}
\end{definition}

Our first step towards the proof of Theorem~\ref{Paper02_thm_tracer_dynamics_no_gene_surfing} will be to state existence, uniqueness and regularity properties for solutions of~\eqref{Paper02_PDE_system_labelled_particles}.

\begin{proposition}
\label{Paper02_basic_properties_linear_parabolic_pdes_tracer}
   Under the assumptions of Theorem~\ref{Paper02_thm_tracer_dynamics_no_gene_surfing}, letting $u$ denote the continuous mild solution to the system of PDEs~\eqref{Paper02_PDE_scaling_limit} given in Proposition~\ref{Paper02_smoothness_mild_solution}, there exists a unique function~$u^{*} = (u^{*}_{k})_{k \in \mathbb{N}_{0}} : [0,\infty) \times \mathbb{R} \rightarrow \ell_{1}^{+}$ such that:
    \begin{enumerate}[(i)]
        \item For any $k \in \mathbb N_0$ and $(T,x) \in [0,\infty) \times \mathbb R$,
        \begin{equation} \label{Paper02_mild_formulation_labelled_particles}
            u_k^*(T,x) = (P_Tf_k^*)(x) + \int_0^T \Big(P_{T-t} F_k^*(u(t,\cdot), u^*(t,\cdot))\Big)(x) \, dt.
        \end{equation}
        \item $\sup_{t \in [0,T]} \| u^*(t,\cdot) \|_{L_{\infty}(\mathbb R; \ell_1)} < \infty$ for all $T > 0$.
        \item $u^* \in \mathscr C((0,\infty) \times \mathbb R; \ell_1)$.
        \item $u^{*}_{k} \in \mathscr{C}^{1,2}((0, \infty) \times \mathbb{R}; \mathbb R)$ for every $k \in \mathbb N_0$.
    \end{enumerate}
    Moreover,~$u^{*} = (u^{*}_{k})_{k \in \mathbb{N}_{0}} : [0,\infty) \times \mathbb{R} \rightarrow \ell_{1}^{+}$ satisfies the following inequality: for any~$k \in \mathbb{N}_{0}$,~$T \geq 0$ and~$x \in \mathbb{R}$,
    \begin{equation} \label{Paper02_simple_inequality_comparing_labelled_original_system_PDEs}
        0 \leq u^{*}_{k}(T,x) \leq u_{k}(T,x).
    \end{equation}
\end{proposition}

Since the proof of Proposition~\ref{Paper02_basic_properties_linear_parabolic_pdes_tracer} follows from standard arguments, we postpone it until Section~\ref{Paper02_appendix_auxiliary_PDE_results} in the appendix.

In the remainder of this subsection, we let~$u^{*} = (u^{*}_{k})_{k \in \mathbb{N}_{0}}$ denote the unique solution of~\eqref{Paper02_PDE_system_labelled_particles} satisfying the conditions of Proposition~\ref{Paper02_basic_properties_linear_parabolic_pdes_tracer}. Our next step will be to write~$u^{*}_{k}$ in terms of a Feynman--Kac representation for every~$k \in \mathbb{N}_{0}$.

\begin{lemma} \label{Paper02_feynman_kac_lemma_labelled_population}
    Under the assumptions of Theorem~\ref{Paper02_thm_tracer_dynamics_no_gene_surfing}, for every $k \in \mathbb{N}_{0}$, any~$T \geq 0$, and any~$x \in \mathbb{R}$,
    \begin{equation*}
\begin{aligned}
    u_{k}^{*}(T,x) = & \mathbb{E}_{x}\left[f^{*}_{k}(W(T))\exp\left(\int_{0}^{T} \Big(s_{k} (1-\mu)q_{+} - q_{-}\Big) \Big(\| u(T-t,W(t))\|_{\ell_{1}}\Big) dt\right)\right] \\ & + \mathds{1}_{\{k \geq 1\}}\mathbb{E}_{x}\Bigg[\int_{0}^{T} \Big(s_{k-1}\mu u_{k-1}^{*}q_{+}(\| u \|_{\ell_{1}})\Big)(T-t, W(t)) \\ & \quad \quad \quad \quad\quad \quad\cdot \exp\left(\int_{0}^{t} \Big(s_{k} (1-\mu)q_{+} - q_{-}\Big) \Big(\| u(T-\tau,W(\tau))\|_{\ell_{1}}\Big) d\tau\right)\; dt\Bigg].
\end{aligned}
\end{equation*}
\end{lemma}

\begin{proof}
    Similarly to the proof of Lemma~\ref{Paper02_feynman_kac_representation_lemma_each_u_k}, the result follows directly from Proposition~\ref{Paper02_general_version_feynman_kac}, Proposition~\ref{Paper02_basic_properties_linear_parabolic_pdes_tracer}, Lemma~\ref{Paper02_uniform_bound_total_mass} and the definition of the reaction term~$F^{*} = (F^{*}_{k})_{k \in \mathbb{N}_{0}}$ in~\eqref{Paper02_reaction_term_PDE_labelled}.
\end{proof}

Our next result establishes an upper bound on the propagation of the labelled population that will be pivotal in the proof of Theorem~\ref{Paper02_thm_tracer_dynamics_no_gene_surfing}.
Before stating this result, we recall the sequence~$(\hat{\pi}_{k})_{k \in \mathbb{N}_{0}} \in \ell_1^+$ defined in Assumption~\ref{Paper02_assumption_initial_condition_modified}(iii), and the sequence of functions~$\phi = (\phi_{k})_{k \in \mathbb{N}}: [0, \infty) \rightarrow \ell_{1}^{+}$ defined in Lemma~\ref{Paper02_evolution_proportions_upper_bound_characterisation}. We also recall the definition of~$\mathfrak{Q}_{\min}$ and~$\mathfrak{Q}_{\max}$ in~\eqref{Paper02_definition_max_min_birth_polynomial} and that the sequence of fitness parameters~$(s_{k})_{k \in \mathbb{N}_{0}}$ satisfies Assumption~\ref{Paper02_assumption_fitness_sequence}.

\begin{lemma} \label{Paper02_prop_pop_with_mutations_die_out}
    Suppose that the assumptions of Theorem~\ref{Paper02_thm_tracer_dynamics_no_gene_surfing} hold and that $f^{*}_{0} \equiv 0$. Then $u^{*}_{0}(T, \cdot) \equiv 0$ for any~$T \geq 0$. Moreover, for every~$k \in \mathbb{N}$, any $T \geq 0$ and any~$x \in \mathbb{R}$,
    \begin{equation} \label{Paper02_prop_pop_with_mutations_die_out_ii}
       u^{*}_{k}(T,x) \leq \Big(\hat{\pi}_{k}\exp\Big(-T\mathfrak{Q}_{\min}(1 - s_{k})(1 - \mu)\Big) + \phi_{k}(T)\Big)u_{0}(T,x).
    \end{equation}
\end{lemma}

\begin{proof}
    The fact that~$u^{*}_{0}(T,\cdot) \equiv 0$ for any~$T \geq 0$ follows directly from Lemma~\ref{Paper02_feynman_kac_lemma_labelled_population} and the assumption that~$f^{*}_{0} \equiv 0$. For the proof of~\eqref{Paper02_prop_pop_with_mutations_die_out_ii}, we will use an induction argument on~$k \in \mathbb{N}$. Since the proof of~\eqref{Paper02_prop_pop_with_mutations_die_out_ii} is similar to the proof of the upper bound in Lemma~\ref{Paper02_uniform_evolution_proportion_over_space_fkpp_case}, we will omit some of the details. For~$k = 1$, since~$u^{*}_{0} \equiv 0$, by the Feynman--Kac representation given in Lemma~\ref{Paper02_feynman_kac_lemma_labelled_population}, and by Assumptions~\ref{Paper02_assumption_initial_condition_modified}(iii) and~\ref{Paper02_assumption_initial_condition_labelled_particles}(iii), we have that for any~$T \geq 0$ and any~$x \in \mathbb{R}$,
    \begin{equation*}
    \begin{aligned}
          u^{*}_{1}(T,x) & \leq \hat{\pi}_{1} \mathbb{E}_{x}\left[f_{0}(W(T))\exp\left(\int_{0}^{T} \Big(s_{1} (1-\mu)q_{+} - q_{-}\Big) \Big(\| u(T-t,W(t))\|_{\ell_{1}}\Big) dt\right)\right] \\ & \leq \hat{\pi}_{1}\exp\Big(-T\mathfrak{Q}_{\min}(1 - s_{1})(1 - \mu)\Big)u_{0}(T,x),
    \end{aligned}
    \end{equation*}
    where for the second inequality we used~\eqref{Paper02_first_step_control_proportion_density_one_mutation_PTW}. Hence,~\eqref{Paper02_prop_pop_with_mutations_die_out_ii} holds for $k = 1$. Now take~$k \in \mathbb{N}$ and suppose that~\eqref{Paper02_prop_pop_with_mutations_die_out_ii} holds for all~$T \geq 0$ and~$x \in \mathbb{R}$. Then by Lemma~\ref{Paper02_feynman_kac_lemma_labelled_population}, Assumptions~\ref{Paper02_assumption_initial_condition_modified}(iii) and~\ref{Paper02_assumption_initial_condition_labelled_particles}(iii), and our assumption that~\eqref{Paper02_prop_pop_with_mutations_die_out_ii} holds, and using~\eqref{Paper02_trivial_conclusion_Q_max_Q_min} and Lemma~\ref{Paper02_feynman_kac_representation_lemma_each_u_k}, we have for any~$T \geq 0$ and any~$x \in \mathbb{R}$,
    \begin{equation} \label{Paper02_prop_pop_with_mutations_die_out_iii}
    \begin{aligned}
        & u^{*}_{k+1}(T,x) \\ & \; \leq \hat{\pi}_{k+1} \mathbb{E}_{x}\left[f_{0}(W(T))\exp\left(\int_{0}^{T} \Big(s_{k+1} (1-\mu)q_{+} - q_{-}\Big) \Big(\| u(T-t,W(t))\|_{\ell_{1}}\Big) dt\right)\right] \\ & \quad \; + \mu \mathfrak{Q}_{\max} s_{k} \hat{\pi}_{k}u_{0}(T,x)\int_{0}^{T} \exp\Big(-t\mathfrak{Q}_{\min}(1 - s_{k+1})(1 - \mu) - (T-t)\mathfrak{Q}_{\min}(1 - s_{k})(1 - \mu)\Big) \, dt \\ & \quad \; + \mu \mathfrak{Q}_{\max} s_{k} u_{0}(T,x)\int_{0}^{T} \exp\Big(-t\mathfrak{Q}_{\min}(1 - s_{k+1})(1 - \mu)\Big) \phi_{k}(T-t) \, dt.
    \end{aligned}
    \end{equation}
    Using~\eqref{Paper02_intermediate_step_upper_bound_prevalence_mutations_iii} to bound the first term on the right-hand side of~\eqref{Paper02_prop_pop_with_mutations_die_out_iii}, and~\eqref{Paper02_intermediate_step_upper_bound_prevalence_mutations_vii} to bound the sum of the second and third terms on the right-hand side of~\eqref{Paper02_prop_pop_with_mutations_die_out_iii}, we conclude that~\eqref{Paper02_prop_pop_with_mutations_die_out_ii} holds for~$u^{*}_{k+1}$. Therefore, by induction, our claim holds.
\end{proof}

We are now ready to prove Theorem~\ref{Paper02_thm_tracer_dynamics_no_gene_surfing}.

\begin{proof}[Proof of Theorem~\ref{Paper02_thm_tracer_dynamics_no_gene_surfing}]
    The claimed existence, uniqueness and regularity properties of~$u^{*}$ follow from Proposition~\ref{Paper02_basic_properties_linear_parabolic_pdes_tracer}. To complete the proof, it remains to show that~$\| u^{*}(T, x) \|_{\ell_{1}}$ converges uniformly (in space) to~$0$ as $T \rightarrow \infty$. Since~$(\hat{\pi}_{k})_{k \in \mathbb{N}} \in \ell_{1}^{+}$ by Assumption~\ref{Paper02_assumption_initial_condition_modified}(iii), and since~$(s_{k})_{k \in \mathbb{N}_{0}}$ is monotonically decreasing with~$s_{1} < 1$ by Assumption~\ref{Paper02_assumption_fitness_sequence} and Definition~\ref{Paper02_assumption_monostable_condition}, by summing both sides of~\eqref{Paper02_prop_pop_with_mutations_die_out_ii} in Lemma~\ref{Paper02_prop_pop_with_mutations_die_out} and then applying Lemmas~\ref{Paper02_uniform_bound_total_mass} and~\ref{Paper02_evolution_proportions_upper_bound_characterisation}, we conclude that for any~$T \geq 0$,
    \begin{equation} \label{Paper02_intermediate_step_proof_no_gene_surfing}
        \| u^{*}(T, \cdot) \|_{L_{\infty}(\mathbb{R}; \ell_{1})} \leq \| \hat{\pi} \|_{\ell_{1}} \exp\Big(-T\mathfrak{Q}_{\min}(1 - s_{1})(1 - \mu)\Big) + \Phi(T),
    \end{equation}
    where~$\Phi(T) = \sum_{k \in \mathbb N} \phi_{k}(T)$. By taking~$T \rightarrow \infty$ on both sides of~\eqref{Paper02_intermediate_step_proof_no_gene_surfing} and then using~\eqref{Paper02_vanishing_sum_error_term} from Lemma~\ref{Paper02_evolution_proportions_upper_bound_characterisation}, we complete the proof. 
\end{proof}

\appendix

\section{Appendix} \label{Paper02_appendix_sec}

\subsection{Brownian motion and Gaussian estimates} \label{Paper02_appendix_BM_gaussian}

 In this subsection, we collect some Brownian motion estimates. For $(t,x) \in (0, \infty) \times \mathbb R$, recall the definition of the Gaussian kernel $p(t,x)$ in~\eqref{Paper02_Gaussian_kernel}. We start by stating bounds on the spatial and temporal increments of~$p$.

\begin{lemma} \label{Paper02_gaussian_estimates_standard}
    Suppose $m > 0$, and let $p:(0, \infty) \times \mathbb R \rightarrow (0, \infty)$ be defined as in~\eqref{Paper02_Gaussian_kernel}. The following estimates hold:
    \begin{enumerate}[(i)]
    \item There exists $C^{(1)} = C^{(1)}(m) > 0$ such that for all $t > 0$ and $x,y \in \mathbb R$,
    \begin{equation} \label{Paper02_gaussian_spatial_increment}
    \left\vert p(t, x) - p(t, y) \right\vert \leq \frac{C^{(1)} \vert x - y \vert}{\sqrt{t}} (p(2t,x) + p(2t,y)).
    \end{equation}
    \item There exists $C^{(2)} = C^{(2)}(m)>0$ such that for all $t > 0$ and $t' \geq 0$,
    \begin{equation} \label{Paper02_gaussian_temporal_increment}
    \int_\mathbb R \vert p(t + t',x) - p(t,x) \vert \, dx \leq C^{(2)} \sqrt{\frac{t'}{t}}.
    \end{equation}
    \end{enumerate}
\end{lemma}

\begin{proof}
    Estimate~\eqref{Paper02_gaussian_spatial_increment} is standard (see~e.g.~\cite[Lemma~7.7]{etheridge2023looking}). For estimate~\eqref{Paper02_gaussian_temporal_increment}, the case $t' = 0$ is trivial. For $t' > 0$, observe that by the semigroup property of the Gaussian kernel, we have for all $t,t' > 0$,
    \[
    p(t + t',x) = \int_{\mathbb R} p(t,x-y)p(t',y) \, dy \quad \forall x \in \mathbb R.
    \]
    Therefore for $t,t' >0$,
    \begin{equation*}
    \begin{aligned}
        \int_\mathbb R \vert p(t + t',x) - p(t,x) \vert \, dx & = \int_\mathbb R \left\vert \int_{\mathbb R} p(t,x-y)p(t',y) \, dy - p(t,x) \right\vert  \, dx \\
        & \leq \int_{\mathbb R} \int_\mathbb R \vert p(t,x-y) - p(t,x) \vert p(t',y) \, dy \, dx \\
        & \leq \frac{C^{(1)}}{\sqrt{t}} \int_\mathbb R \int_\mathbb R \vert y \vert p(t',y) (p(2t,x-y) + p(2t,x)) \, dy \, dx \\
        & = \frac{C^{(1)}}{\sqrt{t}} \int_\mathbb R \int_\mathbb R \vert y \vert p(t',y) (p(2t,x-y) + p(2t,x)) \, dx \, dy \\
        & = \frac{2C^{(1)}}{\sqrt{t}} \int_\mathbb R \vert y \vert p(t',y) \, dy,
    \end{aligned}
    \end{equation*}
    where for the third line we used estimate~\eqref{Paper02_gaussian_spatial_increment}, and for the fourth line we used Fubini's theorem. Using the elementary identity
    \[
    \int_\mathbb R \vert y \vert p(t',y) \, dy = \sqrt{\frac{2mt'}{\pi}},
    \]
    we then conclude that
    \begin{equation*}
    \begin{aligned}
        \int_\mathbb R \vert p(t + t',x) - p(t,x) \vert \, dx \leq 2C^{(1)} \sqrt{\frac{2mt'}{\pi t}}.
    \end{aligned}
    \end{equation*}
    Therefore~\eqref{Paper02_gaussian_temporal_increment} holds. This completes the proof.
\end{proof}
 
Next, we collect some results that are straightforward modifications of the Brownian motion estimates in~\cite{morters2010brownian,roberts2015fine,penington2018spreading}. We refer the reader interested in the proof of each estimate to the indicated reference.

\begin{lemma} \label{Paper02_aux_results_BM}
Let~$(W(t))_{t \geq 0}$ be a Brownian motion run at speed~$m > 0$, with $W(0) = 0$ under $\mathbb P_0$. Then the following estimates are satisfied:
     \begin{enumerate}[(i)]
        \item \cite[Lemma~12.9]{morters2010brownian} For any~$T > 0$ and~$x > 0$,
        \begin{equation} \label{Paper02_elementary_gaussian_estimate_ii}
            \mathbb{P}_{0}(W(T) > x) \leq \frac{(mT)^{1/2}}{x\sqrt{2\pi}} \exp\left(- \frac{x^{2}}{2mT}\right).
        \end{equation}
        
         \item \cite[Equation~(12)]{penington2018spreading} For any $0 \leq x \leq y$ and $t > 0$,
         \begin{equation} \label{Paper02_elementary_gaussian_estimate}
             \frac{y-x}{\sqrt{2 \pi m t}} \exp\left(- \frac{y^{2}}{2mt}\right) \leq \mathbb{P}_{0}(W(t) \in [x,y]) \leq \frac{y-x}{\sqrt{2 \pi m t}} \exp\left(- \frac{x^{2}}{2mt}\right).
         \end{equation}
         
         \item \cite[Lemma~5]{roberts2015fine} For $T > 0$, there exists $C^{(1)}_T > 0$ such that for any $t \geq T$ and $0 \leq R \leq 1$,
    \begin{equation} \label{Paper02_roberts_lemma}
    \mathbb{P}_{0}\Big(\vert W(\tau) \vert \leq 1\; \forall \tau \leq t, \; \vert W(t) \vert \leq R\Big) \geq C^{(1)}_{T} \exp\left(-\frac{\pi^{2}mt}{8}\right) \min\left(R, \tfrac{2}{3}\right).
    \end{equation}
    \item 
    \cite[Equation~(19)]{penington2018spreading} For $R > 0$ and $c^* > 0$, there exists~$C^{(2)}_{R} = C^{(2)}_{R}(c^*) > 0$ such that for $b \in \mathbb{R}$ with $\vert b \vert \leq c^{*}$, $r \in [0,R]$ and $t \geq 1$,
    \begin{equation} \label{Paper02_tilted_tube_BM}
         \mathbb{P}_{0}\Big(\vert W(\tau) - b\tau\vert \leq R\; \forall \tau \leq t, \; \vert W(t) - bt\vert \leq r\Big) \geq C^{(2)}_{R}r\exp\left(-\frac{b^{2}t}{2m} - \frac{\pi^{2}mt}{8R^{2}}\right).
    \end{equation}
     \end{enumerate}
\end{lemma}

Next, we prove that, analogously to its action on real-valued functions, the Gaussian kernel has a smoothing effect on $\ell_1$-valued maps. Recall the definition of the Brownian semigroup $\{P_t\}_{t \geq 0}$ in~\eqref{Paper02_semigroup_BM_action_ell_1_functions}. 

\begin{lemma}\label{Paper02lem:heat_continuity_ell1}
Let $T>0$. Let
\(f:\mathbb R\to\ell_1^+\) and \(g:[0,T]\times\mathbb R\to\ell_1\) be such that
\[
f\in L_\infty(\mathbb R;\ell_1) \quad \textrm{and} \quad g\in L_\infty([0,T]\times\mathbb R;\ell_1).
\]
For \((t,x)\in[0,T]\times\mathbb R\), define
\begin{equation} \label{Paper02_definition_strong_mild}
v(t,x):=(P_t f)(x)+\int_0^t (P_{t-\tau}g(\tau,\cdot))(x)\,d\tau.
\end{equation}
Then, for all $\delta \in (0,T)$, there exists $C_{\delta,T,f,g} > 0$ such that for all $t_1,t_2 \in [\delta,T]$ and $x_1,x_2 \in \mathbb R$,
\begin{equation} \label{Paper02_pre_step_holder}
    \| v(t_1,x_1) - v(t_2,x_2)\|_{\ell_1} \leq C_{\delta,T,f,g} \left(\vert x_1 - x_2 \vert + \vert t_1 - t_2\vert^{1/2} \right).
\end{equation}
In particular, $v\in\mathscr C((0,T]\times\mathbb R;\ell_1)$.
\end{lemma}

\begin{proof}
    We prove~\eqref{Paper02_pre_step_holder} by controlling the spatial and temporal increments separately. Take $\delta \in (0,T)$, and then take $t_1,t_2 \in [\delta,T]$ and $x_1,x_2 \in \mathbb R$. The triangle inequality yields
    \begin{equation} \label{Paper02_smooth_heat_kernel_i}
    \| v(t_1,x_1) - v(t_2,x_2) \|_{\ell_1} \leq \| v(t_1,x_1) - v(t_1,x_2) \|_{\ell_1} + \| v(t_1,x_2) - v(t_2,x_2) \|_{\ell_1}.
    \end{equation}
    We now bound the terms on the right-hand side of~\eqref{Paper02_smooth_heat_kernel_i}. Recall the definition of the Gaussian kernel $p: (0, \infty) \times \mathbb R \rightarrow (0,\infty)$ in~\eqref{Paper02_Gaussian_kernel}. For the first term on the right-hand side of~\eqref{Paper02_smooth_heat_kernel_i}, by using~\eqref{Paper02_definition_strong_mild} and~\eqref{Paper02_semigroup_BM_action_ell_1_functions}, and then by Jensen's inequality,
    \begin{equation*}
    \begin{aligned}
        & \| v(t_1,x_1) - v(t_1,x_2) \|_{\ell_1} \\ & \quad \leq \int_{\mathbb R} \vert p(t_1,x_1 - y) - p(t_1,x_2-y) \vert   \cdot \| f(y)  \|_{\ell_1} \, dy \\
        & \qquad \; + \int_0^{t_1} \int_\mathbb R  \vert p(t_1 - \tau,x_1 - y) - p(t_1 - \tau,x_2-y) \vert \cdot \| g(\tau, y)  \|_{\ell_1} \, dy \, d\tau \\
        & \quad \leq \| f  \|_{L_\infty(\mathbb R;\ell_1)} \int_{\mathbb R} \vert p(t_1,x_1 - y) - p(t_1,x_2-y) \vert \, dy \\
        & \qquad \;+ \| g  \|_{L_{\infty}([0,T] \times \mathbb R;\ell_1)} \int_0^{t_1} \int_\mathbb R  \vert p(t_1 - \tau,x_1 - y) - p(t_1 - \tau,x_2-y) \vert \, dy \, d\tau,
    \end{aligned}
    \end{equation*}
    where for the second estimate we used the fact that $f \in L_{\infty}(\mathbb R; \ell_1)$ and $g \in L_{\infty}([0,T] \times \mathbb R; \ell_1)$. By applying estimate~\eqref{Paper02_gaussian_spatial_increment} from Lemma~\ref{Paper02_gaussian_estimates_standard} and using the fact that~$t_1 \in [\delta, T]$, we conclude that there exists $C^{(1)} = C^{(1)}(m) > 0$ such that
    \begin{equation} \label{Paper02_smooth_heat_kernel_iii}
    \begin{aligned}
         & \| v(t_1,x_1) - v(t_1,x_2) \|_{\ell_1} \\
         & \quad \leq C^{(1)} \| f  \|_{L_\infty(\mathbb R;\ell_1)} \frac{\vert x_1 - x_2 \vert}{\sqrt{t_1}} \int_{\mathbb R} (p(2t_1,x_1 - y) + p(2t_1,x_2 - y)) \, dy \\
         & \qquad \; + C^{(1)}\| g  \|_{L_{\infty}([0,T] \times \mathbb R;\ell_1)} \int_0^{t_1} \frac{\vert x_1 - x_2 \vert }{\sqrt{t_1 - \tau}} \int_{\mathbb R} (p(2(t_1-\tau),x_1 - y) + p(2(t_1-\tau),x_2 - y)) \, dy \, d\tau \\
         & \quad \leq C^{(1)} \left(\frac{2\| f  \|_{L_\infty(\mathbb R;\ell_1)}}{\sqrt{\delta}} + 4 \| g  \|_{L_{\infty}([0,T] \times \mathbb R;\ell_1)} \sqrt{T}\right) \vert x_1 - x_2 \vert.
    \end{aligned}
    \end{equation}

    For the second term on the right-hand side of~\eqref{Paper02_smooth_heat_kernel_i}, by using~\eqref{Paper02_definition_strong_mild},~\eqref{Paper02_semigroup_BM_action_ell_1_functions} and~\eqref{Paper02_Gaussian_kernel}, and then by Jensen's inequality, we have
    \begin{equation*}
    \begin{aligned}
        & \| v(t_1,x_2) - v(t_2,x_2) \|_{\ell_1} \\ & \quad \leq \int_{\mathbb R} \vert p(t_1,x_2 - y) - p(t_2,x_2-y) \vert   \cdot \| f(y)  \|_{\ell_1} \, dy \\
        & \qquad \; + \int_{0}^{t_1 \wedge t_2} \int_\mathbb R \vert p(t_1 - \tau,x_2 - y) - p(t_2 - \tau,x_2-y) \vert \| g(\tau, y)  \|_{\ell_1} \, dy \, d\tau \\
         & \qquad \; + \int_{t_1 \wedge t_2}^{t_1 \vee t_2} \int_\mathbb R p((t_1 \vee t_2) - \tau, x_2 - y) \| g(\tau, y)  \|_{\ell_1} \, dy \, d\tau
        \\
        & \quad \leq \vert t_1 - t_2 \vert \cdot \| g  \|_{L_{\infty}([0,T] \times \mathbb R;\ell_1)} +  \| f  \|_{L_\infty(\mathbb R;\ell_1)}  \int_{\mathbb R} \vert p(t_1,x_2 - y) - p(t_2,x_2-y) \vert \, dy \\
        & \qquad \;+ \| g  \|_{L_{\infty}([0,T] \times \mathbb R;\ell_1)} \int_{0}^{t_1 \wedge t_2} \int_\mathbb R \vert p(t_1 - \tau,x_2 - y) - p(t_2 - \tau,x_2-y) \vert \, dy \, d\tau.
    \end{aligned}
    \end{equation*}
    By applying estimate~\eqref{Paper02_gaussian_temporal_increment} from Lemma~\ref{Paper02_gaussian_estimates_standard} and using the fact that we chose $t_1,t_2$ such that $t_1,t_2 \in [\delta, T]$, it follows that there exists $C^{(2)} = C^{(2)}(m) > 0$ such that
    \begin{equation} \label{Paper02_smooth_heat_kernel_v}
    \begin{aligned}
         & \| v(t_1,x_2) - v(t_2,x_2) \|_{\ell_1} \\ & \quad \leq \sqrt{T} \| g  \|_{L_{\infty}([0,T] \times \mathbb R;\ell_1)} \vert t_1 - t_2 \vert^{1/2} + C^{(2)} \| f  \|_{L_\infty(\mathbb R;\ell_1)} \sqrt{\frac{\vert t_1 - t_2 \vert}{t_1 \wedge t_2}} \\
         & \qquad \; + C^{(2)}  \| g  \|_{L_{\infty}([0,T] \times \mathbb R;\ell_1)} \int_0^{t_1 \wedge t_2} \sqrt{\frac{\vert t_1 - t_2 \vert}{(t_1 \wedge t_2) - \tau}} \, d\tau \\
         & \quad \leq \left(\sqrt{T}(1 + 2C^{(2)}) \| g  \|_{L_{\infty}([0,T] \times \mathbb R;\ell_1)} + \frac{C^{(2)} \| f  \|_{L_\infty(\mathbb R;\ell_1)}}{\sqrt{\delta}} \right) \vert t_1 - t_2 \vert^{1/2}.
    \end{aligned}
    \end{equation}
    Estimate~\eqref{Paper02_pre_step_holder} then follows from applying~\eqref{Paper02_smooth_heat_kernel_iii} and~\eqref{Paper02_smooth_heat_kernel_v} to~\eqref{Paper02_smooth_heat_kernel_i}. Since~\eqref{Paper02_pre_step_holder} implies that $v \in \mathscr C((0,T] \times \mathbb R; \ell_1)$, the proof is complete.
\end{proof}

\subsection{Proof of Proposition~\ref{Paper02_basic_properties_linear_parabolic_pdes_tracer}} \label{Paper02_appendix_auxiliary_PDE_results}

In this subsection, we prove Proposition~\ref{Paper02_basic_properties_linear_parabolic_pdes_tracer}
using a standard Picard--Lindel\"of iteration argument. Recall the definition of $L_{\infty}(\mathbb  R; \ell_1)$ from the beginning of Section~\ref{Paper02_PDES_sequence_spaces}. Since~$\ell_{1}$ is a Banach space, $L_{\infty}(\mathbb{R}; \ell_{1})$ is also a Banach space when equipped with the norm $\| \cdot \|_{L_\infty(\mathbb R; \ell_1)}$ (see e.g.~\cite[Proposition~1.2.29]{hytonen2016analysis}). Recall that $f^* = (f^*_k)_{k \in \mathbb N_0}: \mathbb R \rightarrow \ell_1^+$ is the initial condition of the system of PDEs~\eqref{Paper02_PDE_system_labelled_particles}. By Assumption~\ref{Paper02_assumption_initial_condition_labelled_particles}(i), $f^* \in L_\infty(\mathbb R; \ell_1)$. Take $T_2 > T_1 \geq 0$, and let 
\[
    \mathscr C_b((T_1,T_2]; L_\infty(\mathbb R; \ell_1)) \defeq \left\{\phi \in \mathscr C((T_1, T_2]; L_\infty(\mathbb R; \ell_1)): \; \sup_{t \in (T_1,T_2]} \| \phi(t, \cdot) \|_{L_\infty(\mathbb R; \ell_1)} < \infty \right\}.
\]
Define the norm $\| \cdot \|_{L_{\infty}((T_1,T_2]; L_{\infty}(\mathbb R; \ell_1))}: \mathscr C_b((T_1,T_2]; L_\infty(\mathbb R; \ell_1)) \rightarrow [0, \infty)$ given for all $\phi \in \mathscr C_b((T_1,T_2]; L_\infty(\mathbb R; \ell_1))$ by
\begin{equation} \label{Paper02_norm_C_[T_1,T_2]}
\| \phi \|_{L_{\infty}((T_1,T_2]; L_{\infty}(\mathbb R; \ell_1))} \defeq \sup_{t \in (T_1,T_2]} \| \phi(t,\cdot) \|_{L_\infty(\mathbb R; \ell_1)},
\end{equation}
where the equality follows from the fact that $\phi \in \mathscr C((T_1,T_2]; L_{\infty}(\mathbb R; \ell_1))$. Since $L_\infty(\mathbb R; \ell_1)$ is a Banach space, the vector space $\mathscr C_b((T_1,T_2]; L_\infty(\mathbb R; \ell_1))$ is also a Banach space when equipped with the norm~$\| \cdot \|_{L_{\infty}((T_1,T_2]; L_{\infty}(\mathbb R; \ell_1))}$ (see e.g.~\cite[Theorem~43.6]{MunkresTopology}). Recall the definition of  the semigroup $\{P_t\}_{t \geq 0}$ of Brownian motion run at speed $m$ in~\eqref{Paper02_semigroup_BM_action_ell_1_functions}, and recall the definition of the reaction term $F^*: \ell_1^+ \times \ell_1 \rightarrow \ell_1$ in~\eqref{Paper02_reaction_term_PDE_labelled}. Let $u = (u_k)_{k \in \mathbb N_0}: [0, \infty) \times \mathbb R\rightarrow \ell_1^+$ be the continuous mild solution to~\eqref{Paper02_PDE_scaling_limit} given in Proposition~\ref{Paper02_smoothness_mild_solution}. For all $T_2 > T_1 \geq 0$, $g \in L_{\infty}(\mathbb R; \ell_1)$, $\phi \in \mathscr C_b((T_1,T_2]; L_\infty(\mathbb R; \ell_1))$, $x \in \mathbb R$ and $t \in (T_1,T_2]$, we define
\begin{equation} \label{Paper02_definition_linear_map_pre_picard}
    (H(T_1,T_2,g,t)\phi)(x) \defeq (P_{t - T_1}g)(x) + \int_{T_1}^t \Big(P_{t - \tau} F^*(u(\tau,\cdot),\phi(\tau,\cdot))\Big)(x) \, d\tau.
\end{equation}

\begin{lemma} \label{Paper02_definition_linear_semigroup_non_autonomous}
     Suppose that $(s_k)_{k \in \mathbb N_0}$, $q_+$ and $q_-$ satisfy Assumptions~\ref{Paper02_assumption_fitness_sequence} and~\ref{Paper02_assumption_polynomials}, that $\mu \in (0,1)$ and $m > 0$, that $f$ satisfies Assumption~\ref{Paper02_assumption_initial_condition} and~\ref{Paper02_assumption_initial_condition_modified}, and that the reaction term $F = (F_{k})_{k \in \mathbb{N}_{0}}$ defined in~\eqref{Paper02_reaction_term_PDE} is monostable in the sense of Definition~\ref{Paper02_assumption_monostable_condition}. Let $u = (u_k)_{k \in \mathbb N_0}: [0, \infty) \times \mathbb R\rightarrow \ell_1^+$ be the continuous mild solution to~\eqref{Paper02_PDE_scaling_limit} given in Proposition~\ref{Paper02_smoothness_mild_solution}. For all $T_2 > T_1 \geq 0$ and $g \in L_{\infty}(\mathbb R; \ell_1)$, the following properties hold:
     \begin{enumerate}[(i)]
         \item $H(T_1,T_2,g,t)\phi \in L_\infty(\mathbb R; \ell_1)$ for all $t \in (T_1,T_2]$ and $\phi \in \mathscr C_b((T_1,T_2]; L_{\infty}(\mathbb R; \ell_1))$.
         
         \item $(H(T_1,T_2,g,t)\phi)_{t \in (T_1,T_2]} \in \mathscr C_b((T_1,T_2]; L_{\infty}(\mathbb R; \ell_1))$ for all $\phi \in \mathscr C_b((T_1,T_2]; L_{\infty}(\mathbb R; \ell_1))$.
     \end{enumerate}
     Moreover, there exists $C = C(q_+,q_-) > 0$ such that for all $T_2 > T_1 \geq 0$, $g \in L_\infty(\mathbb R; \ell_1)$ and $\phi, \psi \in \mathscr C_b ((T_1,T_2]; L_\infty(\mathbb R; \ell_1))$,
     \begin{equation} \label{Paper02_contraction_principle}
     \begin{aligned}
        & \sup_{t \in (T_1,T_2]} \left\|(H(T_1,T_2,g,t)\phi) - (H(T_1,T_2,g,t)\psi)\right\|_{L_\infty(\mathbb R; \ell_1)} \\
        & \quad \leq C(T_2 - T_1) \sup_{t \in (T_1,T_2]} \left\|\phi(t,\cdot) -\psi(t, \cdot) \right\|_{L_\infty(\mathbb R; \ell_1)}.
    \end{aligned}
    \end{equation}
\end{lemma}

\begin{proof}

To establish assertion (i), observe that by the definition of the reaction term $F^{*} = (F^{*}_k)_{k \in \mathbb N_0}: \ell_1^+ \times \ell_1 \rightarrow \ell_1$ in~\eqref{Paper02_reaction_term_PDE_labelled}, the fact that $s_k \leq 1 \, \forall k \in \mathbb N_0$ by Assumption~\ref{Paper02_assumption_fitness_sequence}(i) and~(iii), and using that~$q_+$ and~$q_-$ are non-negative on $[0, \infty)$ by Assumption~\ref{Paper02_assumption_polynomials}, we have that for all $z \in \ell_1^+$ and $z^* \in \ell_1$,
    \begin{equation} \label{Paper02_aux_est_label_react_i}
    \begin{aligned}
        \| F^*(z,z^*) \|_{\ell_1} & = \sum_{k = 0}^\infty \left\vert (s_k(1-\mu)z^*_k + \mathds 1_{\{k \geq 1\}}s_{k-1}\mu z^{*}_{k - 1})q_+(\| z \|_{\ell_1}) - z^*_kq_-(\| z \|_{\ell_1}) \right\vert \\
        & \leq \| z^* \|_{\ell_1} (q_+(\| z \|_{\ell_1}) + q_-(\| z \|_{\ell_1})).
    \end{aligned}
    \end{equation}
     Also, since~$q_+,q_-:[0,\infty) \rightarrow [0,\infty)$ are polynomials by Assumption~\ref{Paper02_assumption_polynomials}, Lemma~\ref{Paper02_uniform_bound_total_mass} implies that
    \begin{equation} \label{Paper02_aux_est_label_react_modified}
        \sup_{t \geq 0} \| q_+(\| u(t,\cdot) \|_{\ell_1}) + q_-(\| u(t,\cdot) \|_{\ell_1}) \|_{L_{\infty}(\mathbb R; \mathbb R)} \leq \| q_+ + q_- \|_{L_{\infty}([0,1];\mathbb R)}.
    \end{equation}
    For $(t,x) \in (0,\infty) \times \mathbb R$, recall the definition of the Gaussian kernel~$p(t,x)$ in~\eqref{Paper02_Gaussian_kernel}. For all $T_2 > T_1 \geq 0$, $g \in L_\infty(\mathbb R; \ell_1)$, $\phi \in \mathscr C_b((T_1,T_2]; L_\infty(\mathbb R; \ell_1))$, $t \in (T_1,T_2]$ and $x \in \mathbb R$, we have
    \begin{equation} \label{Paper02_aux_est_label_react_ii}
    \begin{aligned}
        & \| (H(T_1,T_2,g,t) \phi)(x) \|_{\ell_1} \\ 
        & \quad \leq \int_\mathbb R p(t-T_1, x-y) \| g(y) \|_{\ell_1} \, dy + \int_{T_1}^t \int_\mathbb R p(t-\tau,x-y) \| F^*(u(\tau,y),\phi(\tau,y)) \|_{\ell_1} \, dy \, d\tau \\
        & \quad \leq \| g \|_{L_\infty(\mathbb R;\ell_1)} + \| q_+ + q_- \|_{L_{\infty}([0,1];\mathbb R)} \int_{T_1}^t \int_\mathbb R p(t-\tau, x-y) \| \phi(\tau,y)  \|_{\ell_1} \, dy \, d\tau \\
        & \quad \leq \| g \|_{L_\infty(\mathbb R;\ell_1)} + \| q_+ + q_- \|_{L_{\infty}([0,1];\mathbb R)} (T_2 - T_1) \| \phi \|_{L_{\infty}((T_1,T_2]; L_{\infty}(\mathbb R; \ell_1))},
    \end{aligned}
    \end{equation}
    where for the second inequality we used~\eqref{Paper02_aux_est_label_react_i} and~\eqref{Paper02_aux_est_label_react_modified}. Since $\phi \in \mathscr C_b((T_1,T_2]; L_\infty(\mathbb R; \ell_1))$, the right-hand side of~\eqref{Paper02_aux_est_label_react_ii} is finite. Observing that~\eqref{Paper02_aux_est_label_react_ii} holds for all $x \in \mathbb R$, we conclude that assertion (i) holds. 

    In order to establish assertion (ii), we notice that by combining~\eqref{Paper02_aux_est_label_react_i},~\eqref{Paper02_aux_est_label_react_modified} and the definition of $H(T_1,T_2,g,t)\phi$ in~\eqref{Paper02_definition_linear_map_pre_picard}, 
    we can apply estimate~\eqref{Paper02_pre_step_holder} from Lemma~\ref{Paper02lem:heat_continuity_ell1} to conclude that for all $\delta \in (0, T_2 - T_1)$, there exists $C_{\delta,g,T_1,T_2,\phi} > 0$ such that for all $t,t' \in [T_1 + \delta, T_2]$ and $x \in \mathbb R$,
    \[
    \| (H(T_1,T_2,g,t)\phi)(x) - (H(T_1,T_2,g,t')\phi)(x) \|_{\ell_1} \leq C_{\delta,g,T_1,T_2,\phi} \vert t' - t \vert^{1/2}.
    \]
    It then follows that
    \[
    \| H(T_1,T_2,g,t)\phi - H(T_1,T_2,g,t')\phi\|_{L_\infty(\mathbb R; \ell_1)} \leq C_{\delta,g,T_1,T_2,\phi} \vert t' - t \vert^{1/2}, 
    \]
    and therefore $(H(T_1,T_2,g,t)\phi)_{t \in (T_1,T_2]} \in \mathscr C((T_1,T_2]; L_{\infty}(\mathbb R; \ell_1))$. Hence, by observing that~\eqref{Paper02_aux_est_label_react_ii} holds for all~$t \in (T_1,T_2]$ and~$x \in \mathbb R$, the proof of assertion~(ii) is complete.

    It remains to establish~\eqref{Paper02_contraction_principle}. Observe that by the definition of the reaction term $F^{*} = (F^{*}_k)_{k \in \mathbb N_0}: \ell_1^+ \times \ell_1 \rightarrow \ell_1$ in~\eqref{Paper02_reaction_term_PDE_labelled}, the fact that $s_k \leq 1 \, \forall k \in \mathbb N_0$ by Assumption~\ref{Paper02_assumption_fitness_sequence}(i) and~(iii), and using that~$q_+$ and~$q_-$ are non-negative on $[0, \infty)$ by Assumption~\ref{Paper02_assumption_polynomials}, we have that for all $z \in \ell_1^+$ and $z^*,z' \in \ell_1$,
    \begin{equation} \label{Paper02_aux_est_label_react}
    \begin{aligned}
        & \| F^*(z,z^*) - F^*(z,z') \|_{\ell_1} \\
        & \quad = \sum_{k = 0}^\infty \left\vert (s_k(1-\mu)(z^*_k - z'_k) + \mathds 1_{\{k \geq 1\}}s_{k-1}\mu (z^*_{k-1} - z'_{k-1}))q_+(\| z \|_{\ell_1}) - (z^*_k - z'_k)q_-(\| z \|_{\ell_1}) \right\vert \\
        & \quad \leq \| z^* - z' \|_{\ell_1} (q_+(\| z \|_{\ell_1}) + q_-(\| z \|_{\ell_1})).
    \end{aligned}
    \end{equation}
    Then, for all $\phi,\psi \in \mathscr C_b((T_1,T_2]; L_\infty(\mathbb R; \ell_1))$, $g \in L_\infty(\mathbb R; \ell_1)$, $t \in (T_1,T_2]$ and $x \in \mathbb R$,
    \begin{equation} \label{Paper02_aux_est_label_react_xii}
    \begin{aligned}
        & \| (H(T_1,T_2,g, t)\phi)(x) - (H(T_1,T_2,g, t)\psi)(x) \|_{\ell_1} \\
        & \quad \leq \int_{T_1}^t \int_\mathbb R p(t-\tau,x-y) \| F^*(u(\tau,y),\phi(\tau,y)) - F^*(u(\tau,y),\psi(\tau,y)) \|_{\ell_1} \, dy \, d\tau \\
        & \quad \leq \| q_+ + q_- \|_{L_{\infty}([0,1];\mathbb R)} \int_{T_1}^t \int_\mathbb R p(t-\tau,x-y) \| \phi(\tau,y) - \psi(\tau,y) \|_{\ell_1} \, dy\, d\tau \\
        & \quad \leq \| q_+ + q_- \|_{L_{\infty}([0,1];\mathbb R)} (t-T_1) \| \phi - \psi \|_{L_\infty((T_1,T_2]; L_\infty(\mathbb R; \ell_1))},
    \end{aligned}
    \end{equation}
    where for the second inequality we used~\eqref{Paper02_aux_est_label_react} and~\eqref{Paper02_aux_est_label_react_modified}. By taking the supremum over $t \in (T_1,T_2]$ on both sides of~\eqref{Paper02_aux_est_label_react_xii}, we conclude that~\eqref{Paper02_contraction_principle} holds, which completes the proof.
\end{proof}

We are finally ready to prove Proposition~\ref{Paper02_basic_properties_linear_parabolic_pdes_tracer}.

\begin{proof}[Proof of Proposition~\ref{Paper02_basic_properties_linear_parabolic_pdes_tracer}]
    Let $C = C(q_+,q_-) > 0$ be such that estimate~\eqref{Paper02_contraction_principle} from Lemma~\ref{Paper02_definition_linear_semigroup_non_autonomous} holds for any $T_2 > T_1 \geq 0$ and $g \in L_\infty(\mathbb R; \ell_1)$, and take $T^* > 0$ sufficiently small so that $CT^* < 1$.
    Since, by Assumption~\ref{Paper02_assumption_initial_condition_labelled_particles}(i), $f^* = (f^*_k)_{k \in \mathbb N_0} \in L_\infty(\mathbb R; \ell_1)$, Lemma~\ref{Paper02_definition_linear_semigroup_non_autonomous} implies that $\phi \mapsto H(0,T^*,f^*, \cdot) \phi$ is a contraction mapping on $\mathscr C_b((0,T^*]; L_\infty(\mathbb R; \ell_1))$.
    Hence, by the completeness of $\mathscr C_b((0,T^*]; L_\infty(\mathbb R; \ell_1))$, we can apply Banach's fixed-point theorem (see e.g.~\cite[Theorem~5.7]{brezis2011functional}) to conclude that there exists a unique $u^* \in \mathscr C_b((0,T^*]; L_\infty(\mathbb R; \ell_1))$ such that $(H(0,T^*,f^*, t)u^*)(\cdot) = u^*(t, \cdot)$ for all $t \in (0,T^*]$.
    
    Since $u^*(T^*,\cdot) \in L_\infty(\mathbb R; \ell_1)$, we can use the same argument to define $u^*$ on the time interval $(0, 2T^*]$. Repeating this argument iteratively (which is allowed since the constant $C > 0$ in estimate~\eqref{Paper02_contraction_principle} from Lemma~\ref{Paper02_definition_linear_semigroup_non_autonomous} does not depend on the particular time interval), we can define $u^*$ globally in time, i.e.~there is a unique $u^*: (0, \infty) \times \mathbb R \rightarrow \ell_1$ such that for all $n \in \mathbb N_0$, $u^* \vert_{(nT^*,(n+1)T^*]} \in \mathscr C_b((nT^*,(n+1)T^*]; L_\infty (\mathbb R; \ell_1))$ and
    \begin{equation} \label{Eq:proof_uniqueness_solutions_prop_812}
    u^*\vert_{(nT^*,(n+1)T^*]} = H\Big(nT^*,\, (n+1)T^*, \, u^*(nT^*,\cdot), \, \cdot\Big) \, \Big(u^*\vert_{(nT^*,(n+1)T^*]}\Big).
    \end{equation}
    Since this argument is standard in the construction of solutions of both finite and infinite-dimensional dynamical systems, we omit the details and refer the interested reader to e.g.~\cite[Theorem~2.2.2]{kolokoltsov2019differential}.
    
    For $T = 0$, we then define $u^*(0, \cdot ) \equiv f^*(\cdot)$. Then, by the semigroup property of the heat kernel~$\{P_{t}\}_{t \geq 0}$ defined in~\eqref{Paper02_semigroup_BM_action_ell_1_functions}, the function $u^*: [0, \infty) \times \mathbb R \rightarrow \ell_1$ constructed in this way satisfies~\eqref{Paper02_mild_formulation_labelled_particles} for all $(T,x) \in [0, \infty) \times \mathbb R$, i.e.~$u^* = (u^*_k)_{k \in \mathbb N_0}$ satisfies assertion~(i). Since for all $T \in (0,\infty)$, $u^*\vert_{(0,T]} \in L_\infty(0,T] \times \mathbb R; \ell_1)$, by applying Proposition~\ref{Paper02_smoothness_mild_solution}(iii) and estimate~\eqref{Paper02_aux_est_label_react_i}, we have that for all $T \in (0,\infty)$, the map
    \[
    (F^*(u(t, x),u^*(t,x)))_{t \in [0,T], \, x \in \mathbb R} \in L_{\infty}([0,T] \times \mathbb R; \ell_1).
    \]
    
    By using~\eqref{Paper02_mild_formulation_labelled_particles} and estimate~\eqref{Paper02_pre_step_holder} from Lemma~\ref{Paper02lem:heat_continuity_ell1}, it follows that $u^* \in \mathscr C((0, \infty) \times \mathbb R; \ell_1) \cap \mathscr C((0, \infty); L_\infty(\mathbb R;\ell_1))$. Hence, recalling that by Assumption~\ref{Paper02_assumption_initial_condition_labelled_particles}(i), $f^* \in L_\infty(\mathbb R; \ell_1)$ and that $u^*(0, \cdot) = f^*(\cdot)$, it follows that $u^*$ also satisfies assertions~(ii) and~(iii). 
    
    We will now establish that $u^* = (u^*_k)_{k \in \mathbb N_0}$ satisfies assertion~(iv). By the semigroup property of the heat kernel~$\{P_{t}\}_{t \geq 0}$ defined in~\eqref{Paper02_semigroup_BM_action_ell_1_functions} and by assertion~(i) of this proposition, observe that
    for all~$\delta > 0$,~$t > \delta$ and~$x \in \mathbb{R}$, we have
    \begin{equation} \label{Paper02_aux_global_mild_holder_i_labelled}
        u^*(t,x) = (P_{t - \delta}u^*(\delta, \cdot))(x) + \int_{\delta}^{t} \Big(P_{t - \tau}F^*(u(\tau, \cdot), u^*(\tau, \cdot))\Big)(x) \, d\tau.
    \end{equation}
    Since $u^*(\delta, \cdot) \in \mathscr C(\mathbb R; \ell_1)$, by the standard theory for linear parabolic PDEs (see e.g.~\cite[Theorem~8.10.1]{krylov1996lectures}), assertion~(iv) will be proved after establishing that for all $T > \delta > 0$, there exists $C_{\delta, T} > 0$ such that for all~$t_1,t_2 \in [\delta, T]$ and $x_1,x_2 \in \mathbb R$,
    \begin{equation} \label{Paper02_reaction_term_holder_labelled_coordinatewise}
        \| F^*(u(t_1,x_1),u^*(t_1,x_1)) - F^*(u(t_2,x_2),u^*(t_2,x_2)) \|_{\ell_1} \leq C_{\delta,T} ( \vert x_1 - x_2 \vert + \vert t_1 - t_2 \vert^{1/2}).
    \end{equation}
    Observe that by combining assertions~(i) and~(ii) of this proposition with~\eqref{Paper02_aux_est_label_react_i} and estimate~\eqref{Paper02_pre_step_holder} from Lemma~\ref{Paper02lem:heat_continuity_ell1}, we have that for all $T > \delta > 0$, there exists $C^{(1)}_{\delta,T} > 0$ such that for all~$t_1,t_2 \in [\delta,T]$ and~$x_1,x_2 \in \mathbb R$,
    \begin{equation} \label{Paper02_holder_cooordinate_labelled_i}
        \| u^*(t_1,x_1) - u^*(t_2,x_2) \|_{\ell_1} \leq C^{(1)}_{\delta,T} ( \vert x_1 - x_2 \vert + \vert t_1 - t_2 \vert^{1/2}).
    \end{equation}
    Take $T > \delta > 0$. By the definition of the reaction term $F^* = (F^*_{k})_{k \in \mathbb{N}_{0}}$ in~\eqref{Paper02_reaction_term_PDE_labelled}, the triangle inequality and the fact that $s_k \leq 1 \; \forall k \in \mathbb N_0$ by Assumption~\ref{Paper02_assumption_fitness_sequence}(i) and~(iii), we have that for all~$t_1,t_2 \in [\delta,T]$ and~$x_1,x_2 \in \mathbb R$,
    \begin{equation} \label{Paper02_holder_cooordinate_labelled_ii}
    \begin{aligned}
        & \| F^*(u(t_1,x_1),u^*(t_1,x_1)) - F^*(u(t_2,x_2),u^*(t_2,x_2)) \|_{\ell_1} \\
        & \quad \leq (q_+(\|u(t_1,x_1)\|_{\ell_1}) + q_-(\|u(t_1,x_1)\|_{\ell_1})) \| u^*(t_1,x_1) - u^*(t_2,x_2) \|_{\ell_1} \\
        & \qquad + \| u^*(t_2,x_2) \|_{\ell_1} \Big(\vert q_+(\|u(t_1,x_1)\|_{\ell_1}) -  q_+(\|u(t_2,x_2)\|_{\ell_1}) \vert \\
        & \qquad \qquad \qquad \qquad \qquad + \vert q_-(\|u(t_1,x_1)\|_{\ell_1}) - q_-(\|u(t_2,x_2)\|_{\ell_1}) \vert \Big).
    \end{aligned}
    \end{equation}
    For the first term on the right-hand side of~\eqref{Paper02_holder_cooordinate_labelled_ii}, we use~\eqref{Paper02_holder_cooordinate_labelled_i} and estimate~\eqref{Paper02_aux_est_label_react_modified}, and for the second term on the right-hand side of~\eqref{Paper02_holder_cooordinate_labelled_ii} we use the fact that $q_{+},q_{-}: [0,\infty) \rightarrow [0,\infty)$ are polynomials, combined with Lemma~\ref{Paper02_uniform_Holder_continuity}, to deduce that estimate~\eqref{Paper02_reaction_term_holder_labelled_coordinatewise} holds, which completes the proof of assertion~(iv).

    We will now show that $u^* = (u^*_k)_{k \in \mathbb N_0}: [0, \infty) \times \mathbb R \rightarrow \ell_1$ is the unique function satisfying assertions~(i)--(iv). Let $\tilde u^* = (\tilde u^*_k)_{k \in \mathbb N_0}: [0, \infty) \times \mathbb R \rightarrow \ell_1$ be such that $\tilde u^*$ satisfies assertions~(i)--(iv). Then, by~\eqref{Paper02_mild_formulation_labelled_particles}, Proposition~\ref{Paper02_smoothness_mild_solution}(iii),~\eqref{Paper02_aux_est_label_react_i} and estimate~\eqref{Paper02_pre_step_holder} from Lemma~\ref{Paper02lem:heat_continuity_ell1}, we conclude that for all $T \in (0,\infty)$, $\tilde u^* \in \mathscr C_b((0,T]; L_\infty(\mathbb R; \ell_1))$. Moreover, by~\eqref{Paper02_mild_formulation_labelled_particles} and~\eqref{Paper02_definition_linear_map_pre_picard}, $\tilde u^*$ is such that for all $T \in (0,\infty)$,
    \[
    \tilde u^*\vert_{(0,T]} = H(0,T,f^*,\cdot) \Big(\tilde u^*\vert_{(0,T]}\Big).
    \]
    Let $T^* \in (0, \infty)$ be the parameter defined at the beginning of the proof of this proposition. As explained at the beginning of the proof of this proposition, from Lemma~\ref{Paper02_definition_linear_semigroup_non_autonomous}, $\phi \mapsto H(0,T^*,f^*,\cdot)\phi$ is a contraction mapping on $\mathscr C_b((0,T^*]; L_\infty(\mathbb R; \ell_1))$, and therefore by Banach's fixed-point theorem, it follows that
    \[
      \tilde u^*\vert_{(0,T^*]} =  u^*\vert_{(0,T^*]}.
    \]
    Repeating this argument iteratively, we conclude that for all $n \in \mathbb N_0$,
    \[
     \tilde u^*\vert_{(nT^*,(n+1)T^*]} = H\Big(nT^*,\, (n+1)T^*, \, u^*(nT^*,\cdot), \, \cdot\Big) \, \Big(\tilde u^*\vert_{(nT^*,(n+1)T^*]}\Big).
    \]
    Therefore by using the fact that $\phi \mapsto H(nT^*,(n+1)T^*,u^*(nT^*, \cdot),\cdot)\phi$ is a contraction mapping on $\mathscr C_b((nT^*,(n+1)T^*]; L_\infty(\mathbb R; \ell_1))$ by Lemma~\ref{Paper02_definition_linear_semigroup_non_autonomous} and the definition of $T^*$ at the beginning of the proof of this proposition, we can apply Banach's fixed-point theorem and~\eqref{Eq:proof_uniqueness_solutions_prop_812} to conclude that
    \[
     \tilde u^*\vert_{(nT^*,(n+1)T^*]} =  u^*\vert_{(nT^*,(n+1)T^*]} \quad \forall n \in \mathbb N_0,
    \]
    which implies that $\tilde u^* = u^*$, and therefore uniqueness holds.
    
    To complete the proof of Proposition~\ref{Paper02_basic_properties_linear_parabolic_pdes_tracer}, it remains to establish~\eqref{Paper02_simple_inequality_comparing_labelled_original_system_PDEs}. For this purpose, we will use an induction argument on $k \in \mathbb N_0$. Define the function $h_0: [0, \infty) \times \mathbb R \rightarrow \mathbb R$ by 
    \[
    h_0(t,x) \defeq (1-\mu)q_+(\|u(t,x)\|_{\ell_1}) -  q_-(\|u(t,x)\|_{\ell_1}) \quad \forall \, (t,x) \in [0, \infty) \times \mathbb R.
    \]
    By Assumption~\ref{Paper02_assumption_initial_condition_labelled_particles}(iii), we have that $0 \leq u^*_0(0, x) \leq u_0(0,x) \; \forall \, x \in \mathbb R$. Moreover, by recalling the definition of the reaction terms $F = (F_k)_{k \in \mathbb N_0}$ and $F^* = (F^*_k)_{k \in \mathbb N_0}$ in~\eqref{Paper02_reaction_term_PDE} and~\eqref{Paper02_reaction_term_PDE_labelled}, respectively, and using Proposition~\ref{Paper02_smoothness_mild_solution}(i) and~(iv), and assertions~(i) and~(iv) of this proposition, we conclude that for all $(t,x) \in (0, \infty) \times \mathbb R$,
    \begin{equation*}
        \partial_t u^{*}_0(t,x) - \frac{m}{2} \Laplace u^*_0(t,x) - u^*_0(t,x) h_0(t,x) = \partial_t u_0(t,x) - \frac{m}{2} \Laplace u_0(t,x) - u_0(t,x) h_0(t,x).
    \end{equation*}
    Hence, by classical comparison theorems for parabolic PDEs (see e.g.~\cite[Proposition~2.1]{aronson1975nonlinear}), it follows that $0 \leq u^*_0(t,x) \leq u_0(t,x) \; \forall \, (t,x) \in [0, \infty) \times \mathbb R$.

    Suppose now that~\eqref{Paper02_simple_inequality_comparing_labelled_original_system_PDEs} holds for some $k \in \mathbb N_0$. Define the function $h_{k+1}: [0, \infty) \times \mathbb R \rightarrow \mathbb R$ by
    \[
    h_{k+1}(t,x) \defeq s_{k+1}(1-\mu) q_+(\|u(t,x)\|_{\ell_1}) -  q_-(\|u(t,x)\|_{\ell_1}) \quad \forall \, (t,x) \in [0, \infty) \times \mathbb R.
    \]
    By Assumption~\ref{Paper02_assumption_initial_condition_labelled_particles}(iii), we have $0 \leq u^*_{k+1}(0, x) \leq u_{k+1}(0,x) \; \forall \, x \in \mathbb R$. Moreover, by using~\eqref{Paper02_reaction_term_PDE} and~\eqref{Paper02_reaction_term_PDE_labelled} again, and then applying Proposition~\ref{Paper02_smoothness_mild_solution}(i) and~(iv), assertions~(i) and~(iv) of this proposition, and by our induction hypothesis, we conclude that for all $(t,x) \in (0, \infty) \times \mathbb R$,
    \begin{equation*}
    \begin{aligned}
         & \partial_t u^{*}_{k+1}(t,x) - \frac{m}{2} \Laplace u^*_{k+1}(t,x) - u^*_{k+1}(t,x) h_{k+1}(t,x) - s_k\mu u_k(t,x)q_+(\|u(t,x)\|_{\ell_1}) \\
         & \quad \leq \partial_t u_{k+1}(t,x) - \frac{m}{2} \Laplace u_{k+1}(t,x) - u_{k+1}(t,x) h_{k+1}(t,x) - s_k\mu u_k(t,x)q_+(\|u(t,x)\|_{\ell_1}).
    \end{aligned}
    \end{equation*}
    Hence, again using a standard comparison principle,~\eqref{Paper02_simple_inequality_comparing_labelled_original_system_PDEs} holds when~$k$ is replaced by~$k+1$, which completes the proof.
\end{proof}

\subsection{Proof of Lemma~\ref{Paper02_lem_equiv_mild_weak_sol}} \label{Paper02_appendix_section_proof_auxiliary_no_technical_lemmas}

In this subsection, we prove Lemma~\ref{Paper02_lem_equiv_mild_weak_sol}. 

\begin{proof}[Proof of Lemma~\ref{Paper02_lem_equiv_mild_weak_sol}]
    We must establish that $u: [0, \infty) \times \mathbb R \rightarrow \ell_1^+$ satisfies Definition~\ref{Paper02_weak_solution_distributional_sense}(i)-(iv). Definition~\ref{Paper02_weak_solution_distributional_sense}(i) follows directly from condition~(i). By condition~(ii) and by recalling estimate~\eqref{Paper02_trivial_ell_1_bound_reaction_term}, we conclude that for any $T > 0$,
    \begin{equation} \label{Paper02_trivial_estimate}
    F(u) \vert_{[0,T] \times \mathbb R} \in L_{\infty}([0,T] \times \mathbb R; \ell_1).
    \end{equation}
    In particular, $F(u) \in L_{1,\textrm{loc}}([0, \infty) \times \mathbb R; \ell_1)$, and $u$ satisfies Definition~\ref{Paper02_weak_solution_distributional_sense}(iii). Observe also that by condition~(iii), and then by Assumption~\ref{Paper02_assumption_initial_condition}(ii),~\eqref{Paper02_trivial_estimate} and Jensen's inequality, we have that for all $T > 0$,
    \[
    \| u(T, \cdot) \|_{L_{\infty}(\mathbb R; \ell_1)} \leq \| f \|_{L_{\infty}(\mathbb R; \ell_1)} + T  \| F(u)  \|_{L_{\infty}([0,T] \times \mathbb R; \ell_1)} < \infty.
    \]
    In particular, for all $T \geq 0$, $u(T, \cdot) \in L_{1,\textrm{loc}}(\mathbb R; \ell_1)$, and therefore $u$ satisfies Definition~\ref{Paper02_weak_solution_distributional_sense}(ii).

    It remains to establish that $u$ satisfies Definition~\ref{Paper02_weak_solution_distributional_sense}(iv). Let $k \in \mathbb N_0$, $\varphi \in \mathscr C_c^{1,2}([0,\infty) \times \mathbb R; \mathbb R)$ and $T > 0$ be arbitrary. As explained after Definition~\ref{Paper02_weak_solution_distributional_sense}, it will suffice to show that~\eqref{Paper02_coordinate_wise_weak_sol} holds. Note that by condition~(iii),
    \begin{equation} \label{Paper02_mild_is_weak_i}
    \begin{aligned}
        & \int_0^T \int_\mathbb R u_k(t,x) \left(\partial_t + \frac{m}{2} \Laplace\right) \varphi(t,x) \, dx \, dt \\ & \; = \int_0^T \int_\mathbb R (P_tf_k)(x) \left(\partial_t + \frac{m}{2} \Laplace\right) \varphi(t,x) \, dx \, dt \\ & \quad \; + \int_0^T \int_\mathbb R \left(\partial_t + \frac{m}{2} \Laplace\right) \varphi(t,x) \left(\int_0^t (P_{t- \tau} F_k(u(\tau,\cdot)))(x) \, d\tau\right) dx \,dt. 
    \end{aligned}
    \end{equation}
    We will tackle the terms on the right-hand side of~\eqref{Paper02_mild_is_weak_i} separately. For the first term on the right-hand side of~\eqref{Paper02_mild_is_weak_i}, fix $\varepsilon \in (0, T)$ and note that
    \begin{equation} \label{Paper02_mild_is_weak_iii}
    \begin{aligned}
        & \int_0^T \int_\mathbb R (P_tf_k)(x) \left(\partial_t + \frac{m}{2} \Laplace\right) \varphi(t,x) \, dx \, dt \\ & \quad =  \int_\varepsilon^T \int_\mathbb R (P_tf_k)(x) \left(\partial_t + \frac{m}{2} \Laplace\right) \varphi(t,x) \, dx \, dt + \int_0^\varepsilon \int_\mathbb R (P_tf_k)(x) \left(\partial_t + \frac{m}{2} \Laplace\right) \varphi(t,x) \, dx \, dt.
    \end{aligned}
    \end{equation}
    Since $\varphi \in \mathscr C_c^{1,2}([0, \infty) \times \mathbb R; \mathbb R)$, there exists $R_{\varphi} \in (0,\infty)$ such that $\supp \varphi \subset [0, \infty) \times [-R_{\varphi}, R_{\varphi}]$. Moreover, there exists $C_{\varphi} \in (0, \infty)$ such that
    \begin{equation} \label{Paper02_mild_is_weak_ii}
       \sup_{t \in [0,\infty)} \left(\| \partial_t \varphi (t,\cdot) \|_{L_{\infty}(\mathbb R; \mathbb R)} + \frac{m}{2} \| \Laplace \varphi (t,\cdot) \|_{L_{\infty}(\mathbb R; \mathbb R)}\right) \leq C_{\varphi}.
    \end{equation}
    Therefore, by applying Jensen's inequality, we can bound the second term on the right-hand side of~\eqref{Paper02_mild_is_weak_iii} by
    \begin{equation} \label{Paper02_mild_is_weak_iv}
    \begin{aligned}
        \left\vert \int_0^\varepsilon \int_\mathbb R (P_tf_k)(x) \left(\partial_t + \frac{m}{2} \Laplace\right) \varphi(t,x) \, dx \, dt \right\vert \leq 2\varepsilon R_{\varphi}C_{\varphi} \| f_k \|_{L_{\infty}(\mathbb R; \mathbb R)} \leq 2\varepsilon R_{\varphi}C_{\varphi} \| f \|_{L_{\infty}(\mathbb R; \ell_1)}. 
    \end{aligned}
    \end{equation}
    For $(t,x) \in (0, \infty) \times \mathbb R$, recall the definition of the Gaussian kernel $p(t,x)$ in~\eqref{Paper02_Gaussian_kernel}. For the first term on the right-hand side of~\eqref{Paper02_mild_is_weak_iii}, by applying~\eqref{Paper02_semigroup_BM_action_ell_1_functions} and Fubini's theorem, we have that
    \begin{equation} \label{Paper02_mild_is_weak_v}
    \begin{aligned}
        & \int_\varepsilon^T \int_\mathbb R (P_tf_k)(x) \left(\partial_t + \frac{m}{2} \Laplace\right) \varphi(t,x) \, dx \, dt \\ & \quad =   \int_\mathbb R f_k(y) \left(\int_\varepsilon^T \int_\mathbb R p(t,x-y) \left(\partial_t + \frac{m}{2} \Laplace\right) \varphi(t,x) \, dx \, dt\right) dy \\ & \quad = \int_{\mathbb R} f_k(y) \left(\int_ \mathbb R \Big(p(T,x-y) \varphi(T,x) - p(\varepsilon, x-y) \varphi(\varepsilon,x)\Big) \, dx\right) dy \\ & \quad \quad + \int_\mathbb R f_k(y) \left(\int_\varepsilon^T \int_\mathbb R \left(- \partial_t + \frac{m}{2} \Laplace\right) p(t,x-y)  \varphi(t,x) \, dx \, dt\right) dy \\ & \quad = \int_{\mathbb R} (P_Tf_k)(x)\varphi(T,x) \, dx - \int_{\mathbb R} f_k(y) \left(\int_\mathbb R p(\varepsilon, x-y) \varphi(\varepsilon,x) \, dx \right) dy,
    \end{aligned}
    \end{equation}
    where for the second identity we used integration by parts, while for the third equality we used~\eqref{Paper02_semigroup_BM_action_ell_1_functions}, Fubini's theorem, and the fact that by~\eqref{Paper02_Gaussian_kernel}, $\left(- \partial_t + \frac{m}{2} \Laplace\right) p(t, \cdot) \equiv 0$ for $t > 0$. We now claim that
    \begin{equation} \label{Paper02_mild_is_weak_vi}
    \lim_{\varepsilon \downarrow 0} \int_{\mathbb R} f_k(y) \left(\int_\mathbb R p(\varepsilon, x-y) \varphi(\varepsilon,x) \, dx \right) dy = \int_\mathbb R f_k(y) \varphi(0,y) \, dy.
    \end{equation}
    To establish~\eqref{Paper02_mild_is_weak_vi}, we first notice that since $\varphi \in \mathscr{C}_c([0, \infty) \times \mathbb R; \mathbb R)$, by using standard regularity properties of the heat kernel (see e.g.~\cite[Theorem~2.3.1(iii)]{evans2022partial}), we have
    \[
    \lim_{\varepsilon \downarrow 0} \int_\mathbb R p(\varepsilon, x-y) \varphi(\varepsilon,x) \, dx = \varphi(0,y) \quad \forall \, y \in \mathbb R,
    \]
    and therefore, since $f \in L_\infty(\mathbb R; \ell_1)$ by Assumption~\ref{Paper02_assumption_initial_condition}(ii), by the dominated convergence theorem, the limit in~\eqref{Paper02_mild_is_weak_vi} will be proved after establishing that for all $\varepsilon \in (0,T)$ and $y \in \mathbb R$,
    \begin{equation} \label{Paper02_mild_is_weak_vii}
    \begin{aligned}
        & \left\vert \int_\mathbb R p(\varepsilon, x-y) \varphi(\varepsilon,x) \, dx \right\vert \\ & \quad \leq \mathds{1}_{\{y \in [-R_{\varphi},R_{\varphi}]\}} \|\varphi\|_{L_{\infty}([0, \infty) \times \mathbb R; \mathbb R)} + \mathds{1}_{\{y \not\in [-R_{\varphi},R_{\varphi}]\}} \|\varphi\|_{L_{\infty}([0, \infty) \times \mathbb R; \mathbb R)} e^{- {(\vert y \vert - R_\varphi)^2}/{2mT}},
    \end{aligned}
    \end{equation}
    recalling that we chose $R_{\varphi} \in (0,\infty)$ such that $\supp \varphi \subset [0, \infty) \times [-R_{\varphi}, R_\varphi]$. Now, observe that~\eqref{Paper02_mild_is_weak_vii} follows from Jensen's inequality and the elementary identity $\int_\mathbb R p(t,z) \, dz = 1 \; \forall t > 0$ for $y \in [-R_{\varphi}, R_\varphi]$, and for $y \not\in [-R_{\varphi}, R_\varphi]$,
    \begin{equation*}
    \begin{aligned}
         \left\vert \int_\mathbb R p(\varepsilon, x-y) \varphi(\varepsilon,x) \, dx \right\vert & \leq \|\varphi\|_{L_{\infty}([0, \infty) \times \mathbb R; \mathbb R)} \int_\mathbb R p(\varepsilon, x-y) \mathds{1}_{\{x \in [-R_\varphi,R_\varphi]\}}\, dx \\ & = \|\varphi\|_{L_{\infty}([0, \infty) \times \mathbb R; \mathbb R)} \int_{\vert y \vert - R_{\varphi}}^{\vert y \vert + R_{\varphi}} p(\varepsilon, z) \, dz \\ & \leq \|\varphi\|_{L_{\infty}([0, \infty) \times \mathbb R; \mathbb R)} e^{- {(\vert y \vert - R_\varphi)^2}/{2mT}}.
    \end{aligned}
    \end{equation*}
    Hence,~\eqref{Paper02_mild_is_weak_vii} holds, and therefore~\eqref{Paper02_mild_is_weak_vi} also holds. By applying~\eqref{Paper02_mild_is_weak_iv} and~\eqref{Paper02_mild_is_weak_v} to~\eqref{Paper02_mild_is_weak_iii}, and then by observing that we can take $\varepsilon \in (0,T)$ to be arbitrarily small and using~\eqref{Paper02_mild_is_weak_vi}, we conclude that
    \begin{equation} \label{Paper02_mild_is_weak_viii}
    \begin{aligned}
        \int_0^T \int_\mathbb R (P_tf_k)(x) \left(\partial_t + \frac{m}{2} \Laplace\right) \varphi(t,x) \, dx \, dt = \int_{\mathbb R} (P_Tf_k)(x)\varphi(T,x) \, dx - \int_{\mathbb R} f_k(x) \varphi(0,x) \, dx. 
    \end{aligned}
    \end{equation}
    We now tackle the second term on the right-hand side of~\eqref{Paper02_mild_is_weak_i}. Fix $\varepsilon \in (0,T)$, and observe that
    \begin{equation} \label{Paper02_mild_is_weak_ix}
    \begin{aligned}
        & \int_0^T \int_\mathbb R \left(\partial_t + \frac{m}{2} \Laplace\right) \varphi(t,x) \left(\int_0^t (P_{t- \tau} F_k(u(\tau,\cdot)))(x) \, d\tau\right) dx \,dt \\ & \quad = \int_\varepsilon^T \int_\mathbb R \left(\partial_t + \frac{m}{2} \Laplace\right) \varphi(t,x) \left(\int_0^{t-\varepsilon} (P_{t- \tau} F_k(u(\tau,\cdot)))(x) \, d\tau\right) dx \,dt \\ & \quad \quad + \int_\varepsilon^T \int_\mathbb R \left(\partial_t + \frac{m}{2} \Laplace\right) \varphi(t,x) \left(\int_{t - \varepsilon}^{t} (P_{t- \tau} F_k(u(\tau,\cdot)))(x) \, d\tau\right) dx \,dt \\ & \quad \quad + \int_0^\varepsilon \int_\mathbb R \left(\partial_t + \frac{m}{2} \Laplace\right) \varphi(t,x) \left(\int_{0}^{t} (P_{t- \tau} F_k(u(\tau,\cdot)))(x) \, d\tau\right) dx \,dt.
    \end{aligned}
    \end{equation}
    For the first term on the right-hand side of~\eqref{Paper02_mild_is_weak_ix}, by using integration by parts,
    \begin{equation} \label{Paper02_mild_is_weak_x}
    \begin{aligned}
        & \int_\varepsilon^T \int_\mathbb R \left(\partial_t + \frac{m}{2} \Laplace\right) \varphi(t,x) \left(\int_0^{t-\varepsilon} (P_{t- \tau} F_k(u(\tau,\cdot)))(x) \, d\tau\right) dx \, dt \\ 
        & \quad = \int_\mathbb R \varphi(T,x) \left(\int_0^{T-\varepsilon} (P_{T- \tau} F_k(u(\tau,\cdot)))(x) \, d\tau\right) dx \\ 
        & \quad \quad + \int_\varepsilon^T \int_\mathbb R  \varphi(t,x) \left(-\partial_t + \frac{m}{2} \Laplace\right)\left(\int_0^{t-\varepsilon} (P_{t- \tau} F_k(u(\tau,\cdot)))(x) \, d\tau\right) dx \, dt \\ 
        & \quad = \int_\mathbb R \varphi(T,x) \left(\int_0^{T-\varepsilon} (P_{T- \tau} F_k(u(\tau,\cdot)))(x) \, d\tau\right) dx  - \int_\varepsilon^T \int_\mathbb R \varphi(t,x) (P_{\varepsilon} F_k(u(t-\varepsilon,\cdot)))(x) \, dx \, dt \\ & \quad \quad +  \int_\varepsilon^T \int_\mathbb R \varphi(t,x) \left(\int_0^{t-\varepsilon} \int_\mathbb R \left(-\partial_t + \frac{m}{2} \Laplace\right) p(t - \tau, x - y) F_k(u(\tau,y)) \, dy \, d\tau\right) dx \, dt \\ & \quad =  \int_\mathbb R \varphi(T,x) \left(\int_0^{T-\varepsilon} (P_{T- \tau} F_k(u(\tau,\cdot)))(x) \, d\tau\right) dx  -  \int_\varepsilon^T \int_\mathbb R \varphi(t,x) (P_{\varepsilon} F_k(u(t-\varepsilon,\cdot)))(x) \, dx \, dt,
    \end{aligned}
    \end{equation}
    where for the second equality we used~\eqref{Paper02_semigroup_BM_action_ell_1_functions},~\eqref{Paper02_Gaussian_kernel}, the fact that $p(t,x)$ has uniformly bounded derivatives of all orders on $[\varepsilon, \infty) \times \mathbb  R$,~\eqref{Paper02_trivial_estimate} and the dominated convergence theorem, and for the third equality we used the fact that $\left(- \partial_t + \frac{m}{2} \Laplace\right) p(t, \cdot) \equiv 0$ for $t > 0$. We will now compute the limit for both terms on the right-hand side of~\eqref{Paper02_mild_is_weak_x} as $\varepsilon \downarrow 0$. By~\eqref{Paper02_trivial_estimate}, we have that for all $x \in \mathbb R$, the map
    \[
    [0, T] \ni t \mapsto \int_0^{t} (P_{T- \tau} F_k(u(\tau,\cdot)))(x) \, d\tau
    \]
    is continuous, and therefore by~\eqref{Paper02_trivial_estimate} and the dominated convergence theorem,
    \begin{equation} \label{Paper02_mild_is_weak_xi}
        \lim_{\varepsilon \downarrow 0} \int_\mathbb R \varphi(T,x) \left(\int_0^{T-\varepsilon} (P_{T- \tau} F_k(u(\tau,\cdot)))(x) \, d\tau\right) dx = \int_\mathbb R \varphi(T,x) \left(\int_0^{T} (P_{T- t} F_k(u(t,\cdot)))(x) \, dt\right) dx.
    \end{equation}
    For the second term on the right-hand side of~\eqref{Paper02_mild_is_weak_x}, using Fubini's theorem, and then by using~\eqref{Paper02_trivial_estimate} and the same argument we used to derive~\eqref{Paper02_mild_is_weak_vi}, we conclude that
    \begin{equation} \label{Paper02_mild_is_weak_xii}
    \begin{aligned}
        \lim_{\varepsilon \downarrow 0} \int_\varepsilon^T \int_\mathbb R \varphi(t,x) (P_{\varepsilon} F_k(u(t,\cdot)))(x) \, dx \, dt & = \lim_{\varepsilon \downarrow 0} \int_\varepsilon^T \int_\mathbb R (P_{\varepsilon}\varphi(t,\cdot))(x) F_k(u(t,x)) \, dx \, dt \\ & = \int_0^T \int_\mathbb R \varphi(t,x) F_k(u(t,x)) \, dx \, dt.
    \end{aligned}
    \end{equation}
    Recall from after~\eqref{Paper02_mild_is_weak_iii} that $R_\varphi \in (0,\infty)$ is such that $\supp \varphi \subset [0, \infty) \times [-R_\varphi, R_\varphi]$. For the second and third terms on the right-hand side of~\eqref{Paper02_mild_is_weak_ix}, observe that by~\eqref{Paper02_trivial_estimate} and since $\varphi \in \mathscr{C}^{1,2}_{c}([0, \infty) \times \mathbb R; \mathbb  R)$,
    \begin{equation} \label{Paper02_mild_is_weak_xiii}
    \begin{aligned}
        \Bigg\vert & \int_\varepsilon^T \int_\mathbb R \left(\partial_t + \frac{m}{2} \Laplace\right) \varphi(t,x) \left(\int_{t - \varepsilon}^{t} (P_{t- \tau} F_k(u(\tau,\cdot)))(x) \, d\tau\right) dx \,dt \\ & \; + \int_0^\varepsilon \int_\mathbb R \left(\partial_t + \frac{m}{2} \Laplace\right) \varphi(t,x) \left(\int_{0}^{t} (P_{t- \tau} F_k(u(\tau,\cdot)))(x) \, d\tau\right) dx \,dt \Bigg\vert \\ & \quad \quad \leq 4R_\varphi T \varepsilon \left\| \left(\partial_t + \frac{m}{2} \Laplace\right) \varphi \right\|_{L_\infty([0, \infty) \times \mathbb R; \mathbb R)} \| F(u) \|_{L_{\infty}([0, T] \times \mathbb R; \ell_1)}.
    \end{aligned}
    \end{equation}
    By applying~\eqref{Paper02_mild_is_weak_x} and~\eqref{Paper02_mild_is_weak_ix}, and then by using the fact that we can take $\varepsilon \in (0, T)$ to be arbitrarily small, and using~\eqref{Paper02_mild_is_weak_xi},~\eqref{Paper02_mild_is_weak_xii} and~\eqref{Paper02_mild_is_weak_xiii}, it follows that
    \begin{equation} \label{Paper02_mild_is_weak_xiv}
    \begin{aligned}
        & \int_0^T \int_\mathbb R \left(\partial_t + \frac{m}{2} \Laplace\right) \varphi(t,x) \left(\int_0^t (P_{t- \tau} F_k(u(\tau,\cdot)))(x) \, d\tau\right) dx \,dt \\ & \quad = \int_\mathbb R \varphi(T,x) \left(\int_0^{T} (P_{T- t} F_k(u(t,\cdot)))(x) \, dt\right) dx - \int_0^T \int_\mathbb R \varphi(t,x) F_k(u(t,x)) \, dx \, dt.
    \end{aligned}
    \end{equation}
    Finally, applying~\eqref{Paper02_mild_is_weak_viii} and~\eqref{Paper02_mild_is_weak_xiv} to~\eqref{Paper02_mild_is_weak_i}, and using condition~(iii), we conclude that~\eqref{Paper02_coordinate_wise_weak_sol} holds for any $k \in \mathbb N_0$, $\varphi \in \mathscr C_c^{1,2}([0, \infty) \times \mathbb R; \mathbb R)$ and $T > 0$, which completes the proof.
\end{proof}

\subsection{Weak selection--low mutation regime} \label{Paper02_app_weak_selec_low_mut}

In this subsection, we will establish the limit in~\eqref{Paper02_lim_Poisson_equil_ell_1}.

\begin{lemma} \label{Paper02_app_lem_weak_selec_low_mut}
    Let $s, \mu \in (0, \infty)$ and $(s^{(n)})_{n \in \mathbb{N}}, (\mu^{(n)})_{n \in \mathbb{N}} \subset (0,1)$ be such that~\eqref{Paper02_weak_selection_low_mutation_regime} holds. For $n \in \mathbb{N}$ and $k \in \mathbb N_0$, let $\alpha^{(n)}_k$ be defined as in~\eqref{Paper02_sequence_proportions_weak_selection_low_mutation}, $\hat \alpha_k$ be defined as in~\eqref{Paper02_lim_Poisson_equil}, and let $\alpha^{(n)} \defeq (\alpha^{(n)}_k)_{k \in \mathbb N_0}$ and $\hat \alpha \defeq (\hat \alpha_k)_{k \in \mathbb N_0}$. Then $(\alpha^{(n)})_{n \in \mathbb N}$ converges in $\ell_1$ as $n \rightarrow \infty$ to $\hat \alpha$.
\end{lemma}

\begin{proof}
    We first establish bounds, uniform in $n \in \mathbb N$, on $\| \alpha^{(n)} \|_{\ell_1}$ and on the tails of $\alpha^{(n)}$. We start by recalling the elementary inequalities
    \begin{equation} \label{Paper02_elementary_inequality_exp}
        1 + y \leq e^y \quad \forall \, y \in \mathbb R,
    \end{equation}
    and
    \begin{equation} \label{Paper02_elementary_inequality_exp2}
        1 - e^{-y} \geq (1 - e^{-1}) (y \wedge 1) \quad \forall \, y \in [0, \infty).
    \end{equation}
    By~\eqref{Paper02_elementary_inequality_exp} and then by~\eqref{Paper02_elementary_inequality_exp2}, we have that for every $n \in \mathbb{N}$ and $i \in \mathbb N$,
    \begin{equation} \label{Paper02_first_trivial_ineq_exp_weak_selec_low_mut}
    \begin{aligned}
        1 - (1-s^{(n)})^i \geq 1 - e^{-is^{(n)}} \geq (1 - e^{-1})(is^{(n)} \wedge 1).
    \end{aligned}
    \end{equation}
    Observe that by~\eqref{Paper02_weak_selection_low_mutation_regime}, there exists $n^* \in \mathbb{N}$ such that
    \begin{equation} \label{Paper02_first_trivial_ineq_exp_weak_selec_low_mut_iv}
    \mu^{(n)} < \frac{1}{2}e^{-1}(1 -e^{-1}) \quad \forall \, n \geq n^*.
    \end{equation}
    In particular,
    \begin{equation} \label{Paper02_first_trivial_ineq_exp_weak_selec_low_mut_v}
    1 - \mu^{(n)} > 1 -\frac{1}{2}e^{-1}(1 -e^{-1}) > e^{-1} \quad \forall \, n \geq n^*.
    \end{equation}
    For each $n \in \mathbb N$, let
    \begin{equation} \label{Paper02_first_trivial_ineq_exp_weak_selec_low_mut_extra_index}
        I^{(n)} \defeq \max \{i \in \mathbb N_0: \, is^{(n)} \leq 1\}.
    \end{equation}
    By the definition of $\alpha^{(n)}$ in~\eqref{Paper02_sequence_proportions_weak_selection_low_mutation}, and then the fact that $s^{(n)}, \mu^{(n)} \in (0,1)$,~\eqref{Paper02_first_trivial_ineq_exp_weak_selec_low_mut} and~\eqref{Paper02_first_trivial_ineq_exp_weak_selec_low_mut_v}, and finally using~\eqref{Paper02_first_trivial_ineq_exp_weak_selec_low_mut_extra_index}, we conclude that for $n \geq n^*$,
    \begin{equation} \label{Paper02_aux_weak_sel_low_mut_i}
    \begin{aligned}
        \| \alpha^{(n)} \|_{\ell_1} & = 1 + \sum_{k = 1}^{\infty} \left(\frac{\mu^{(n)}}{1 - \mu^{(n)}}\right)^{k} \prod_{i = 1}^{k} \frac{(1 - s^{(n)})^{i-1}}{1 - (1 - s^{(n)})^{i}}
        \\ 
        & \leq 1 + \sum_{k = 1}^\infty (\mu^{(n)}e)^k \prod^k_{i = 1} \frac{1}{(1 - e^{-1})(is^{(n)} \wedge 1)}
        \\
        & \leq 1 + \sum_{k = 1}^{I^{(n)}} \frac{1}{k!} \Bigg(\frac{\mu^{(n)}e}{s^{(n)} ( 1 - e^{-1})}\Bigg)^k + \frac{1}{I^{(n)}!} \Bigg(\frac{\mu^{(n)}e}{s^{(n)} ( 1 - e^{-1})}\Bigg)^{I^{(n)}} \sum_{k = I^{(n)}+1}^{\infty} \left(\frac{\mu^{(n)}e}{1 - e^{-1}}\right)^{k - I^{(n)}}.
    \end{aligned}
    \end{equation}
    Using~\eqref{Paper02_first_trivial_ineq_exp_weak_selec_low_mut_iv}, and applying~\eqref{Paper02_weak_selection_low_mutation_regime} to~\eqref{Paper02_aux_weak_sel_low_mut_i}, it follows that
    \begin{equation} \label{Paper02_aux_weak_sel_low_mut_ii}
        \sup_{n \in \mathbb N, \, n\ge n^*} \| \alpha^{(n)} \|_{\ell_1} < \infty. 
    \end{equation}
    By an argument similar to the one we used to establish~\eqref{Paper02_aux_weak_sel_low_mut_i}, and using~\eqref{Paper02_weak_selection_low_mutation_regime} and~\eqref{Paper02_first_trivial_ineq_exp_weak_selec_low_mut_extra_index}, there exists $C > 0$ such that for all $j \in \mathbb N$ and $n \geq n^*$,
    \begin{equation*}
    \begin{aligned}
    \sum_{k = j}^\infty \alpha^{(n)}_{k} 
    & \leq \sum_{k = j}^\infty \frac{1}{k!} \Bigg(\frac{\mu^{(n)}e}{s^{(n)}( 1 - e^{-1})}\Bigg)^k + \frac{1}{I^{(n)}!}\Bigg(\frac{\mu^{(n)}e}{s^{(n)}( 1 - e^{-1})}\Bigg)^{I^{(n)}} \sum_{k = j \vee (I^{(n)}+1)}^\infty \left(\frac{\mu^{(n)}e}{1 - e^{-1}}\right)^{k-I^{(n)}}
    \\ & \leq \sum_{k = j}^{\infty} \frac{1}{k!} C^k\frac{e^k}{( 1 - e^{-1})^k} + \frac{1}{I^{(n)}!}\Bigg(C\frac{e}{( 1 - e^{-1})}\Bigg)^{I^{(n)}} \sum_{k = j \vee (I^{(n)}+1)}^\infty \left(\frac{\mu^{(n)}e}{1 - e^{-1}}\right)^{k-I^{(n)}}.
    \end{aligned}
    \end{equation*}
    Since~$I^{(n)} \rightarrow \infty$ as $n \rightarrow \infty$ by~\eqref{Paper02_weak_selection_low_mutation_regime}, we conclude from~\eqref{Paper02_first_trivial_ineq_exp_weak_selec_low_mut_iv} that for any $\varepsilon > 0$, there exists $j(\varepsilon) \in \mathbb N$ such that for all $j \in \mathbb N$ such that $j \geq j(\varepsilon)$,
    \begin{equation} \label{Paper02_aux_weak_sel_low_mut_iii}
        \sup_{n \in \mathbb N, \, n\ge n^*} \; \sum_{k = j}^\infty \alpha^{(n)}_{k} \leq \varepsilon.
    \end{equation}
    By estimates~\eqref{Paper02_aux_weak_sel_low_mut_ii} and~\eqref{Paper02_aux_weak_sel_low_mut_iii}, it remains to establish that $\alpha^{(n)}_{k} \rightarrow \hat \alpha_k$ as $n \rightarrow \infty$ for every $k \in \mathbb N_0$. The limit is trivial for $k = 0$. For $k \in \mathbb N$, observe that~\eqref{Paper02_weak_selection_low_mutation_regime} implies that for every $i \in \mathbb N$, there exists $C_i \in (0,\infty)$ such that for all $n \in \mathbb N$,
    \begin{equation*}
        is^{(n)} - C_i n^{-2} \leq 1 - (1 - s^{(n)})^i \leq is^{(n)} + C_i n^{-2}.
    \end{equation*}
    Therefore, for all $k \in \mathbb N$,~\eqref{Paper02_weak_selection_low_mutation_regime} and~\eqref{Paper02_sequence_proportions_weak_selection_low_mutation} imply that 
    \begin{equation*}
        \lim_{n \rightarrow \infty} \alpha^{(n)}_{k} = \lim_{n \rightarrow \infty} \frac{1}{k!}\left(\frac{\mu^{(n)}}{s^{(n)}}\right)^k = \frac{1}{k!}\left(\frac{\mu}{s}\right)^k = \hat \alpha_k,
    \end{equation*}
    which completes the proof.
\end{proof}

\subsubsection*{Acknowledgements}

The authors are grateful to Matthias Birkner, Alison Etheridge, Félix Foutel-Rodier, Karsten Matthies and Matt Roberts for helpful comments and suggestions. 
While this work was being carried out, JLdOM was supported by a scholarship from the EPSRC Centre for Doctoral Training in Statistical Applied Mathematics at Bath (SAMBa), under the project EP/S022945/1.
MO is partially supported by EPSRC grant EP/X040089/1. 
SP is supported by a Royal Society University Research Fellowship. While part of this work was being carried out, SP was visiting SLMath as a Research Member of the Probability and Statistics of Discrete Structures program.

\bibliographystyle{abbrv}
\bibliography{bibli_law_large_numbers_mullers_ratchet}

\end{document}